\documentclass[12pt]{amsart}

\usepackage[margin=1in,marginparwidth=0.8in, marginparsep=0.1in]{geometry}

\pdfoutput=1

\usepackage{amsfonts,amssymb, amscd, latexsym, graphicx, psfrag, color,float}
\usepackage[normalem]{ulem}

\usepackage{dbnsymb}
\numberwithin{equation}{subsection}
\numberwithin{figure}{subsection}

\usepackage[bookmarks=true, bookmarksopen=true,%
bookmarksdepth=3,bookmarksopenlevel=2,%
colorlinks=true,%
linkcolor=blue,%
citecolor=blue,%
filecolor=blue,%
menucolor=blue,%
urlcolor=blue]{hyperref}

\usepackage[all, knot, poly]{xy}

\xyoption{knot}

\newtheorem{dummy}{dummy}[section]
\newtheorem{lemma}[dummy]{Lemma}
\newtheorem{theorem}[dummy]{Theorem}

\newtheorem{corollary}[dummy]{Corollary}
\newtheorem{proposition}[dummy]{Proposition}
\theoremstyle{definition}
\newtheorem{definition}[dummy]{Definition}
\newtheorem{example}[dummy]{Example}
\newtheorem{remark}[dummy]{Remark}

\newcommand{\bA}{\mathbb{A}}
\newcommand{\bC}{\mathbb{C}}

\newcommand{\bF}{\mathbb{F}}

\newcommand{\bP}{\mathbb{P}}
\newcommand{\bQ}{\mathbb{Q}}
\newcommand{\bR}{\mathbb{R}}
\newcommand{\bZ}{\mathbb{Z}}

\newcommand{\cA}{\mathcal{A}}
\newcommand{\cB}{\mathcal{B}}
\newcommand{\cC}{\mathcal{C}}

\newcommand{\cE}{\mathcal{E}}
\newcommand{\cF}{\mathcal{F}}
\newcommand{\cG}{\mathcal{G}}
\newcommand{\cH}{\mathcal{H}}

\newcommand{\cL}{\mathcal{L}}
\newcommand{\cM}{\mathcal{M}}

\newcommand{\cP}{\mathcal{P}}

\newcommand{\cR}{\mathcal{R}}
\newcommand{\cS}{\mathcal{S}}
\newcommand{\cT}{\mathcal{T}}

\newcommand{\op}{\operatorname}

\newcommand{\Ext}{\mathrm{Ext}}

\newcommand{\Hom}{\mathrm{Hom}}

\renewcommand{\log}{{\op{log}}}

\newcommand{\RR}{\mathbb{R}}

\renewcommand{\SS}{\mathit{SS}}

\newcommand{\redact}[1]{}

\newcommand{\ZZ}{\mathbf{Z}}

\newcommand{\mM}{\mathcal{M}}

\renewcommand{\subset}{\subseteq}

\newcommand{\twFuk}{\mathbf{Fuk}^\bullet}
\newcommand{\dgfun}{\mathbf{Fun}^\bullet}
\newcommand{\dgsh}{\mathbf{Sh}^\bullet}
\newcommand{\Sh}{\mathrm{Sh}}
\newcommand{\Fun}{\mathrm{Fun}}

\newcommand{\Mo}{\mathrm{Mo}}
\newcommand{\PGL}{\mathrm{PGL}}

\newcommand{\mmon}{\mu\mathrm{mon}}

\newcommand{\GL}{\mathrm{GL}}
\newcommand{\Br}{\mathit{Br}}
\newcommand{\Gr}{\mathrm{Gr}}

\newcommand{\cross}[9]{\xymatrix{
{}\ar@{.}[ddrr]&#4&\ar@{.}[ddll]\\
#2\ar[ur]^{#7}&{}&#3\ar[ul]_{#8}\\
{}&#1\ar[ul]^{#5}\ar[ur]_{#6}\ar@{~>}[uu]_{#9}&{}\\
}}

\usepackage{mathrsfs}
\newcommand{\barbell}{\mbox{$\bullet \!\!\! -\!\!\!-\!\!\!\bullet$}}
\newcommand{\barbarbell}{\mbox{$\bullet \!\!\! -\!\!\!-\!\!\!\bullet \!\!\! -\!\!\!-\!\!\!\bullet$}}
\newcommand{\boing}{\mbox{$ \bigcirc \!\! \bullet$}}
\newcommand{\rainbow}{\mbox{${}_\prec \!\!\!\frown \!\!\! {}_\succ$}}

\let\oldtocsection=\tocsection

\let\oldtocsubsection=\tocsubsection

\let\oldtocsubsubsection=\tocsubsubsection

\renewcommand{\tocsection}[2]{\hspace{0em}\oldtocsection{#1}{#2}}
\renewcommand{\tocsubsection}[2]{\hspace{3em}\oldtocsubsection{#1}{#2}}
\renewcommand{\tocsubsubsection}[2]{\hspace{6em}\oldtocsubsubsection{#1}{#2}}

\title[Legendrian Knots and Constructible Sheaves]{Legendrian Knots and Constructible Sheaves}
\author{Vivek Shende, David Treumann, and Eric Zaslow}

\begin{document}

\begin{abstract}
We study the unwrapped Fukaya category of Lagrangian branes ending on a
Legendrian knot.
  Our knots live at contact infinity in the cotangent bundle of a surface,
the Fukaya category of which is equivalent
  to the category of constructible sheaves on the surface itself.
Consequently, our category can be
   described as constructible sheaves with singular support controlled by
the front projection of the knot.  We use a theorem of
Guillermou-Kashiwara-Schapira to show that the resulting
 category is invariant under Legendrian isotopies. 
A subsequent article establishes its equivalence
to a category of representations of 
the Chekanov-Eliashberg differential graded algebra.

 We also find two connections to topological knot theory.
 First, drawing a positive braid closure on the annulus,
  the moduli space of rank-$n$ objects maps to the space of local systems
on a circle.
The second page of the spectral sequence
 associated to the weight filtration on the pushforward of the constant
sheaf is the (colored-by-$n$) triply-graded Khovanov-Rozansky homology.
 Second, drawing a positive braid closure in the plane, the number of
points of our moduli spaces
 over a finite field with $q$ elements recovers the lowest coefficient in
`$a$' of the HOMFLY polynomial of
 the braid closure.   
\end{abstract}

\maketitle

\tableofcontents

\section{Introduction}

Isotopy invariants of knots have a long history.  Of late, they have taken the form of multiply-graded chain complexes \cite{Kh,KhR}, with isotopies inducing quasi-isomorphisms.  Witten has provided a physical context for such invariants in
gauge theory, whether they be numerical \cite{WJ}
or ``categorified'' \cite{WK}.  
In the context of Legendrian knots, further invariants can
distinguish Legendrian isotopy classes within the same topological class.
The classical Legendrian invariants are the rotation number
and the Thurston-Bennequin number, but the most powerful known invariant
is the Chekanov-Eliashberg differential graded algebra,
distinguishing as it does pairs of knots with the same classical invariants.
Though it has a combinatorial description, the
Chekanov-Eliashberg invariant \cite{C,E} and its higher-dimensional generalization Legendrian contact homology
\cite{EES,EES2}
have Morse-Floer-Fukaya-theoretic descriptions in terms of
holomorphic disks.  This suggests a route from knot invariants
to physics through open strings.  One such route was taken in \cite{AENV},
where the authors begin with a topological knot in three-space and compute
the Legendrian contact homology of the
associated Legendrian torus, i.e. the unit conormal in the cosphere bundle of
three-space.  In this paper we explore another connection:  by embedding
the standard contact three-space into the cosphere bundle of the plane,
we interpret a Legendrian knot or link as living at infinity in the cotangent
bundle of the plane.  From this geometric set-up, we define a category:
the Fukaya category of the cotangent of the plane whose geometric objects
are Lagrangians asymptotic to the knot at infinity.

Microlocalization --- the relationship between the symplectic topology of
a cotangent and the topology of the base manifold --- affords us another
perspective on this category.  In its categorical form, microlocalization means
that the Fukaya category of the cotangent of the plane is equivalent to the
category of constructible sheaves on the plane \cite{N,NZ}.  Tamarkin \cite{T} has also
developed a sheaf-theoretic treatment of symplectic problems in the cotangent bundle.
Under the microlocalization theorem of \cite{N} (see Remark 3.24 of \cite{N2}), our category is in fact equivalent to sheaves
in the plane constructible with respect to the stratification defined by the front diagram
of the knot, and satisfying some microlocal conditions. 
In this form, Hamiltonian invariance follows from the
general work of Guillermou-Kashiwara-Schapira
\cite{GKS}. 
As a category of sheaves constructible with respect to a fixed stratification,
the category also has a combinatorial, quiver-type description where invariance
under Reidemeister moves can be seen explicitly.  One final perspective on this
category is used to facilitate calculations, and this exploits the relationship between
constructible sheaves and algebraic geometry:  after isotoping the front diagram
to a rectilinear `grid diagram,' the coherent-constructible
correspondence of \cite{FLTZ} gives a description of the category in terms of
modules over the polynomial ring in two variables.  This means calculations in
the category can easily be programmed into a computer.

A Fukaya category has a distinguished
class of \emph{geometric objects}:
smooth Lagrangians with flat line bundles, or local systems.  
The Lagrangians in our category bound the knot at infinity
in the cotangent of the plane, and rank-one local systems restrict
at infinity to rank-one local systems on the knot.
The philosophy of symplectic field theory tells us that such Legendrian fillings, being cobordisms with the empty set, should furnish
one-dimensional representations, or `augmentations', of the Chekanov-Eliashberg 
differential graded algebra --- and they do \cite{Ek}.
More generally, we establish in \cite{NRSSZ} the existence of a category of augmentations
equivalent to the full subcategory of rank-one
objects of our category; a version of this statement was originally conjectured
in a previous draft of this paper.

From the perspective of mirror symmetry,  consideration of
the moduli space of rank-one objects in our category --- which we view to be the analogue in this context of special Lagrangian branes --- 
is almost obligatory: it is a kind of mirror to the knot.
Due to the fact that a geometric deformation of a special Lagrangian is related
by the complex structure on moduli space to a deformation of a unitary local system,
we think of an open set in this moduli space as being the character variety of
a generic filling surface of the Legendrian knot, i.e. one of maximal genus.

Whatever the heuristics, the augmentation conjecture above invites comparison with the augmentation varieties constructed by
Henry and Rutherford \cite{HR}.  
Central to the study of Legendrian knots by their front diagrams is the notion of
a ruling, a decomposition of the diagram into disks bounded on the left and
right by cusps.
In the sheaf-theoretic description of rank-one objects
of our category we define the notion of a ruling filtration,
i.e.~a filtration whose associated
graded sheaves are supported on the disks of a ruling.
Much more can be said in the case that the knot is the
``rainbow"
 closure of a positive braid:
the different rulings arising from ruling filtrations provide a stratification of the
moduli space into pieces that have the same structural form as those
found by Henry and Rutherford for the augmentation variety \cite{HR} of the knot.
It follows that the number of points on our moduli spaces
over the finite field with $q$ elements 
is governed by the ruling polynomial of the braid closure.  Another theorem of
Rutherford \cite{Ru} identifies this expression with a {\em topological} knot invariant --- 
the polynomial in `$q$' which is the lowest order
term in `$a$' of the HOMFLY polynomial.

This term of the HOMFLY polynomial has appeared in recent work
in the algebraic geometry of singular plane curves.  Specifically,  
the Poincar\'e polynomial of the perverse Leray filtration on the cohomology
of the compactified Jacobian of a singular plane curve is equal 
to this term in the HOMFLY polynomial of its link
 \cite{M, MY, MiSh}.  These links are all closures of positive braids,
 so fit into our present story, and the result on the number of points on our moduli
 spaces in this case can be restated as asserting that the Poincar\'e polynomial of 
 the weight filtration on their cohomology is equal to this term of the HOMFLY polynomial. 
 This identification between the perverse filtration on the cohomology of one space
 and the weight filtration on the cohomology of another has appeared before: 
 precisely the same relation is conjectured to exist between the Hitchin system 
 and character variety of a smooth curve \cite{dCHM}.  
 
This is no accident.  As originally conjectured here, and subsequently explained
in \cite{STWZ}, Betti moduli spaces of irregular connections on curves can be identified
with moduli of constructible sheaves with singular support in certain Legendrian links.  
The irregular nonabelian Hodge theory results of \cite{BiBo, Fre} serve to 
identify the Betti moduli of connections on $\mathbb{P}^1$ with a single irregular 
singularity with a moduli space of Higgs bundles.  In the case of torus knots, 
this moduli space retracts to its central fibre, which can be identified with the compactified
Jacobian of a curve of the form $x^a = y^b$.  A numerical version of the ``P = W'' 
conjecture of \cite{dCHM} for these spaces follows from the calculations of 
\cite{OS, ORS}, the above remarks, and the calculations here.  Details will appear elsewhere
\cite{Sh2}.
It may be expected 
that the constructions of \cite{OY} may have counterparts on our side: 
the spherical rational Cherednik algebra for $\GL_n$ acts on the cohomology of 
the moduli space of Lagrangian branes ending on a torus knot; perhaps
in the present context, the operators will arise from symplectic geometry, e.g. by 
considerations as in \cite{W}.

Motivated by the appearance of these wild character varieties, we 
speculated in a previous version of this article that perhaps
the intriguing structures of \cite{GMN} may be found in our
moduli spaces more generally.  This remains largely conjectural, but 
the connection between general cluster varieties and the moduli spaces
here has since been clarified in \cite{STWZ}; a key role there was played
by constructions similar to that of Proposition \ref{prop:rf}.

The connection to \cite{OS, ORS} suggests one connection between our work and 
the Khovanov-Rozansky  \cite{Kh,KhR} triply-graded knot homology.  We have
found another, more complicated, but rigorous and more general, 
relation to the Khovanov-Rozansky homology.  If we close up a positive braid by wrapping 
  it around a cylinder,
  then the moduli spaces in this case are constructed from open Bott-Samelson-type spaces over the flag
  variety. 
  The same spaces arise in Webster-Williamson's
  \cite{WW-HOMFLY} geometric construction of the Khovanov-Rozansky invariants, which
  means the category we construct gives a route to these categorified knot
  invariants.
  In some more detail, the moduli spaces of sheaves have a geometrically induced map to the adjoint
  quotient of the general linear group, obtained by restricting the sheaf
  to the top of the cylinder, where it is a local system on a circle.
  The map induces a weight filtration on cohomology, and the Khovanov-Rozansky 
  invariants are the second page of the associated spectral sequence.

\addtocontents{toc}{\protect\setcounter{tocdepth}{1}}

\subsection*{Results}

We continue now with precise statements of the key results of this paper, and some 
results of its sequels \cite{NRSSZ, Sh, NRSS, STWZ}, which were originally conjectured here.

Let $M$ be a real analytic manifold
and let $k$ be a field.  The cosphere bundle
$T^{\infty}M$ is naturally a contact manifold; let $\Lambda \subset T^{\infty}M$ be a Legendrian submanifold. 
Let $\dgsh_{\Lambda}(M,k)$ denote the dg category of constructible sheaves of $k$-modules
whose singular support intersects $T^{\infty} M$ in $\Lambda$. 
Using general results of
Guillermou, Kashiwara, and Schapira \cite{GKS}, we show in Section \ref{sec:invariance} (see Theorem \ref{thm:4.1}):

\begin{theorem}\label{thm:introinvariance} A contactomorphism inducing
a Legendrian
isotopy induces a quasi-equivalence
\begin{equation}
\label{eq:luciusjuniusbrutus}
\dgsh_{\Lambda}(M,k) \xrightarrow{\sim} \dgsh_{\Lambda'}(M,k).
\end{equation}
\end{theorem}

Henceforth we take $M = \bR^2$ or $S^1 \times \bR$.  
It is convenient to restrict to those objects which have acyclic stalks for $z \ll 0$; we denote
the full subcategory of such objects by $\dgsh_\Lambda(M, k)_0$ --- the equivalence \eqref{eq:luciusjuniusbrutus} preserves this subcategory.

In Section \ref{sec:mishmash}, we show that 
fixing a Maslov potential on the front diagram of $\Lambda$ determines a functor
$$\mmon:  \dgsh_\Lambda(M, k) \rightarrow Loc(\Lambda)$$
to local systems (of complexes of $k$-modules up to quasi-isomorphism) on the
knot $\Lambda$. 
Define the ``subcategory of rank-$r$ objects'' $\cC_r(\Lambda),$  with moduli space $\cM_r(\Lambda),$ by setting
$$\cM_r(\Lambda) = \{ F\in \dgsh_\Lambda(M, k)_0 \mid \mmon(F) \text{ is a rank-$r$ local system
in degree zero} \}.$$

The construction of $\mmon$ reveals (compare Proposition \ref{prop:mumonshift}):

\begin{proposition} \label{prop:norot}
If $\Lambda$ has rotation number $r$, then every element of $\dgsh_\Lambda(M,k)$ is periodic
with period $2r$; in particular, if $r \ne 0$, then there are no bounded complexes of sheaves in 
$\dgsh_\Lambda(M, k)$. 
\end{proposition}

and we also show (Proposition \ref{prop:zigzag})

\begin{proposition} \label{prop:nostab}
If $\Lambda$ is a stabilized legendrian knot, then every element of $\dgsh_\Lambda(M,k)$ is locally
constant; in particular $\dgsh_\Lambda(M,k)_0 = 0$.  
\end{proposition}

Some Maslov potentials relieve us of the need to work with homological algebra, dg-categories, etc.
Proposition \ref{prop:bimaslov} gives:

\begin{proposition}
When the front diagram of $\Lambda$ carries a Maslov potential taking only the values $0$ and $1$,  then
every element of $\cC_r(\Lambda)$ is quasi-isomorphic to it's zeroeth cohomology sheaf.  Moreover, the moduli spaces
$\mM_r(\Lambda)$ are algebraic stacks.  
\end{proposition}

The simplest case is that of a positive braid carrying the zero Maslov potential, which we study in Section \ref{sec:braids}.  
Write $s_1, \ldots, s_{n-1}$ for the standard generators of the braid group $Br_n$. 
Fixing a rank $r$, we write $G := \GL_{nr}$ and $P$ for the parabolic subgroup with
block upper triangular matrices which have $n$ blocks each of size $(r\times r)$ along 
the diagonal.  By the
open Bott-Samelson variety $BS(s_{i_1} \cdots s_{i_k}) \subset (G/P)^{k+1}$, we mean
$(k+1)$-tuples of flags such that the pair $F_t, F_{t+1}$ 
are in the Schubert cell labeled by the transposition corresponding to $s_{i_t}$. 
It is well known that
$BS(s_{i_1} \cdots s_{i_k})$ only depends on the braid $s_{i_1} \cdots s_{i_k}$ and
not on its expression; this also follows from Theorem \ref{thm:introinvariance}, which yields
invariance under the braid relation $s_{j}s_{j+1}s_j = s_{j+1}s_j s_{j+1}$ through Reidemeister-3, and from
Proposition \ref{prop:modbs} and Remark \ref{rmk:modbs}, which give:
\begin{proposition}
Let $\beta$ be a positive braid.  Then $\mM_r(\beta) = G \backslash BS(\beta)$. 
\end{proposition}

In fact, the open Bott-Samelson variety does not depend on much at all; it is an iterated affine space bundle. 
However, the natural maps between such spaces encode a great deal of information.  The maps arise geometrically: 
in particular, glueing $\RR^2$ into a cylinder and identifying the endpoints of the braid recovers Lusztig's 
horocycle correspondence \cite[\S 2.5]{Lu}.  In this cylindrical setting, 
the composition $\Lambda \hookrightarrow \bR \times S^1 \to S^1$ is \'etale of degree equal to the braid index $n$, 
and pushing forward the local system 
$\mmon$ gives a map  $$\pi: \mM_r(\Lambda) \to \mathrm{Loc}_{\mathrm{GL}_{rn}}(S^1) = \mathrm{GL}_{rn} / \mathrm{GL}_{rn}$$ 

These ingredients can be related in a straightforward manner to those used by Williamson and Webster in their geometric construction
\cite{WW-HOMFLY}
of the Khovanov-Rozansky categorification \cite{KhR} of the HOMFLY polynomial, and we show (see Theorem \ref{thm:HHH}):

\begin{theorem} \label{thm:kr}
Let $\Lambda$ be the cylindrical closure of an $n$-stranded positive braid.  
Then the  HOMFLY homology of $\Lambda$ colored by $r$ is 
the $E_2$ page in the hypercohomology spectral sequence associated to the weight filtration
of 
$\pi_* \bQ_{\mM_r(\Lambda)}$.
\end{theorem}

We can also close a positive braid in the plane, by joining the ends
in a ``rainbow'' pattern (nonintersecting
and above the braid itself).  In this case we find the moduli spaces
carry natural decompositions by graded rulings. Proposition \ref{prop:rulingmoduli} gives:

\begin{theorem}
\label{thm:augthm}
Let $\Lambda$ be the rainbow closure of a positive braid with $w$ crossings,
carrying a Maslov potential taking the value zero on all strands of the braid.
Let $\cR$ be the set of graded, normal rulings of the front diagram.  
$$\cM_1(\Lambda) \cong \bigsqcup_{r \in \cR} U(r) $$
If $|r|$ is the number of switches of the ruling $r$, then
$U(r)$ is an iterated bundle with fibres
$(\mathbb{G}_a^{x_i} \times \mathbb{G}_m^{y_i}) / \mathbb{G}_m$, where 
$\sum x_i = (w - |r|) / 2$ and $\sum y_i = |r|$. 
\end{theorem}

The theorem is proven by studying ``ruling filtrations'' of objects
$\cF \in \cC_1(\Lambda)$. 
These are filtrations whose associated graded pieces are constant sheaves supported
on the eyes of a ruling of the front diagram.  
Ruling filtrations are interesting in their own right, and can for instance be used to construct
an extension of the microlocal monodromy to a local system on an abstract (not equipped with a map to $T^*\bR^2$) surface
bounding the knot.  We leave open the very natural question
of how, in general, the `normality' condition on rulings interacts with ruling filtrations.  Resolving this
question for a rainbow closure of a positive braid is a crucial step in the above result.

The category $\dgsh_\Lambda(M,k)_0$ was created as a constructible-sheaf analogue of
the Fukaya category near infinity.  As such it bears a kinship with Legendrian contact
homology, which in the dimension at hand is the Chekanov-Eliashberg differential
graded algebra (``C-E dga'') of a Legendrian knot, or its linearization through augmentations.
In fact, we construct  a unital $A_\infty$ category of
augmentations of the C-E dga in \cite[\S 4]{NRSSZ}, and establish there the following: 
\begin{theorem}[{\cite[Theorem 7.1]{NRSSZ}}]
\label{conj:mainconj}
There exists an equivalence of $A_\infty$ categories
$$\cC_1(K) \cong \cA ug(K)$$
\end{theorem}

The proof of Theorem \ref{conj:mainconj} given in \cite{NRSSZ} is combinatorial: it proceeds by 
comparing the combinatorial presentation of  $\dgsh_\Lambda(M,k)_0$
established in Section \ref{sec:lkacs} --- see especially Proposition \ref{prop:braidlegible} --- to 
a combinatorial presentation of the augmentation category resting ultimately 
on works of Chekanov \cite{C}, Ng \cite{L}, and Sivek \cite{Siv}.  The same combinatorics
led to a proof of the equivalence between the Morse-theoretic
``generating family homology'' and linearized Legendrian contact homology \cite[Cor. 7]{Sh},
previously conjectured in \cite{Tra, FR}.  

A version of Theorem \ref{conj:mainconj} appeared as a conjecture in an earlier draft of this paper.
Propositions \ref{prop:norot} and \ref{prop:nostab} served as immediate sanity checks:
analogous statements are known to hold for augmentations.  
Theorem \ref{thm:augthm} gave a very compelling piece of evidence: it asserts that, 
at least for rainbow braid closures, 
$\cM_1(\Lambda)$ has the same sort of structural decomposition
as the ``augmentation variety'' of Henry-Rutherford \cite{HR}.  
In particular, comparing Theorem \ref{thm:augthm} and \cite[Th. 1.1]{HR} gives (Theorem \ref{thm:homflycount} in the text),
independently of \cite{NRSSZ}:

\begin{corollary}
Let $\beta$ be a positive braid, whose braid closure has $c$ components.  Let $V(\beta)$
denote the augmentation variety of its C-E dga, as in \cite[\S 3]{HR}.  Then up to a power of $q$
(indicated by a question mark below): 
$$\#\mM_1(\beta^\succ)(\bF_q) = \#  V_{\beta^\succ}(\bF_q) \cdot q^? (q-1)^{-c}$$
\end{corollary}

We have since established the following statement, a version of which was
originally conjectured here: 

\begin{theorem}[{\cite[Th. 1]{NRSS}}] \label{thm:numbers} 
For any Legendrian link $\Lambda \subset \mathbb{R}^3$, the following are equal. 
\begin{itemize}
\item 
The homotopy cardinality of the augmentation category of $\Lambda$ over $\bF_q$.
\item The homotopy cardinality
of $\cM_1(\Lambda)(\bF_q)$
\item The expression $q^{?}  (q-1)^{-c} \# V_\Lambda(\bF_q)$.
\end{itemize}
\end{theorem} 

\begin{remark}
By work of Henry and Rutherford \cite{HR}, the third option above was known
to equal a certain sum over normal rulings.  
The analogous statement in the case of periodic complexes remains open.  
This is the most interesting case, since by earlier work of Rutherford \cite{Ru}, 
the third term above is equal to a certain coefficient of $a$ in the HOMFLY polynomial of
the link.  We do know this result in the case of positive braid closures, since in this case all
2-periodic structures canonically lift to $\mathbb{Z}$-graded structures. 
\end{remark}

We also include in Section \ref{sec:conj} some computations and examples.   
We match the graded dimensions
of endomorphisms of objects with
those of the linearized Legendrian contact homology for
the $(2,3)$ torus knot (trefoil), the $(3,4)$ torus knot, and the
$m(8_{21})$ knot of \cite{CN}.
For the trefoil, we match the graded dimensions of all homs
with the bilinearized Legendrian contact homology computed in \cite{BC}.  Finally, 
we show, independent of the relation to the augmentation category, how our category separates the Chekanov pair
(see Section \ref{sec:chekpair}):

\begin{proposition} \label{prop:chek}
If $\Lambda, \Lambda'$ are the two Chekanov knots, then
$$\dgsh_\Lambda(\RR^2, k)_0 \not \cong \dgsh_{\Lambda'}(\RR^2, k)_0.$$
\end{proposition}

As shown in \cite{NRSSZ}, linearized contact homology is the endomorphism complex
of the augmentation category, so Proposition \ref{prop:chek} follows from
Theorem \ref{conj:mainconj} above, together with 
Chekanov's original calculation \cite{C}.  However we give here an independent proof, solely in the
language of sheaf theory.

\medskip
The point count of Theorem \ref{thm:augthm} is of particular interest when combined with 
a theorem of Rutherford relating the ruling polynomial to the HOMFLY polynomial  \cite{Ru}. 
Let $K$ be a topological knot.  We write $P(K) \in \bZ[(q^{1/2}- q^{-1/2})^{\pm 1}, a^{\pm 1}]$ for the HOMFLY polynomial of
$K$.  Our conventions are given by the following skein relation and normalization:

\begin{eqnarray*}
  \label{eq:skein1}
  a \, P(\undercrossing) -  a^{-1} \, 
  P(\overcrossing) & = & (q^{1/2} - q^{-1/2})  
  \, P(\smoothing) \\
  a - a^{-1} & = & (q^{1/2}-q^{-1/2})P(\bigcirc)
\end{eqnarray*}

In these conventions, the lowest order term in $a$ of the HOMFLY
polynomial of the closure of a positive braid with $w$ crossings and $n$ strands is $a^{w-n}$; 
note $w-n$ is also the Thurston-Bennequin
invariant of its rainbow closure.  

Combining Theorem \ref{thm:augthm} with \cite[Theorem 4.3]{Ru}, we deduce:

\begin{theorem} 
\label{thm:112}
Let $\Lambda$ be the rainbow closure of a positive braid with $w$ crossings on $n$ strands.  Then 
$$ \left.(-aq^{-1/2})^{n-w} P_{\Lambda}(a,q)\right|_{a=0} = \# \cM_1(\Lambda)(\mathbb{F}_q) = \sum_{i,j} (-1)^i q^{j/2} \mathrm{dim \,Gr}^W_j H^i_c(\cM_1(\Lambda)(\bC), \mathbb{Q} )$$
Here, $\mathrm{Gr}^W$ is the associated graded with respect to the weight filtration, and so the expression on the right is 
the weight series of the stack $\cM_1(\Lambda)$.   
\end{theorem}

We close this introduction with a brief discussion of the relationship of the present work
to wild character varieties and the Hitchin system, a subject explored in parallel \cite{STWZ} and subsequent \cite{Sh2} work.
The stack $\cM_1(\Lambda)$ is not proper, and so its Poincar\'e series cannot be expected to be symmetric.  
But, again because the space is noncompact,
the weight series differs from the Poincar\'e series.  It follows from the above result
and the $q \to -q^{-1}$ symmetry of the HOMFLY polynomial that, curiously, 
the weight series is symmetric.  The same phenomenon was observed by Hausel and Rodrigues-Villegas
in character varieties of surfaces \cite[Cor. 1.1.4]{HRV}.

The symmetry of the weight polynomials is the beginning of a series of conjectures about the cohomology of the character
varieties of surfaces, the strongest of which is currently the ``P = W'' conjecture \cite{dCHM}.  This asserts that the weight filtration (after
re-indexing)
on the cohomology of the (twisted) character variety of a surface can be identified with the perverse Leray filtration associated
to the moduli space of Higgs bundles on a Riemann surface with the same underlying topological surface.  The (unfiltered)
cohomologies of these spaces are identified by the nonabelian Hodge theorem, which gives a diffeomorphism
between the character variety and the moduli of Higgs bundles \cite{Do,Hi,Co, Si} .  

For the link of a plane curve singularity, there is a natural candidate for a ``P = W'' conjecture for $\cM_1$,
suggested by the relation between HOMFLY polynomial of such a curve singularity to moduli of sheaves
on the curve \cite{OS, MY, MiSh, DSV, DHS, M}.  If $C$ is a rational curve with a unique planar singularity, the cohomology of its compactified Jacobian carries a canonical filtration $P_{\leq i}$.  It is the perverse Leray filtration induced by a family (any family:  \cite[\S 3.8]{MY} or \cite[\S 6]{MiSh}) of relative compactified Jacobians with smooth total space.  
It is known that
\begin{equation}
\label{eq:last-one}
\left[(aq^{-1/2})^{1-\mu} P_{K}(a,q)\right|_{a=0} = 
(1-q)^{-b} \sum_{i,j} (-1)^i q^j \mathrm{Gr}^P_j H^i(\overline{J}(C), \mathbb{Q} )
\end{equation}
Indeed the right-hand side is identified with the Poincar\'e polynomial of a certain Hilbert scheme of  points on $C$ by the main result of any of \cite{MY, MiSh} (see also \cite{Re}), the left-hand side is identified with the same by \cite[Th. 1.1]{M}.  In \eqref{eq:last-one} $\mu$ is the Milnor number and $b$ is the number of branches of the singularity, i.e. the number of components of the link.  The link of a plane curve singularity always admits a positive braid presentation, and if this has $n$ strands and $w$ crossings, we have $1 - \mu = n - w$.  Combining \eqref{eq:last-one} and Theorem \ref{thm:112} gives us:

\begin{corollary}   \label{cor:eulpequalsw}
Let $C$ be a rational plane curve with a unique singularity; let $\Lambda$ be a rainbow closure of 
a positive braid such that the topological knot underlying $\Lambda$ is the link of the singularity. 
Then
$$ \sum_{i,j} (-1)^i q^{j/2} \mathrm{Gr}^W_j H^i_c(\cM_1(\Lambda)(\mathbb{C}), \mathbb{Q} ) 
= (q-1)^{-b} \sum_{i,j} (-1)^i q^j \mathrm{Gr}^P_j H^i(\overline{J}(C), \mathbb{Q} )$$
\end{corollary}

As we previously conjectured in an earlier version of this paper, and have now proven in \cite{STWZ}, 
when $\Lambda$ is a torus knot, the moduli space appearing on the 
left hand side can be interpreted as the Betti moduli space of 
connections on $\mathbb{P}^1$ with a certain irregular singularity.  
The original nonabelian Hodge theorem has been generalized to treat Higgs bundles 
with tame \cite{Si2}, split irregular \cite{BiBo}, 
and finally general irregular singularities \cite{Fre}.  As a consequence of this last result, 
the right hand side can be identified with the cohomology of the 
corresponding Dolbeault space; the above statement establishes a 
numerical form of the ``P = W" conjecture of \cite{dCHM} in this special case.

\subsection*{Acknowledgements} 
We would like to thank Philip Boalch,  Fr\'ed\'eric Bourgeois, Daniel Erman, 
Paolo Ghiggini, Tam\'as Hausel, Jacob Rasmussen, Dan Rutherford, and Steven Sivek 
for helpful conversations. 
The work of DT is supported by NSF-DMS-1206520 and a Sloan Foundation
Fellowship.  The work of EZ is supported
by NSF-DMS-1104779 and by a Simons Foundation Fellowship.

\addtocontents{toc}{\protect\setcounter{tocdepth}{2}}

\section{Legendrian knot basics}
\label{sec:lkb}

\subsection{Contact geometry}
\label{subsec:contactgeometry}

Here we review relevant notions of contact geometry --- see, e.g., \cite{Et, G} for
elaborations.

A contact structure on a $(2n-1)$-dimensional manifold $X$
is a maximally nonintegrable distribution of $(2n-2)$-planes.  A contact structure is locally the kernel
of a one-form $\alpha$, with $\alpha \wedge (d\alpha)^{n-1}$ nowhere vanishing; evidently
$f \alpha$ defines the same contact structure for any nowhere zero function $f$. 
A contact structure defined {\em globally} as $\mathrm{ker}\, \alpha$
for chosen one-form $\alpha$ is said to be ``co-oriented'' and
$\alpha$ is said to be a co-orientation. 
Given $\alpha,$ we define the Reeb vector field $R$ by the
conditions $\alpha(R)=1,$ $\iota_R d\alpha = 0.$
We will assume a chosen co-orientation $\alpha$ for all contact structures considered.

An $(n-1)$-dimensional submanifold $\Lambda \subset X$ is said to be Legendrian if its tangent bundle
is in the contact distribution, i.e., $\Lambda$ is $\alpha$-isotropic,
i.e. if $\alpha\vert_\Lambda$ vanishes.  (The nondegeneracy condition forbids
isotropic submanifolds of dimension greater than $n-1$.)  Isotopies of Legendrian submanifolds
through Legendrian submanifolds
are generated 
by (time-dependent) functions, and are also known as
Hamiltonian isotopies.  
In formulas, if a time-dependent vector field $v$ generates the isotopy,
we define the associated time-dependent
Hamiltonian function by $H = \alpha(v).$  Conversely,
given a Hamiltonian $H$, define $v$ uniquely by
$\iota_v \alpha = H$ and $\iota_v d\alpha = dH(R)\alpha
- dH.$

The cotangent bundle $T^* M$ of an $n$-dimensional manifold $M$ carries a canonical one-form; in local coordinates 
$(q_i, p_i)$, it 
is $\theta := \sum p_i dq_i$.  This gives $(T^* M, d\theta)$ the structure of an exact symplectic
manifold.  The restriction of $\theta$
equips the $(2n-1)$ dimensional cosphere bundle (at any radius)
with a (co-oriented) contact structure.   
We think of the radius of the cosphere bundle as infinite and
denote it as $T^{*,\infty}M$ or $T^\infty M$.

More generally, in an exact symplectic manifold $(Y, \omega = d \theta)$ equipped
with a ``Liouville'' vector field $v$, i.e. 
$\iota_v \omega = \theta,$ 
the symplectic primitive $\theta$ restricts to a contact
one-form on any hypersurface transverse to the Liouville vector field $v$.
In the cotangent example, $v$ is the generator of dilations.

As another example, the jet bundle $J^1(M) = T^*M \times \bR_z$
of a manifold has contact form $\alpha = \theta - dz.$

We are interested in the above
cotangent constructions when $n=2$ and  $M = \bR^2$ or $\bR \times \bR/\bZ.$
In this case, Legendrians are one-dimensional and are called ``knots,'' or sometimes ``links''
if they are not connected.  We refer to both as ``knots.''

\subsubsection{The standard contact structure on $\bR^3$}
On $\bR^3 = \bR_{x,y,z}^3$ the one-form $\alpha = dz - ydx$  defines a
contact structure with Reeb vector field $R=\partial_z$.
This space
embeds as an open contact submanifold $T^{\infty,-}\bR^2$ of ``downward'' covectors in
$T^\infty \bR_{x,z}^2 \cong \bR^2\times S^1,$ as follows.  First coordinatize $T^*\bR^2$ as $(x,z;p_{x},p_{z})$
with symplectic primitive $\theta = -p_{x} dx  -p_{z} dz.$  
Then the hypersurface $p_{z} = -1$ is transverse to the Liouville
vector field $v = p_{x} \partial_{p_x} + p_{z} \partial_{p_z}$ and
inherits the contact structure defined
by $p_{x} dx - dz$.  Identifying $y$ with $p_{x}$, i.e. mapping
$(x,y,z)\mapsto (x,z;y,-1),$ completes the construction.

Also note that the contact structure on $\bR^3$
can be identified with $J^1(\bR_x),$ the first jet space,
since $J^1(\bR_x) = T^*\bR_x \times \bR_z = \bR^2_{x,y}\times \bR_z.$  That is,
the coordinate $z$ represents the functional value, and $y$
representing the derivative of the function with respect to $x.$  Given a function $f(x),$ the
curve $(x,y = f'(x),z = f(x))$ is Legendrian.

\subsubsection{The standard contact structure on $\bR/\bZ \times \bR^2$}
The constructions of the previous section are all invariant under the $\bZ$-action
generated by $x\mapsto x+1$.  Writing $S^1_x$ for $\bR_x/\bZ$,
this observation says that the standard contact structure $ydx - dz$
on $ S_x^1 \times \bR^2_{y,z}$ embeds as the contact submanifold
$T^{\infty,-} (S^1_x \times \bR_z)$ of the thickened torus $T^\infty(S^1_x \times \bR_z) \cong \bR \times S^1 \times S^1.$

The relationship to the jet space is respected by the quotient, so
$S^1_x \times \bR^2_{y,z}$ is contactomorphic to $J^1(S^1_x).$

\subsection{The Front Projection}
\label{sec:frontprojection}

We call the map $\pi: T^\infty M \to M$ the ``front projection.''  Under the above identification
$\bR^3_{x,y,z} \subset T^\infty \bR^2_{x,z}$, the restriction of this 
map is just projection to the first and third coordinates
$\pi: \bR^3_{x,y,z} \to \bR^2_{x,z}$.  Similarly we have a front projection
$S^1_x \times \bR^2_{y,z} \subset T^\infty (S^1_x \times \bR_z) \xrightarrow{\pi} 
S^1_x \times \bR_z$.  
We term $\bR^2_{x,z}$ the \emph{front plane} and $S^1_x \times \bR_z$ the
{\em front cylinder}.
For specificity, we primarily discuss front projections in the front plane, noting
differences for the front cylinder when appropriate.

Given a Legendrian $\Lambda \subset \bR^3_{x,y,z}$, we call $\pi(\Lambda)$ the
front projection of the knot, and at least at immersed points of $\pi(\Lambda)$
we may recover $\Lambda$ from $\pi(\Lambda)$: since the contact form $ydx - dz$ vanishes, 
we have $y = dz/dx$.  In particular, since $y = dz/dx$ is finite, $\pi(\Lambda)$ can have no vertical tangencies.  
We say $\Lambda$ is in general position if:
\begin{itemize}
\item  $\pi|_\Lambda$ is locally injective
\item there are only
finitely many points of $\bR^2$ at which $\pi(\Lambda)$ is not an embedded submanifold of 
$\bR^2$; these are either: 
\begin{itemize}
\item  {\em cusps} where $\pi|_\Lambda$ is injective and
$dz/dx$ has a well defined limit of $0$ 
\item {\em crossings} where $\pi(\Lambda)$ is locally a transverse intersection of two smooth curves
\end{itemize} 
\end{itemize}
Any $\Lambda$ may be put in general position by a Hamiltonian isotopy, and we henceforth restrict
to such $\Lambda$. 

In practice we will start not with $\Lambda \subset \bR^3$ but with a picture in the plane.  
By a \emph{front diagram} $\Phi$, we mean a 
smooth parametrized curve with possibly disconnected
domain, locally of the form 
\begin{eqnarray*} 
\phi: \bR \supset U & \to & V \subset \bR_{x,z} \mbox{ or } S_x^1 \times \bR_z \\
t & \mapsto & (x(t),z(t))
\end{eqnarray*}
with the following properties: 
\begin{enumerate}
\item Away from a finite set, called cusps, $(x(t),z(t))$ is an immersion.  
\item When $(x(t_0),z(t_0))$ is a cusp, $\lim_{t \to t_0} \dot{z}(t)/\dot{x}(t) = 0$.  In other words, the tangent line is well-defined and horizontal.
\item There are finitely many self-intersections, they are transverse, and they are distinct from the cusps.
\end{enumerate}
Note in particular that a nonzero
covector $\lambda\cdot (\dot{z}(t),-\dot{x}(t)),$ $\lambda \neq 0,$
conormal to $(x(t),z(t)),$ can never be horizontal, i.e. can never have $\dot{x}(t) = 0$.
Therefore $(x(t),z(t))$ can be lifted in a unique way to a smooth Legendrian curve $\Lambda \in T^{\infty,-}\bR^2$, the curve of ``downward conormals'' to the front diagram.  

\begin{remark}[$C^0$ front diagrams]
\label{rem:C0}
It is sometimes useful to allow a finite number of points of $\Phi$ (away from the cusps and crossings) to fail to be differentiable.  We assume that at any such point $(x(t_0),z(t_0))$ the limits
\[
a = \lim_{t  \to t_0^{-}} (\dot{x}(t),\dot{z}(t)) \qquad b = \lim_{t \to t_0^+} (\dot{x}(t),\dot{z}(t))
\]
exist.
Then $\Phi$ lifts to a Legendrian $\Lambda(\Phi) \subset T^{\infty,-} M$ where, at any $t_0$ where $(x(t_0),z(t_0))$ is not differentiable, we add the 
angle in $T^{\infty,-}_{x(t),z(t)} M$ between $a$ and $b$.  The resulting knot or link is the same as if we had replaced $\Phi$ by a smoothed $\Phi'$ that differed from $\Phi$ only in a neighborhood of the nondifferentiable points.
\end{remark}

\vspace{2mm}
\noindent {\bf Terminology for the Front Projection}.
Let $\Phi = \pi(\Lambda)  = \phi(U)$ be a front diagram.  
$\Phi$ defines a stratification of $\bR^2$ into
zero-, one-, and two-dimensional strata.  The zero dimensional strata
are the cusps and crossings of $\Phi$.
Cusps are either 
{\em left} ($\prec$) or {\em right} ($\succ$). 
Given an orientation of $\Lambda$, or taking the orientation from the parameterization, 
cusps are also either {\em up} or {\em down}
according as the $z$ coordinate is increasing or decreasing
as you pass the cusp in the direction of orientation.  

The one-dimensional strata are the maximal smooth subspaces of $\Phi$, which we
call \emph{arcs}.  The two-dimensional strata are the connected components
of the complement of $\Phi$ which we term \emph{regions}.
Front diagrams in the  plane have one non-compact region, while
front diagrams in the front cylinder have two, an ``upper'' ($z\gg 0$) and
a ``lower'' ($z\ll 0$).

Each crossing of a front projection pairs two pairs of arcs, a northwest with a southeast and a
southwest with a northeast.  We can make
from these pairings an equivalence relation,
and the closure of the union of all arcs in an equivalence class is called a
\emph{strand}.  Informally, a strand is what you get when you start at a left
cusp and go ``straight through'' all crossings until you reach a right cusp.
In the front cylinder, strands may wrap around the circle $S^1_x$
and have no start or end.

\begin{example}
\label{ex:trefoil}
The front diagram of Figure \ref{fig:planartrefoil}
\begin{figure}[H]
\includegraphics[scale = .3]{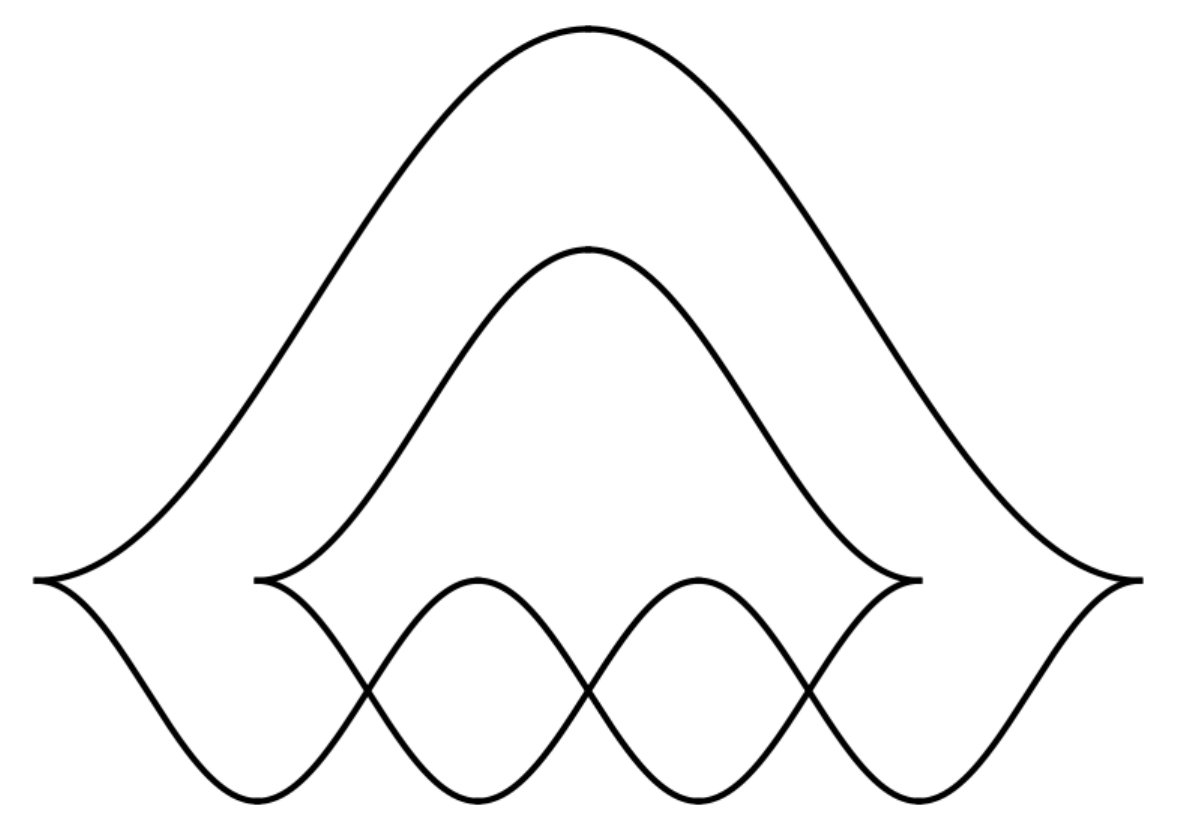}
\caption{A front plane projection of the Legendrian trefoil knot.}\label{fig:planartrefoil}
\end{figure}
\noindent in $\bR^2$ lifts to a Legendrian trefoil in $T^{\infty,-} \bR^2 \cong \bR^3$.  It has two left cusps, two right cusps, 3 crossings, 4 strands, 10 arcs, and 5 regions (4 compact regions).
\end{example}

\begin{example}
\label{ex:hopfwrap}
The cylindrical front diagram in Figure \ref{fig:cylexample}
\begin{figure}[H]
\includegraphics[scale=.3]{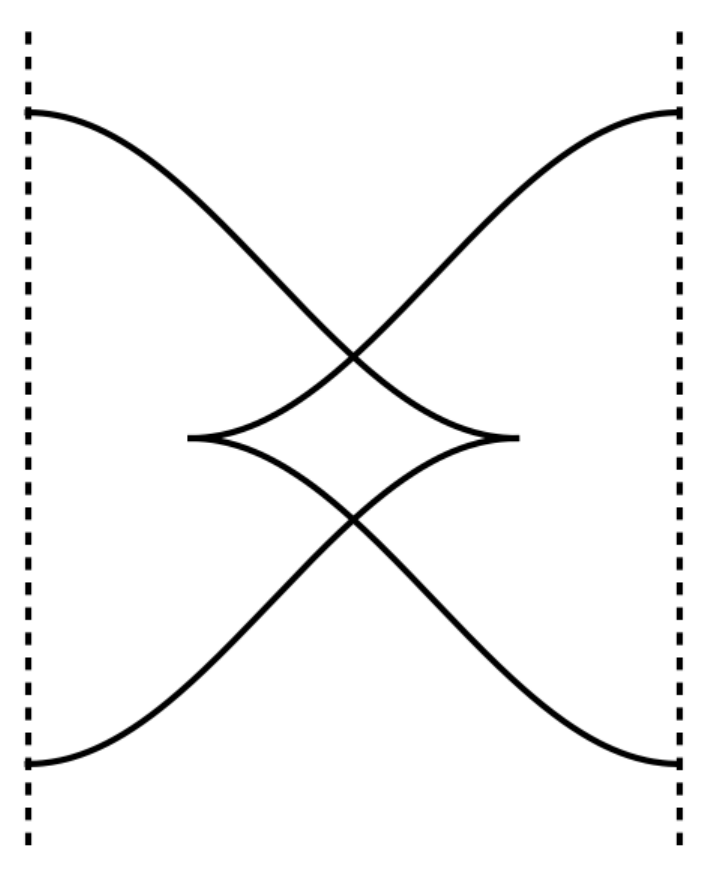}
\caption{Cylindrical front diagram of a Legendrian knot.}
\label{fig:cylexample}
\end{figure}
lifts to the knot in the solid torus which is the pattern for forming Whitehead doubles.  It has one left cusp, one right cusp, 
two crossings, two strands, 6 arcs,   
4 regions.
\end{example}

\begin{example}
\label{ex:cyltrefoil}
The front diagram of Figure \ref{fig:cyltrefoil}
\begin{figure}[H]
\includegraphics[scale=.3]{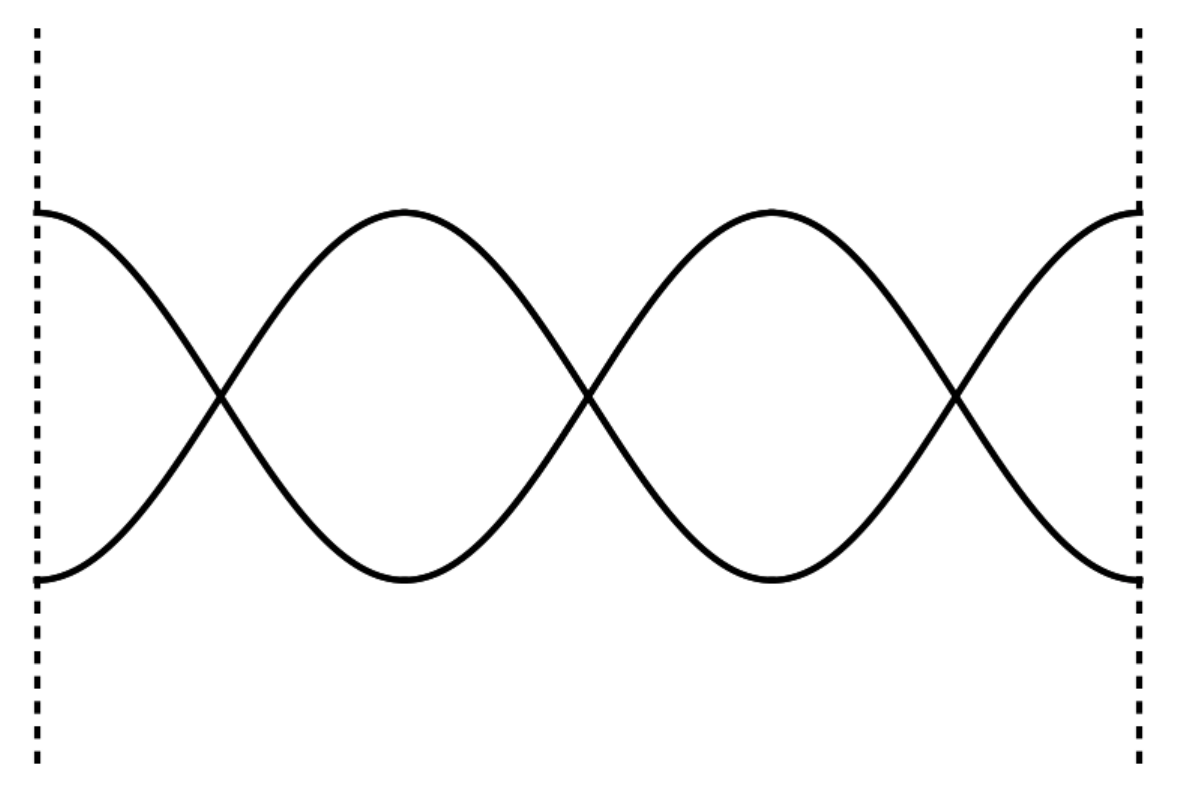}
\caption{Cylindrical front projection of a Legendrian trefoil in the solid torus.}
\label{fig:cyltrefoil}
\end{figure}
lifts to a Legendrian trefoil in the solid torus.  The dotted lines on the left and right are to be identified.  This front diagram has one strand, 5 regions,  
6 arcs, 
no cusps, 3 crossings.
\end{example}

\subsubsection{The extended Legendrian attached to a front projection}
\label{subsec:augmentedLeg}

In our study of rulings, the following construction will be useful.
Let $M$ be one of $\bR^2$ or $S^1 \times \bR$, and $\Lambda \subset T^{\infty,-} M$ a Legendrian knot in general position.  If $c \in \Phi$ is a crossing on the front projection of $\Lambda$, then $\Lambda$ meets the semicircle $T^{\infty,-}_c M$ in exactly two points.  We let $\Lambda^+$ denote the union of $\Lambda$ and the arcs of $T^{\infty,-}_c M$ joining these two points, where $c$ runs over all crossings of $\Phi$.  $\Lambda^+$ is a trivalent Legendrian graph in $T^{\infty,-} M$ --- note that it is sensitive to $\Phi$, in fact the topological type of $\Lambda^+$ is not invariant under Reidemeister moves.

\subsection{Legendrian Reidemeister moves}
\label{subsec:LegendrianReidemeisterMoves}
Two front diagrams, in either $\bR^2$ or $S^1 \times \bR$, represent the same Hamiltonian isotopy class of knots if and only if they differ by ``Legendrian Reidemeister moves'' which are pictured in Figures \ref{fig:r1}, \ref{fig:r2}, and \ref{fig:r3}.  We assume included, but have not drawn, the
reflection  across the $x$ axis of the
pictured Reidemeister 1 move, and the reflection across the $z$ axis of the pictured Reidemeister 2 move.

\begin{figure}[H]
\includegraphics[scale =.52]{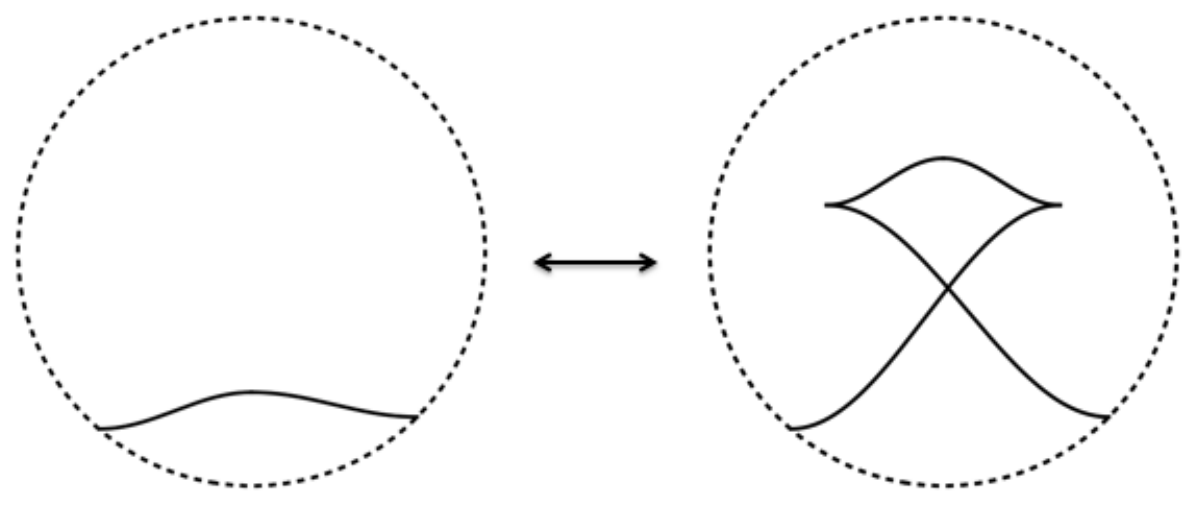}
\caption{ Reidemeister 1}
\label{fig:r1}
\end{figure}

\begin{figure}[H]
\includegraphics[scale =.52]{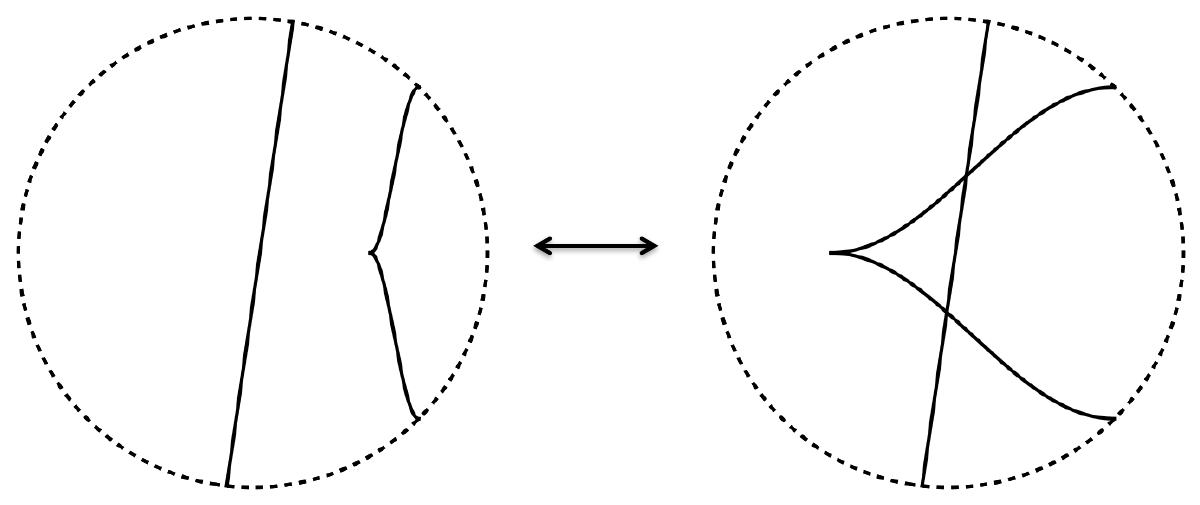}
\caption{ \label{fig:r2} Reidemeister 2}
\end{figure}

\begin{figure}[H]
\includegraphics[scale =.52]{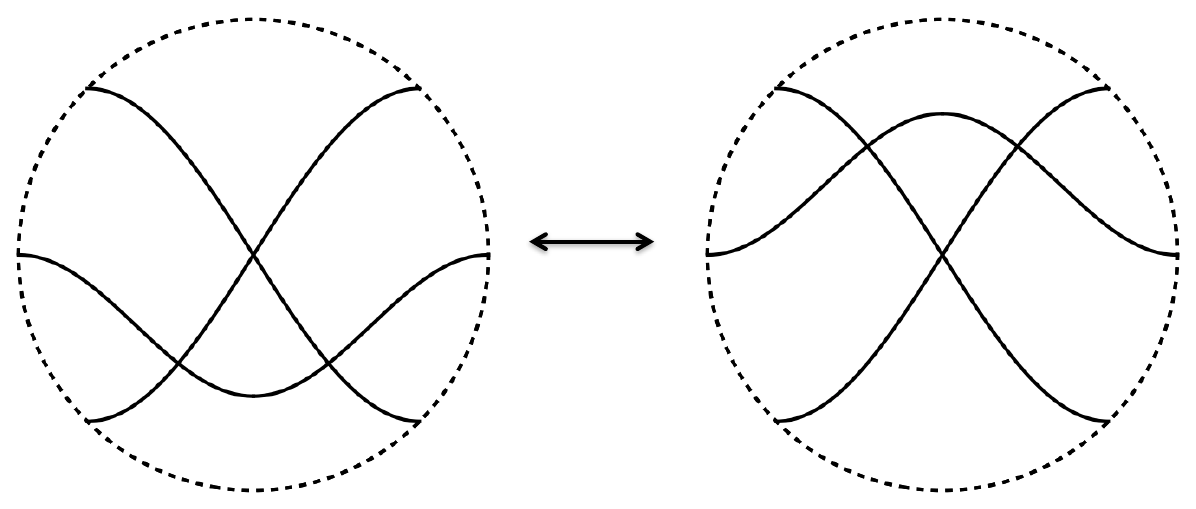}
\caption{ \label{fig:r3} Reidemeister 3}
\end{figure}

\subsection{Classical invariants} 
\label{sec:classicalinvariants}

Given a Legendrian knot $\Lambda \subset \bR^3$ (or its front diagram $\Phi \subset \bR^2$), 
the \emph{Thurston-Bennequin number} measures the twisting of the 
contact field around $\Lambda$, and can be computed as the linking number $tb(\Lambda) = lk(\Lambda,\Lambda + \epsilon z)$.
It is equal to the writhe minus the number of right cusps. 

To a connected Legendrian knot $\Lambda$ we attach a \emph{rotation number}, $r(\Lambda)$, 
which measures the obstruction to extending a vector field along $\Lambda$ to 
a section of the contact field on an embedded surface with boundary $\Lambda$.  
It is equal to half of the difference between the number of up cusps and the number
of down cusps; in particular, it depends via a sign on the orientation of $\Lambda$.

These quantities are defined in an invariant way, but also
one can see that the diagrammatic prescriptions of these quantities 
are preserved by the Legendrian Reidemeister
moves.  They are not adequate for classifying Legendrian knots up
to Legendrian isotopy:  Chekanov \cite{C} and Eliashberg \cite{E} independently
found a pair of topologically isotopic knots with the same
values of $tb$ and $r$ but which are not Legendrian isotopic.

A Maslov potential is an assignment 
$\mu: \mathrm{strands}(\Phi) \to \bZ$,
such that when two strands meet at a cusp, 
$\mu(\mbox{lower strand}) + 1 = \mu(\mbox{upper strand})$; note 
the existence of a Maslov potential is equivalent to requiring that every component
of $\Lambda$ has rotation number zero. 
More generally, 
an $n$-periodic Maslov potential is a map
$\mu:  \mathrm{strands}(\Phi) \to \bZ/n \bZ$ satisfying the same constraint;
the existence of an $n$-periodic Maslov potential is equivalent to the assertion
that, for each component of $\Lambda$,  twice the rotation number is divisible by $n$.

\section{Different views of a Legendrian invariant} \label{sec:lkacs}

We describe the category attached to a 
Legendrian knot $\Lambda \subset \bR^3 \cong T^{\infty,-}\bR^2$ in four languages. 

\begin{enumerate}
\item
As a category of sheaves.  To $\Lambda$, we associate a category $\dgsh_\Lambda(M)$ of sheaves
with singular support at infinity in $\Lambda$.  
We use this description 
for proving general theorems such as invariance
under Legendrian isotopy.  
\item As a Fukaya category. Define $\twFuk_\Lambda(T^*M)$ to be the subcategory of $\twFuk(T^*M)$ whose
Fukaya-theoretic singular support (in the sense of \cite{J,N2}) at infinity lies in $\Lambda$.
The dictionary of \cite{N,NZ} gives an 
(A-infinity quasi-)equivalence  $\dgsh(M)  \xrightarrow{\sim} \twFuk (T^* M)$,
and $\twFuk_\Lambda(T^*M)$ is the essential image of $\dgsh_\Lambda(M)$ under this map.
Informally, it is the subcategory of $\twFuk(T^*M)$ consisting of objects that ``end on $\Lambda$.''
\item As a category of representations of a quiver-with-relations associated to the dual graph of the front diagram. 
The category of sheaves on a triangulated space or more generally a  ``regular cell complex'' is equivalent
to the functor category from the poset $\cS$ of the stratification.  By translating the singular support 
condition to this functor category, we obtain a combinatorial description $\dgfun_\Lambda(\cS, k)$. 
We use this description for describing local properties, computing 
in small examples, and studying braid closures. 
\item As a category of modules over $k[x,y]$.  When the front diagram of $\Lambda$ is a ``grid'' diagram,
after introducing additional slicings, we can simplify the quiver mentioned above to one with vertices the lattice
points  $\mathbb{Z}^2 \subset \mathbb{R}^2$, and edges going from $(i,j)$ to $(i+1, j)$ and $(i, j+1)$.  
We use this description  for computer calculations of larger examples.
\end{enumerate}

We work with dg categories and triangulated categories. 
Some standard references are \cite{20}, \cite{21}, and \cite{D}.

\subsection{A category of sheaves}

\subsubsection{Conventions}
For $k$ a commutative ring, and $M$ a real analytic manifold, we 
write $\dgsh_{\mathit{naive}}(M; k)$ for the triangulated dg category whose objects are chain complexes of sheaves of $k$-modules on $M$ whose cohomology
is bounded and comprised of sheaves that are constructible
(i.e., locally constant with perfect stalks on each stratum) with respect to some
stratification ---
with the usual complex of morphisms between two complexes of sheaves.
We write 
$\dgsh(M; k)$ for the localization of this dg category with respect to acyclic complexes in the sense of \cite{D}.
We work in the analytic-geometric category of subanalytic sets, and consider only Whitney stratifications which are 
$C^p$ for a large number $p$.  Given a Whitney stratification $\mathcal{S}$ of $M$, we write 
$\dgsh_{\cS} (M; k)$ for the full subcategory of complexes whose cohomology sheaves are constructible with respect
to $\mathcal{S}$.  We suppress the coefficient $k$ and just write $\dgsh(M)$, $\dgsh_{\cS}(M)$, etc.,
when appropriate.\footnote{We do not always work with sheaves of 
$\bC$-vector spaces, but otherwise, our conventions concerning Whitney stratifications and 
constructible sheaves are the same as \cite[\S 3,4]{NZ} and \cite[\S 2]{N}.  }

\subsubsection{Review of singular support} 

To each $F \in \dgsh(M)$ is attached a closed conic subset $\SS(F) \subset T^* M$, called the singular support of $F$.  
The theory is
introduced and thoroughly developed in \cite{KS} for general, not necessarily constructible, sheaves --- see especially Chapter V and the treatment of constructible sheaves in Chapter VIII. 
Since our focus is more narrow, we have chosen to an approach through stratified Morse theory \cite{GM}, similar to the point of view of \cite{Schurmann}.  
The choice is not essential, and the more general treatment in \cite{KS} may be used.

Fix a Riemannian metric to determine $\epsilon$-balls $B_{\epsilon}(x) \subset M$
around a point $x \in M$; the following constructions are nonetheless independent of the metric. 
Let $F$ be an $\cS$-constructible sheaf on $M$.  Fix a point $x \in M$ and a smooth function $f$ in a neighborhood of $x$. 
For $\epsilon, \delta > 0$, the \emph{Morse group} $\Mo_{x,f,\epsilon,\delta}(F)$ in the dg derived category of $k$-modules is the cone on the restriction map
\[
\Gamma(B_\epsilon(x) \cap f^{-1}(-\infty,f(x) + \delta);F) \to \Gamma(B_{\epsilon}(x) \cap f^{-1}(-\infty,f(x) - \delta);F)
\]
For $\epsilon' < \epsilon$ and $\delta' < \delta$, there is a canonical restriction map
\[
\Mo_{x,f,\epsilon,\delta}(F) \to \Mo_{x,f,\epsilon',\delta'}(F)
\]

We recall that $f$ is said to be stratified Morse at $x$ if the restriction of $f$ to the
stratum containing $x$ is either (a) noncritical at $x$ or (b) has a Morse critical point
at $x$ in the usual sense,
and moreover $df_x$ does not lie in the closure of $T^*_S M$ for any larger stratum $S$.
The
above restriction is an isomorphism so long as $\epsilon$ and $\delta$ are sufficiently small, and $f$ is stratified Morse at $x$ 
(\cite{GM}, or in our present sheaf-theoretic setting \cite[Chapter 5]{Schurmann}). 
This allows us to define $\Mo_{x,f}(F)$ unambiguously for $f$ suitably generic with respect to $\cS$ (Proposition 7.5.3 of \cite{KS}).
In fact, $\Mo_{x,f}(F)$ depends only on $df_x$, up to a shift that depends only on the Hessian of $f$ at $x$.

\begin{example}
Suppose $F$ is the constant sheaf on $\bR^n$ and $f = -x_1^2 - \cdots - x_i^2 + x_{i+1}^2 + \cdots + x_n^2$.  Then $\Mo_{0,f}(F)$ is the relative cohomology of the pair $(A,B)$, where $A = \{x \in B_\epsilon(0) \mid f(x) < \delta\}$ is contractible and $B = \{x \in B_\epsilon(0) \mid f(x) < -\delta\}$ has the homotopy type of an $(i-1)$-dimensional sphere (an empty set if $i = 0$), so long as $\delta < \epsilon$.  Thus $\Mo_{0,f}(F) \cong k[i]$ --- the shift is the index of the quadratic form $f$.  
\end{example}

A cotangent vector $(x,\xi) \in T^* M$ is called characteristic with respect to $F$ if, for some stratified Morse $f$ with $df_x = \xi$, the Morse group $\Mo_{x,f}(F)$ is nonzero.\footnote{The notion of characteristic covector depends on the stratification, since in particular the existence of a stratified Morse function
to define the Morse group can already depend on the stratification.  However, the notion of singular support does not.} 
The \emph{singular support} of $F$ is the closure of the set of characteristic covectors for $F$.  We write $\SS(F) \subset T^*M$ for the singular support of $F$.  This notion enjoys
the following properties:

\begin{enumerate}
\item  If $F$ is constructible, then $\SS(F)$ is a conic Lagrangian, i.e. it is stable under dilation (in the cotangent fibers) by positive real numbers, and it is a Lagrangian subset of $T^* M$ wherever it is smooth.
Moreover, if $F$ is constructible with respect to a Whitney stratification $\cS$, then $\SS(F)$ is contained in the characteristic variety of $\cS$, defined as
\[
\Lambda_{\cS} := \bigcup_{S \in \cS} T^*_S M
\]
\item If $F' \to F \to F'' $ is a distinguished triangle in $\dgsh(M)$, then $\SS(F) \subset \SS(F') \cup \SS(F'')$.
\item (Microlocal Morse Lemma) Suppose $f:M \to \bR$ is a smooth function such that, for all $x \in f^{-1}([a,b])$, the cotangent vector $(x,df_x) \notin \SS(F)$.  Suppose additionally that $f$ is proper on the support of $F$.  Then the restriction map 
\[
\Gamma(f^{-1}(-\infty,b);F) \to \Gamma(f^{-1}(-\infty,a);F)
\]
is a quasi-isomorphism.  \cite[Corollary 5.4.19]{KS} 

\end{enumerate}

\begin{remark}
\label{rem:numapompilius}
As a special case of (1), $F$ is locally constant over an open subset $U \subset M$ if and only if $\SS(F)$ contains no nonzero cotangent vector along $U$.
\end{remark}

\begin{definition}
For a conic closed subset (usually a conic Lagrangian) $L \subset T^*M$, we write $\dgsh_L(M; k)$  for the full subcategory 
of $\dgsh(M; k)$ whose objects are sheaves with
singular support in $L$.
\end{definition}

\subsubsection{Legendrian Definitions}  
\label{subsubsec:LegDefs}

We write $0_M$ for the zero section of $T^* M$. 
If $\Lambda \subset T^\infty M$ is Legendrian, we write $\bR_{>0} \Lambda \subset T^* M$ for the cone over $\Lambda$, and 
$$\dgsh_{\Lambda}(M; k):= \dgsh_{\bR_{>0} \Lambda \cup 0_M}(M;k).$$

In our main application, $M$ is either $\bR^2$ or $S^1 \times \bR$, and we prefer to study a full subcategory $\dgsh_{\Lambda}(M)$ of sheaves that vanish on a noncompact region of $M$.  We therefore make the following notation.  If $M$ is the interior of a manifold with boundary and we distinguish one boundary component of $M$, we let $\dgsh_{\Lambda}(M;k)_0$ denote the full subcategory of $\dgsh_{\Lambda}(M;k)$ of sheaves that vanish in a neighborhood of the distinguished boundary component.  In particular we use $\dgsh_{\Lambda}(\bR^2;k)_0$ to denote the full subcategory of compactly supported sheaves, and $\dgsh_\Lambda(S^1 \times \bR;k)_0$ to denote the full subcategory of sheaves that vanish for $z \ll 0$.  

When $\Lambda$ is a Legendrian knot in $T^{\infty,-} \bR^2$ or $T^{\infty,-} (S^1 \times \bR)$, we often make use of a slightly larger category:

\begin{definition}
Let $M$ be one of $\bR^2$ or $S^1 \times \bR$.  When $\Lambda$ is a Legendrian knot in general position, let $\Lambda^+$ be as in \S\ref{subsec:augmentedLeg}.  This is again
a Legendrian (with singularities), so we have categories $\dgsh_{\Lambda^+} (M; k) \supset \dgsh_{\Lambda}(M;k)$ and 
$\dgsh_{\Lambda^+} (M; k)_0 \supset \dgsh_{\Lambda^+}(M;k)_0$.
\end{definition}

\begin{example}[Unknot]
\label{ex:31U}
The standard front diagram $\Phi$ for the Legendrian unknot $\Lambda$ is pictured left.  Let $E$ denote the union of the unique compact region and the upper strand of $\Phi$, but not the cusps.  The set $E$ is a half-open disk (it is homeomorphic to $[0,1) \times [0,1]$).  A chain complex $V^{\bullet}$ determines a constant sheaf with fiber $V^{\bullet}$ on $E$, and its extension-by-zero to $\bR^2$ is pictured right.  When $V^{\bullet} = k$, it is an example of an ``eye sheaf'' \S\ref{sec:rulingsheaf}.
\begin{figure}[H]
\begin{center}
\includegraphics[scale = .2]{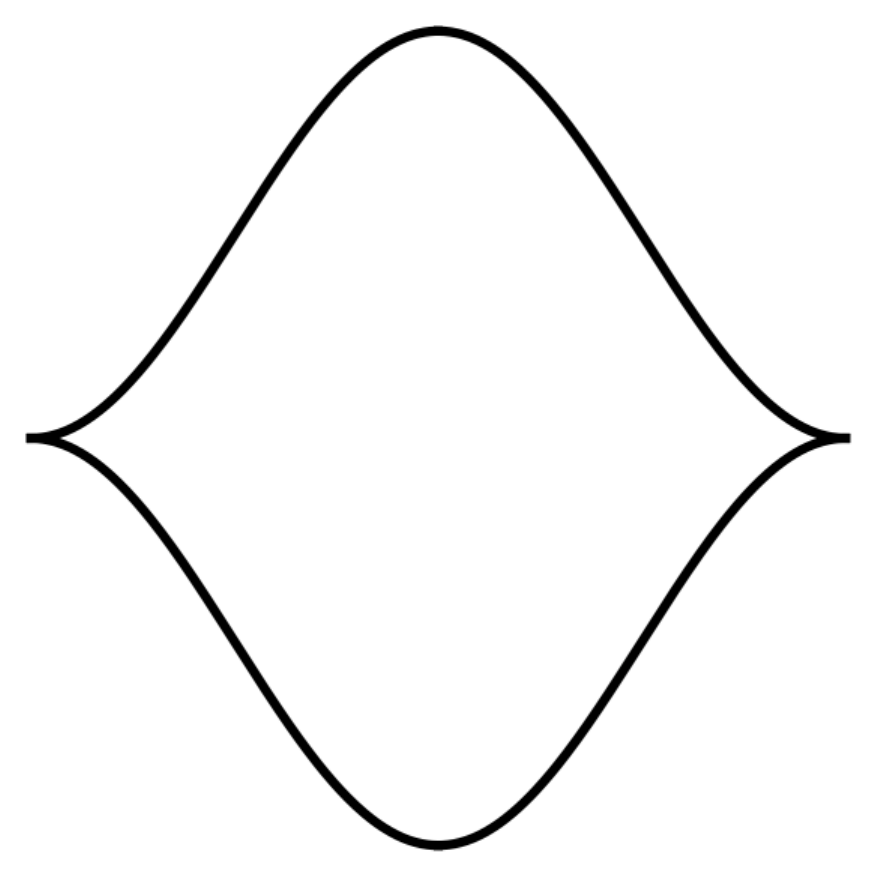} 
\quad
\includegraphics[scale = .2]{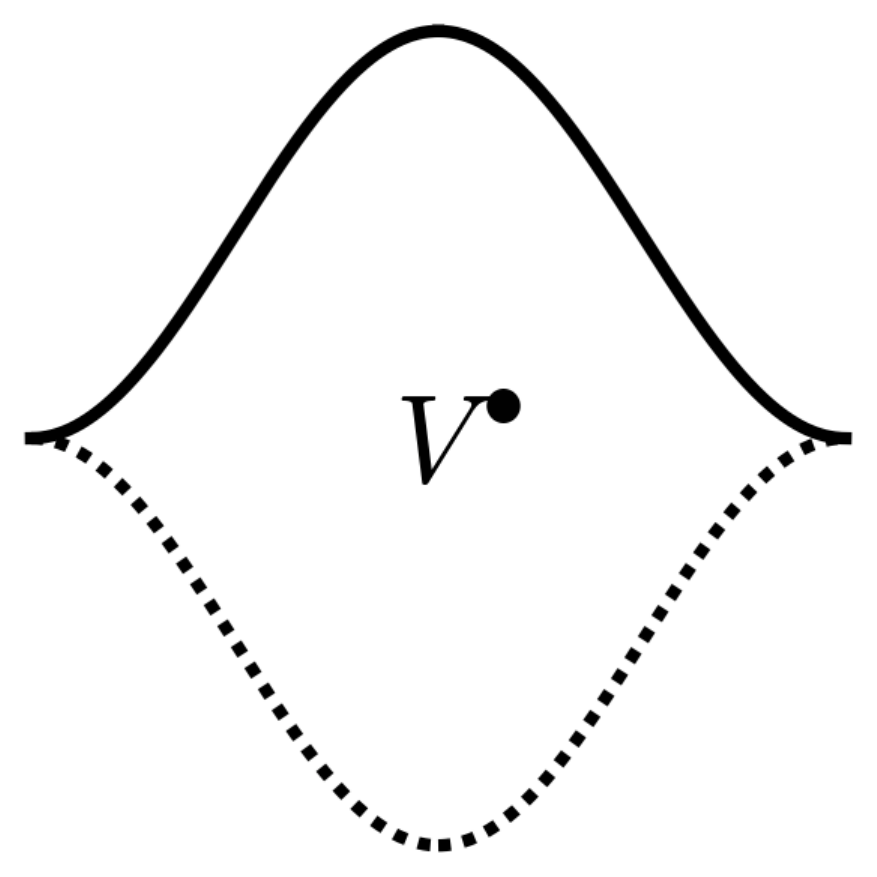}
\caption{Unknot front diagram (left) and sheaf in $\dgsh_{\Lambda}(\bR^2,k)_0$ (right).}
\label{fig:unknoteyesheaf}
\end{center}
\end{figure}
\noindent The microlocal behavior of these sheaves near the cusps is discussed in \cite[Example 5.3.4]{KS}.  In particular, they have singular support in $\Lambda$ and in fact the construction is an equivalence from the derived category of $k$-modules to $\dgsh_{\Lambda}(\bR^2)_0$.
\end{example}

\begin{example}[``R1-twisted'' unknot]
\label{ex:R1unknotsheaf}
Applying the Legendrian Reidemeister move R1 to the standard unknot of Figure
\ref{fig:unknoteyesheaf} gives the unknot $\Lambda$ depicted in Figure \ref{fig:twistedunknotsheaf}.  Legendrian
Reidemeister moves
can be realized by Legendrian isotopy, and we show in Theorem
\ref{thm:4.1} that $ \dgsh_{\Lambda}(\bR^2)_0$ is a Legendrian isotopy invariant, so the sheaves 
described in Example \ref{ex:31U} must have counterparts here.  They are produced as follows.  
Note the picture in Figure \ref{fig:twistedunknotsheaf} is topologically the union of two unknot diagrams, which
intersect at the crossing point $c$. 
Let $V^\bullet$ be a chain complex, and let $\mathcal{V}_t^\bullet, \mathcal{V}_b^\bullet$ be the sheaves
obtained by applying the construction of Example \ref{ex:31U} separately to the top and bottom unknot diagrams.  
Then  the sheaf 
$\cH om(\mathcal{V}_b^\bullet, \mathcal{V}_t^\bullet)$ 
is a skyscraper sheaf at $c$ with stalk
$\Hom(V^\bullet, V^\bullet)[-2]$. 
Thus 
$\Ext^1(\mathcal{V}_b^\bullet, \mathcal{V}_t^\bullet[1])$
$= \Hom(V^\bullet, V^\bullet)$.  One can show that
the singular support condition imposed by $\Lambda$ at the crossing $c$ is satisfied by an extension
as above if and only if its class is an isomorphism in $\Hom(V^\bullet, V^\bullet)$.  These extensions are the counterparts
of the sheaves described in Example \ref{ex:31U}, and the case $V^\bullet = k$ is illustrated in 
Figure \ref{fig:twistedunknotsheaf}. 
\begin{figure}[H]
\begin{center}
\includegraphics[scale = .3]{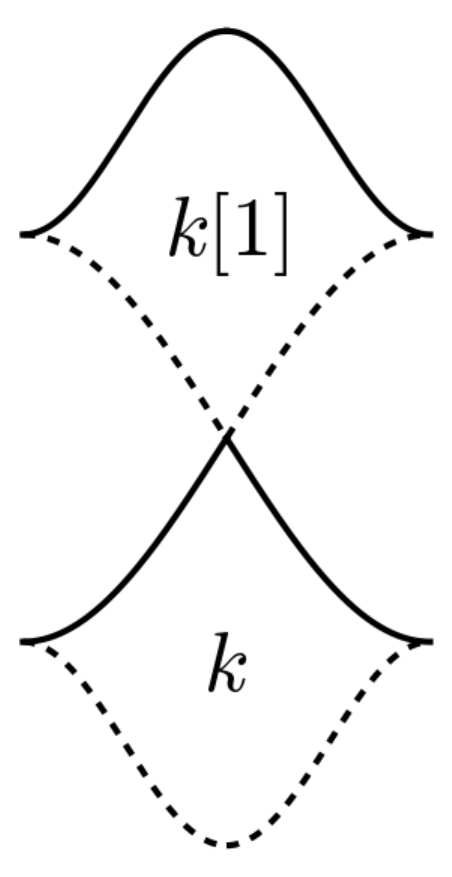}
\caption{A sheaf on the R1-twisted unknot.}
\label{fig:twistedunknotsheaf}
\end{center}
\end{figure}
\end{example}

\subsection{A Fukaya category}

We recall the relationship between
the Fukaya category of a cotangent bundle and constructible sheaves on the base manifold, which motivates many of the constructions of this paper.  In spite of this motivation, the rest of the paper is formally independent of this section, so we will not provide many details.

In \cite{NZ}, the authors
define the so-called ``unwrapped'' Fukaya category of the cotangent bundle $T^* M$
of a compact real analytic manifold $M$.
Objects are constructed from exact Lagrangian submanifolds
with some properties
and structures whose description we relegate here to a footnote (see \cite{NZ}
for details).
A composition law is constructed from moduli spaces of holomorphic disks in the usual way.  
The additional structures and properties ensure graded morphisms as well as
regularity, orientedness and compactness
of these moduli spaces.\footnote{The
additional structures are a local system on the Lagrangian, a brane structure,
and perturbation data.  The brane structure
consists of a relative pin structure ensuring orientedness of the moduli spaces
and a grading, ensuring graded morphisms.  The perturbation data
ensures regularity.  A Lagrangian object $L$ need not be compact, but must
have a ``good'' compactification $\overline{L}$ in the
spherical compactification $\overline{T}^*M$ of $T^*M$, meaning $\overline{L}$
is a subanalytic subset (or lies in a chosen analytic-geometric category).
This property, as well as an additional required ``tame perturbation,'' ensures
compactness of the moduli space by (1) bringing intersections into
compact space from infinity through perturbations, and (2) establishing tameness
so that holomorphic disks stay within compact domains and do not run off to infinity.
Lagrangian objects do not generate a triangulated category.  The triangulated
envelope is constructed from twisted complexes whose graded pieces are
Lagrangian objects.  (Another option for constructing the triangulated envelope is
the the full triangulated subcategory of modules generated by
Yoneda images of Lagrangian objects).}
We denote the triangulated $A_\infty$-category so constructed by $\twFuk(T^*M).$

\subsubsection{Standard opens and microlocalization}

In \cite{NZ}, a ``microlocalization'' functor
$$\mu_M: \dgsh(M) \to \twFuk(T^*M)$$
was constructed and shown to be a quasi-embedding; later in \cite{N} this functor was
shown to be an equivalence. 
We sketch here the construction, which is largely motivated by Schmid-Vilonen's \cite{SV}.  A fundamental observation is that, as a triangulated
category, $\dgsh(M)$ is generated by \emph{standard opens,} the pushforwards
of local systems on open sets.  Let $j:U\hookrightarrow M$ be an open embedding,
$\cL$ a local system on $U$, and
and let $j_* \cL$ be the associated standard open.  To define the corresponding
Lagrangian object $L_{U,\cL}$, choose a defining function $m:\overline{U}
\rightarrow \bR_{\geq 0}$, positive on $U$ and zero on the boundary $\overline{U}\setminus U,$
and set $$L_U := \Gamma_{d \, \log \, m\vert_U},$$
the graph of $d\, \log \, m$ over $U.$
Note that if $U = M$ then $L_U = L_M$ is the zero section.
We equip $L_U$ with the local system $\pi^*\cL,$ where $\pi: L_U
\stackrel{\sim}\rightarrow U$ and write $L_{U,\cL}$ for the corresponding object.

\subsubsection{The case of noncompact $M$}

In the formalism of \cite{NZ}, $M$ is assumed to be compact.  One crude way around this is to embed $M$ as a real analytic open subset of a compact manifold $M'$.  The functor $j_*$ is a full embedding of $\dgsh(M)$ into $\dgsh(M')$, and we may define $\twFuk(T^* M)$ to be the corresponding full subcategory of $\twFuk(T^* M')$.  

With this definition, the Lagrangians that appear in branes of $\twFuk(T^*M)$ will be contained in $T^*M \subset T^* M'$, but they necessarily go off to infinity in the fiber directions as they approach the boundary of $M$.  The Schmid-Vilonen construction suggests a more reasonable family of Lagrangians in $T^* M$.  Namely, suppose the closure of $M$ in $M'$ is a manifold with boundary (whose interior is $M$), and let $m:M' \to \bR$ be positive on $M$ and vanish outside of $M$.  Then we identify $T^*M$ with an open subset of $T^* M'$ not by the usual inclusion, but by $j_m:(x,\xi) \mapsto (x,\xi + (d \, \log\, m)_x)$. 

Thus, we define $\twFuk(T^* M)$ to be the category whose objects are twisted complexes of Lagrangians whose image under $j_m$ is an object of $\twFuk(T^*M')$ (and somewhat tautologically compute morphisms and compositions in $\twFuk(T^* M')$ as well).

\subsubsection{Singular support conditions}
\label{subsubsec:Xin}

There is a version of the local Morse group functor \cite{J} which allows us to define a purely
Fukaya-theoretic notion of singular support (see also  \S 3.7 of \cite{N2}). 
If $\Lambda \subset T^{\infty} M$ is a Legendrian,
we write 
$\twFuk_{\Lambda}(T^* M)$ for the subcategory of $\twFuk(T^*M)$ of objects whose
singular support at infinity is contained in $\Lambda$; by 
\cite{J} it is the essential image under $\mu_M$ of $\dgsh_{\Lambda}(M)$.  According to \cite{SV}, the singular support of a standard open sheaf is the limit $\lim_{\epsilon \rightarrow 0^+}\Gamma_{\epsilon d\, \log\, m\vert_U},$ thus we regard $\twFuk_{\Lambda}(T^* M)$ informally as the subcategory of objects ``asymptotic to $\Lambda$.''  The counterpart $\twFuk_{\Lambda}(M)_0$ of $\dgsh_{\Lambda}(M)_0$ is the full subcategory spanned by branes whose projection to $M$ is not incident with the distinguished boundary component of $M$.

\begin{example}[Unknot]
\label{ex:32U}
Let $\Phi$ and $\Lambda$ be as in Example \ref{ex:31U}. 
Let $m_T:\bR^2 \to \bR$ vanish on the top strand and be positive below it, and let $m_B:\bR^2 \to \bR$ vanish on the bottom strand and be positive above it.  Consider $f = \log(m_T/m_B)$ defined on the overlap.  Then the graph of $d \log(f)$ (with appropriate brane structures) is the correspondent of the sheaf of Example \ref{ex:31U}.
\end{example}

\subsection{A combinatorial model}
\label{subsec:acm}

When $M = \bR^2$ or $S^1 \times \bR$ and 
$\Lambda \subset T^{\infty, -} M$ is a Legendrian knot in general position --- i.e., its front diagram $\Phi$ has
only cusps and crossings as singularities ---
we give here a combinatorial description of $\dgsh_{\Lambda}(M)$.  The sheaves in this category are constructible with respect
to the Whitney stratification of $M$ in which the zero-dimensional strata are the cusps and crossings, 
the one-dimensional strata are the arcs, and the two-dimensional strata are the regions.

\begin{definition} 
Given a stratification $\cS$, the star of a stratum $S \in \cS$ is the union of strata that contain $S$ in their closure. 
We view $\cS$ as a poset category in which every stratum has a unique map (generization) to every stratum in its star. 
We say that $\cS$ is a regular cell complex if every stratum is contractible and moreover
the star of each stratum is contractible.  
\end{definition}

Sheaves constructible with respect to a regular cell complex can be captured combinatorially, as a
subcategory of a functor category from a poset.  If $C$ is any category and $A$ is an abelian category,
we write $\dgfun_{\mathit{naive}}(C,A)$ for the dg category of functors from $C$ to the category 
whose objects are cochain complexes in $A$,
and whose maps are the cochain maps. 
We write $\dgfun(C,A)$ for the dg quotient \cite{D}
of  $\dgfun_{\mathit{naive}}(C, A)$ by the thick subcategory of functors taking values in acyclic complexes. 
For a ring $k$, we abbreviate the case where $A$ is the abelian category of $k$-modules to $\dgfun(C,k)$.

\begin{proposition} \label{prop:star} \cite{Kashiwara84} \cite{Shepard}, \cite[Lemma 2.3.2]{N}.
Let $\cS$ be a Whitney stratification of the space $M$.  Consider the functor
\begin{equation}
\label{eq:luciustarquiniuspriscus}
 \Gamma_{\cS}: \dgsh_\cS(M;k)  \to  \dgfun(\cS,k) \qquad \qquad
\cF  \mapsto  [s \mapsto \Gamma(\text{\rm star of $s$};\cF) ]
\end{equation}
If $\cS$ is a regular cell complex, then $\Gamma_{\cS}$ is a quasi-equivalence.  
\end{proposition}

\begin{remark} Note in case $\cS$ is a regular cell complex, the restriction map from 
$ \Gamma(\text{star of $s$};\cF)$ to the stalk of $\cF$ at any point of $s$ is a
quasi-isomorphism. \end{remark}

The Whitney stratification induced by a front diagram is not usually regular, so the corresponding 
functor $\Gamma$ can fail to be an equivalence.
We therefore choose a regular cell complex $\cS$ refining this stratification.  For convenience we will require that the
new one-dimensional strata do not meet the crossings (but they may meet the cusps).  
By Proposition \ref{prop:star} above, the restriction of 
$\Gamma_{\cS}$ to the full subcategories $\dgsh_{\Lambda}(M; k)$, $\dgsh_{\Lambda}(M;k)_0$, $\dgsh_{\Lambda^+}(M; k)$ and $\dgsh_{\Lambda^+}(M; k)_0$
of $\dgsh_{\cS}(M; k)$
is quasi-fully faithful.  By describing their essential images in $\dgfun(\cS, k)$, we give combinatorial models of these categories. 
Describing the essential image amounts to translating the singular 
support conditions $\Lambda$ and $\Lambda^+$ into constraints on elements of $\dgfun(\cS,k) $.

\begin{definition}
\label{def:romulus}
Let $\cS$ be a regular cell complex refining the stratification induced by the front diagram. 
We write $\dgfun_{\Lambda^+}(\mathcal{S}, k)$ for the full subcategory of  $\dgfun(\mathcal{S}, k)$ 
whose objects are characterized by:
\begin{enumerate}
\item Every map from a zero dimensional stratum in $\cS$ which is not a cusp or crossing, or from a one dimensional stratum
which is not contained in an arc, is sent to a quasi-isomorphism.  
\item If $S, T \in \mathcal{S}$ such that $T$ bounds $S$ from above, then $T \to S$ is sent to a quasi-isomorphism. 
\end{enumerate}

We write $\dgfun_{\Lambda}(\mathcal{S}, k)$ for the full subcategory of $\dgfun_{\Lambda^+}(\mathcal{S}, k)$
whose objects are characterized by satisfying the following additional condition at each crossing $c \in \cS$. 

Label the subcategory of $\cS$ of all objects admitting maps from $c$ as follows:
\begin{equation}
\label{eq:tullushostilius}
\xymatrix{
& & N\\
& nw \ar[ur] \ar[dl] & & ne \ar[ul] \ar[dr] \\
W  & & c \ar[ul] \ar[ur] \ar[dl] \ar[dr] \ar[uu] \ar[dd] \ar[ll] \ar[rr] & & E\\
& sw \ar[ul] \ar[dr] & & se \ar[ur]   \ar[dl] \\
& & S
}
\end{equation}
As all triangles in this diagram commute, we may form a bicomplex $F(c) \to F(nw) \oplus F(ne) \to F(N)$.
Then the additional condition that $F \in \dgfun_{\Lambda}(\mathcal{S}, k)$ must satisfy is that the complex
\[
\mathrm{Tot}\bigg(F(c) \to F(nw) \oplus F(ne) \to F(N)\bigg)
\]
is acyclic. 
\end{definition}

\begin{theorem} \label{thm:comb} Let $\cS$ be a regular cell complex refining the stratification induced
by the front diagram.  The essential image of $\Gamma_{\cS}: \dgsh_{\Lambda^+}(M; k) \to \dgfun(\mathcal{S}, k)$ 
is $\dgfun_{\Lambda^+}(\mathcal{S}, k)$, and the essential image of 
 $\dgsh_{\Lambda} (M; k)$ is $\dgfun_{\Lambda}(\mathcal{S}, k)$. \end{theorem}
\begin{proof}
The assertion amounts to the statement that a sheaf on $M$ has certain singular support if and only if its image in 
$\dgfun(\mathcal{S}, k)$ has certain properties.  Both the calculation of singular support and the asserted properties 
of the image can be checked locally --- we may calculate on a small neighborhood of a point in a stratum.
The singular support does not depend on the stratification, so can be calculated using only generic points in
the stratification induced by the front diagram.  It follows that the condition of Definition \ref{def:romulus}\,(1)
must hold
at a zero dimensional stratum in $\cS$ which is not a cusp or crossing, or at a point in a one dimensional stratum
which is not contained in an arc.
It remains to study the singular support condition in the vicinity of an arc, of a cusp,
and of a crossing. 

The basic point is that by definition the Legendrian lives in the downward conormal space.  Thus for any $p \in M$  
and (stratified Morse)
$f$ such that $df_p = adx + b dz$ with $b > 0$, and any $\cF \in \dgsh_{\Lambda^+}(M,k)$, we must have $\Mo_{x,f} = 0$. 
It is essentially obvious from this that the downward generization maps in $\cS$ are sent by $\Gamma_{\cS}(\cF)$ to quasi-isomorphisms;
the significance of the following  calculations is that
they will show this to be the only condition required, except at crossings. 

{\bf Arc.}  Let $L$ be a one-dimensional stratum in $\cS$ which bounds the two-dimensional stratum $B$ from above. 
Let $\ell \in L$ be any point, and $B_\epsilon(\ell)$ a small ball around $\ell$.  Then for  $\cF \in \dgsh_{\cS}(M,k)$, we should show
$\cF \in \dgsh_{\Lambda^+}(B_\epsilon(\ell), k)$ if and only if $\Gamma_S(\cF)(L \to B)$ is a quasi-isomorphism. 

Let $A$ be the two dimensional cell above $L$. 
Let $f$ be negative on $B$, zero on $L$, positive on $A$, and have $df$ nonvanishing along $L$.  Then $f$ is stratified
Morse, and $df_\ell$ ``points into $A$'', 
so if $A$ is above $B$, we have $df_\ell \notin T^{\infty, -} \RR^2$.  So for sufficiently small $\delta$, 
we must have 
$$0 \cong \Mo_{\ell,f}(\cF) \cong \mathrm{Cone}(\Gamma(f^{-1}(-\infty, \delta) \cap B_\epsilon(\ell), \cF) \to 
\Gamma(f^{-1}(-\infty, -\delta) \cap B_\epsilon(\ell), \cF)) \cong $$
$$\mathrm{Cone}(\Gamma(L, \cF) \to \Gamma(B, \cF)) \cong \mathrm{Cone}(\Gamma(\mathrm{star}\, L, \cF) \to \Gamma(\mathrm{star}\, B, \cF))
\cong \mathrm{Cone}(\Gamma_S(\cF)(L \to B))$$

{\bf Cusp.} 
Let $c$ be a cusp point.   First we treat the case when no additional one-dimensional strata were introduced ending at $c$. 
Then the subcategory of $\cS$ of objects that receive a map from $c$ consists of two arcs $a$ above and $b$ below
the cusp, and two regions, $I$ inside the cusp and $O$ outside.  The maps involved are as below: 
\begin{equation}
\label{eq:luciustarquiniuscollatinus}
\xymatrix{
& & a \ar[dll] \ar[d]  \\
O &  c \ar[dr] \ar[ur]  \ar[l] \ar[r] & I \\
& & b \ar[ull] \ar[u]
}
\end{equation}

The claim near a cusp: 
for $\cF \in \dgsh_{\cS}(M;k)$, we have
$\cF \in \dgsh_{\Lambda}(B_\epsilon(c); k) = \dgsh_{\Lambda^+}(B_\epsilon(c); k)$ if and only if the maps
\[
\begin{array}{ccccc}
\Gamma_S(\cF)(c \to b) & & \Gamma_S(\cF)(b \to O) & & \Gamma_S(\cF)(a \to I) 
\end{array}
\]
are all quasi-isomorphisms.  

We check stratum by stratum; in $O, I$ there is nothing to check; in $a, b$ the statement
is just what we have already shown above.  It remains to compute Morse groups at $c$.   Since these depend (up to a shift) only
on the differential, we may restrict ourselves to linear Morse functions $f = \alpha x + \beta z$.
This function is stratified Morse at $c$ whenever $\alpha \neq 0$.  On the other hand, the only point at infinity of $\Lambda$ (or $\Lambda^+$)
over $c$ is 
$-dz,$ which would require $\alpha = 0$.  Thus the condition 
at the stratum $c$ of being
in $\dgsh_{\Lambda^+}(B_\epsilon(c), k)$ is the vanishing of the Morse group $\mathrm{Mo}_{c,f}(\cF)$ for all nonzero $\alpha$.

The Morse group is the cone on a morphism
$
\Gamma(Y_1,\cF) \to \Gamma(Y_2,\cF)
$
where $Y_1$ and $Y_2$ are, respectively, the intersection of a 
small open
set containing the cusp with $f^{-1}((-\infty, \epsilon))$ and 
$f^{-1}((-\infty, -\epsilon))$.  Topologically, there are two cases, 
according as whether or not the line 
$\alpha x + \beta z = \epsilon$ intersects the positive $x$ axis; 
i.e., according as to whether $\alpha > 0$ or $\alpha < 0$.  

The following picture depicts the lines $f = \pm \epsilon$, 
where we have taken $f$ to be $\alpha x + \beta z$ for $\alpha > 0$;
it also depicts these lines if we take $f = - (\alpha x + \beta z)$.  
Thus it serves to represent both topological possibilities.

\begin{center}
\includegraphics[scale = .33]{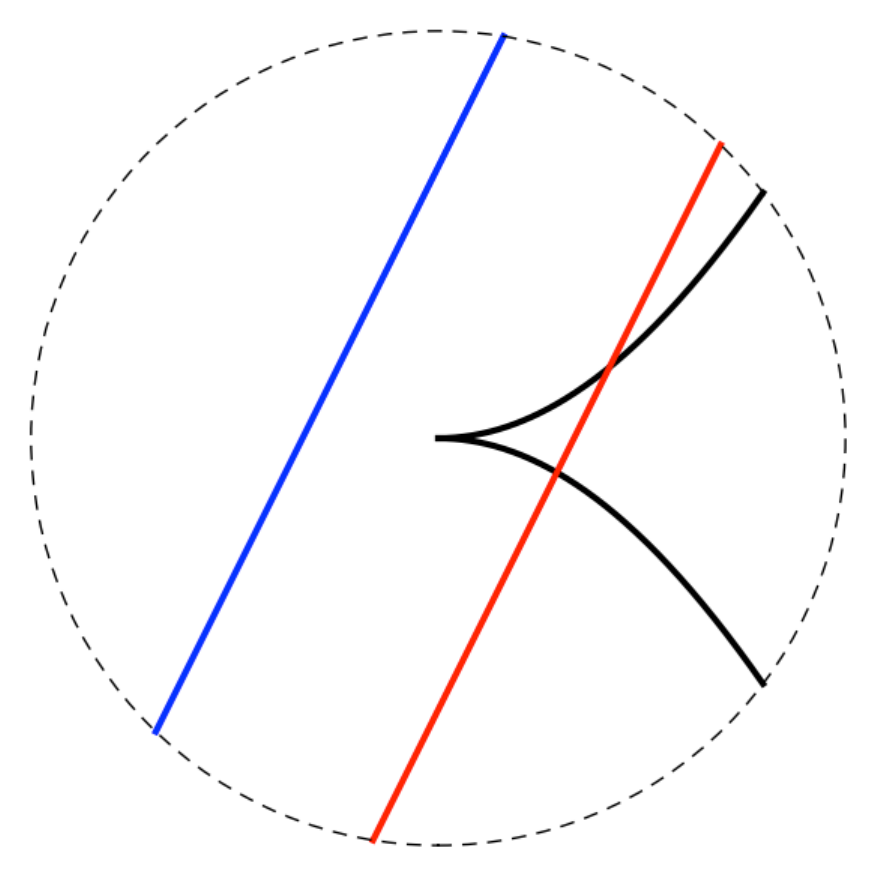}
\end{center}

The case $f = - (\alpha x + \beta z)$ is when $Y_1$ and $Y_2$ are the regions to the right of the blue and the red lines, respectively.
In this case, $\Gamma(Y_1, \cF)$ is naturally identified with 
$\Gamma(B_\epsilon(C), \cF) \cong \Gamma_{\cS}(\cF) (c)$, and 
$\Gamma(Y_2, \cF)$ with the cone of
\[
\Gamma_{\cS}(\cF) (a) \oplus \Gamma_{\cS}(\cF) (b) \xrightarrow{\Gamma_{\cS}(\cF) (a \to I) \ominus
\Gamma_{\cS}(\cF) (b \to I)} \Gamma_{\cS}(\cF) (I)
\]
The map $\Gamma(Y_1, \cF) \to \Gamma(Y_2, \cF)$ is induced by the morphisms 
$c \to a, b$, and is a quasi-isomorphism if and only if the total complex 
$$\Gamma_{\cS}(\cF)(c) \to \Gamma_{\cS}(\cF)(a) \oplus \Gamma_{\cS}(\cF)(b) \to \Gamma_{\cS}(\cF)(I) $$
is acyclic.  
As $\Gamma_\cS(a \to I)$ is a quasi-isomorphism, this happens if and only if 
$\Gamma_{\cS}(c \to b)$ is a quasi-isomorphism as well. 

The case $f = \alpha x + \beta z$ is when $Y_1$ and $Y_2$ are the regions to the left of the red and blue lines, respectively.  In this case 
the map is just $\Gamma_{\cS}(\cF)(c) \to \Gamma_{\cS}(\cF)(O)$.  This map is the composition 
$\Gamma_{\cS}(b \to O) \circ \Gamma_{\cS}(c \to b)$, both of which must be quasi-isomorphisms if the other 
singular support conditions hold.  

{\bf Cusp with additions.}  
The situation if some additional one-dimensional strata terminate at the cusp $c$ is  similar.  The inside
and outside  are subdivided; the outside as
$b \to O_1 \leftarrow o_1 \rightarrow O_2 \leftarrow o_2 \cdots \rightarrow O_n \to a $, and similarly the inside 
as $b \to I_1 \leftarrow i_1 \rightarrow I_2 \leftarrow i_2 \cdots \rightarrow I_n$, where all maps $i \to I$ or $o \to O$, 
and also $a \to I_n$ and $b \to O_1$ are
quasi-isomorphisms.  For a Morse function topologically like $f=-x$ , we now have
$$\Gamma(Y_1, \cF) =
\mathrm{Cone} \left( \Gamma_\cS(\cF) (a) \oplus \Gamma_\cS(b) \oplus \bigoplus \Gamma_\cS(\cF)(i_k) 
\to \oplus \bigoplus \Gamma_\cS(\cF)(I_k) \right)$$ while again
$\Gamma(Y_2, \cF) = \Gamma_\cS(\cF)(c)$.  
As before, since 
$\Gamma_\cS(\cF) (a) \oplus \bigoplus \Gamma_\cS(\cF)(i_k) \to \oplus \bigoplus \Gamma_\cS(\cF)(I_k)$
is a quasi-isomorphism, the acyclicity of the cone $\Gamma(Y_1, \cF) \to \Gamma(Y_2, \cF)$ is equivalent
to $\Gamma_\cS(\cF)(c) \to \Gamma_\cS(b)$ being a quasi-isomorphism.  
Note that since
$b \to O_1 \leftarrow o_1 \rightarrow O_2 \leftarrow o_2 \cdots \rightarrow O_n$ are all quasi-isomorphisms, 
it follows that all maps $\Gamma_{\cS}(\cF)(c) \to \Gamma_{\cS}(\cF)(O_i), \Gamma_{\cS}(\cF)(o_i)$ are
quasi-isomorphisms as well.  

For other Morse functions, the Morse group is 
$$\mathrm{Cone}\left(\Gamma_{\cS}(\cF)(c) \to \mathrm{Cone}\left(\bigoplus_{i=l}^m \Gamma_\cS(\cF)(o_i) \to 
\bigoplus_{i=l}^{m-1} \Gamma_{\cS}(\cF)(O_i) \right) \right)$$
and so again its acyclicity is ensured once the other singular support conditions are satisfied.

{\bf Crossing.} 
Let $c \in \cS$ be a crossing.  Let $N, E, S, W$ be the north, east, south, and west regions adjoining $c$, and let
$nw, ne, sw, se$ be the arcs separating them. 
The subcategory of $\cS$ of objects receiving maps from $c$ is: 
\[
\xymatrix{
& & N\\
& nw \ar[ur] \ar[dl] & & ne \ar[ul] \ar[dr] \\
W  & & c \ar[ul] \ar[ur] \ar[dl] \ar[dr] \ar[uu] \ar[dd] \ar[ll] \ar[rr] & & E\\
& sw \ar[ul] \ar[dr] & & se \ar[ur]   \ar[dl] \\
& & S
}
\]
For $\cF \in \dgsh_\cS(M,k)$ to be in $\dgsh_{\Lambda^+}(B_\epsilon(c), k)$ near the arcs is equivalent to 
the maps $nw \to W$, $ne \to E$, $sw \to S$, $se \to S$ all being carried to quasi-isomorphisms by 
$\Gamma_{\cS}(\cF)$.  It remains to show that having $\cF \in \dgsh_{\Lambda^+}(B_\epsilon(c), k)$ is equivalent 
to these together with the further condition that $c \to se$ and $c \to sw$ are sent to quasi-isomorphisms as well, and 
that having $\cF \in \dgsh_{\Lambda}(B_\epsilon(c), k)$ is equivalent to all of the above, together with the further condition that
\[
\Gamma_{\cS}(\cF) (c) \to \Gamma_{\cS}(\cF) (nw) \oplus \Gamma_{\cS}(\cF) (ne) \to \Gamma_{\cS}(\cF) (N)
\]
is acyclic.

Again we use a linear function to do Morse theory.  If $f(x,z) = ax + bz$ is transverse to the two strands of the crossing, then it is Morse with respect to the stratification.  There are (topologically) four cases, which we can represent by $f = x, -x, z, -z$.  A sheaf is in $\dgsh_{\Lambda^+}(B_\epsilon(c), k)$
if the Morse groups corresponding to the 
first three vanish (as these correspond to non-negative covectors), and in $\dgsh_{\Lambda}(\RR^2)$  if the Morse group of the fourth vanishes as well. 

\begin{center}
\includegraphics[scale = .33]{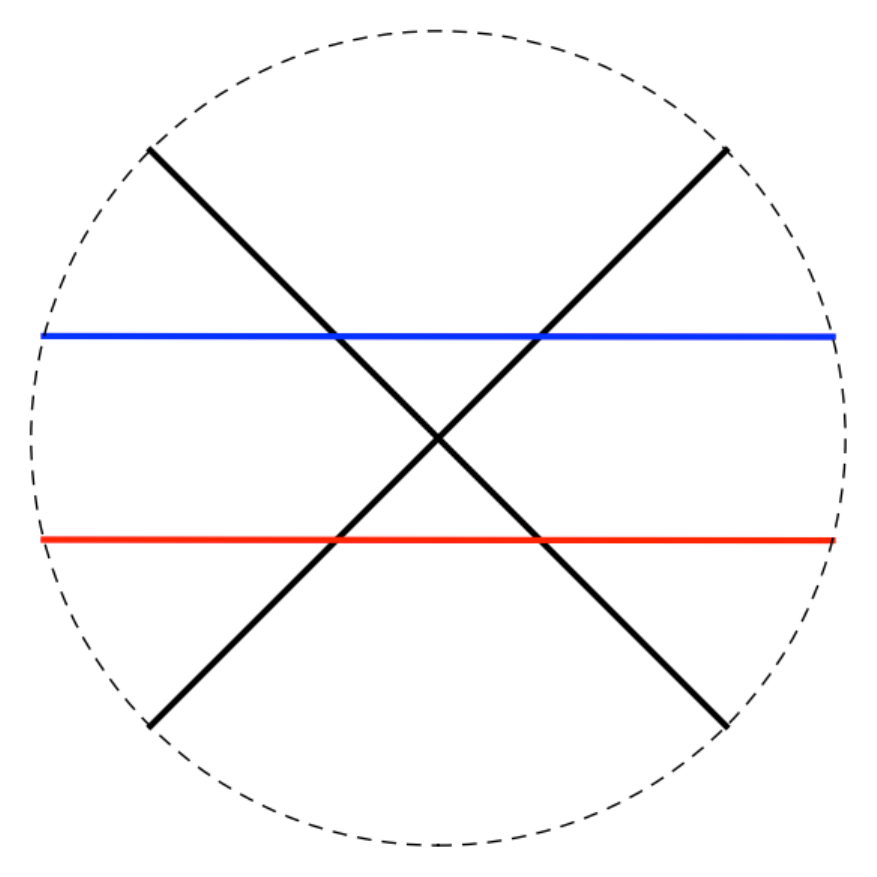}
\end{center}
When $f = z$, then
$f^{-1}(-\infty, \delta)$ is the region below the blue line in the figure above, and $f^{-1}(-\infty, \delta)$ is the region below the red line. 
We have $\Gamma(f^{-1}(-\infty, \delta), \cF) \cong \Gamma_{\cS}(\cF)(c)$ and 
on the other hand $\Gamma(f^{-1}(-\infty, -\delta), \cF)\cong \mathrm{Tot}( \Gamma_{\cS}(\cF)(se) \oplus  \Gamma_{\cS}(\cF)( sw) \to  \Gamma_{\cS}(\cF)(S) )$.  Thus
the vanishing of the Morse group $Mo_{c,f(x,z)=z}(\cF)$ 
is equivalent to the acyclicity of the complex 
$$ \mathrm{Tot}(\Gamma_{\cS}(\cF)(c) \to  \Gamma_{\cS}(\cF)(se) \oplus  \Gamma_{\cS}(\cF)( sw) \to  \Gamma_{\cS}(\cF)(S) )$$ 
Since $\Gamma_{\cS}(\cF)( sw) \to  \Gamma_{\cS}$ and $\Gamma_{\cS}(\cF)( se) \to  \Gamma_{\cS}$ are quasi-isomorphisms,
the kernel of $\Gamma_{\cS}(\cF)(se) \oplus  \Gamma_{\cS}(\cF)( sw) \to  \Gamma_{\cS}(\cF)(S)$ is quasi-isomorphic to each
of $\Gamma_{\cS}(\cF) (se)$ and $\Gamma_{\cS}(\cF) (sw)$.  Thus the vanishing of this Morse group is equivalent to the natural maps
$\Gamma_{\cS}(\cF)(s \to se)$ and $\Gamma_{\cS}(\cF)(s \to sw)$ being quasi-isomorphisms.

When $f = x$, vanishing of the Morse group $Mo_{c,f(x,z)=x}(\cF)$ is likewise equivalent to acyclicity of 
$$ \mathrm{Tot}(\Gamma_{\cS}(\cF)(c) \to  \Gamma_{\cS}(\cF)(sw) \oplus  \Gamma_{\cS}(\cF)( nw) \to  \Gamma_{\cS}(\cF)(W) )$$ 
Since $\Gamma_\cS(\cF)(nw \to W)$ is a quasi-isomorphism, acyclicity of the complex is equivalent to requiring that 
$\Gamma_{\cS}(\cF)(c) \to  \Gamma_{\cS}(\cF)(sw) $ be acyclic as well. 

When $f = -x$, vanishing of the Morse group $Mo_{c,f(x,z)=-x}(\cF)$ is likewise equivalent to acyclicity of 
$$ \mathrm{Tot}(\Gamma_{\cS}(\cF)(c) \to  \Gamma_{\cS}(\cF)(se) \oplus  \Gamma_{\cS}(\cF)( ne) \to  \Gamma_{\cS}(\cF)(E) )$$ 
Since $\Gamma_\cS(\cF)(ne \to E)$ is a quasi-isomorphism, acyclicity of the complex is equivalent to requiring that 
$\Gamma_{\cS}(\cF)(c) \to  \Gamma_{\cS}(\cF)(se) $ be acyclic as well.  

Thus we see that the singular support condition $\dgsh_{\Lambda^+}(B_\epsilon(c), k)$ near a crossing is equivalent to the statement
that downward arrows be sent to quasi-isomorphisms by $\Gamma_{\cS}(\cF)$.  

Finally, when $f = -z$, vanishing of the Morse group $Mo_{c,f(x,z)=-z}(\cF)$ is equivalent to the acyclicity of the complex
$$ \mathrm{Tot}(\Gamma_{\cS}(\cF)(c) \to  \Gamma_{\cS}(\cF)(ne) \oplus  \Gamma_{\cS}(\cF)( nw) \to  \Gamma_{\cS}(\cF)(N) )$$ 
which is the remaining  condition asserted to be imposed by $\dgsh_{\Lambda}(B_\epsilon(c), k)$.
\end{proof}

\begin{remark}
Suppose $M$ is one of $\bR^2$ or $S^1 \times \bR$, and $\Lambda$ is compact in $T^{\infty,-} M$.
Under the equivalence of the Theorem, the subcategories $\dgsh_{\Lambda}(M,k)_0, \dgsh_{\Lambda^+}(M,k)_0$ of \S\ref{subsubsec:LegDefs} are carried to the full subcategory of $\dgfun_{\Lambda}(\cS,k), \dgfun_{\Lambda^+}(M,k)$ spanned by functors whose value on any cell in the noncompact region (for $\bR^2$) or the lower region (for $S^1 \times \bR$) is an acyclic complex.  We denote these categories by $\dgfun_{\Lambda}(M,k)_0, \dgfun_{\Lambda^+}(M,k)_0$.
\end{remark}

In order to discuss cohomology objects, 
we write $\Sh(M,k)$ for the category of constructible sheaves (not complexes of sheaves) on $M$ with coefficients
in $k$, and similarly $\Sh_\cS(M,k)$ for the sheaves constructible with respect to a fixed stratification.  
Then any $\cF \in \dgsh(M)$ has cohomology sheaves $h^i(\cF) \in \mathrm{Sh}(M)$. 
When $\cS$ is a stratification, we write $\Fun(\cS, k)$ for the functors from the poset of $\cS$ to the category of $k$-modules,
so that 
$F \in \dgfun(\cS, k)$ has cohomology objects $h^i(F) \in \Fun(\cS, k) $ given by 
$h^i(F) (s) = h^i( F(s) )$.  We view $\mathrm{Sh}(M)$ as a subcategory of $\dgsh(M)$ in the obvious way,
and similarly $\Fun(\cS, k) $ as a subcategory of $\dgfun(\cS, k)$; we write 
$\mathrm{Sh}_{\Lambda^+}(M) = \dgsh_{\Lambda^+}(M) \cap \mathrm{Sh}(M)$,
and $\Fun_{\Lambda^+}(S; k) = \dgfun_{\Lambda^+}(S; k) \cap \Fun(S; k)$, 
etcetera.

\begin{proposition}
\label{prop:cohompreserve}
The categories with singular support in $\Lambda^+$ are preserved
by taking cohomology, that is, 
 $h^i: \dgfun_{\Lambda^+}(\cS, k) \to \Fun_{\Lambda^+}(\cS; k)$ and 
$h^i: \dgsh_{\Lambda^+}(M; k) \to \mathrm{Sh}_{\Lambda^+}(M) $.
\end{proposition}
\begin{proof}
The category $\dgfun_{\Lambda^+}(\cS, k)$ is defined by requiring certain maps in $\cS$ to go
to quasi-isomorphisms.  By definition the corresponding maps in the cohomology objects will go to isomorphisms,
so the cohomology objects are again in $\dgfun_{\Lambda^+}(\cS, k)$.  The statement about
sheaves follows from Theorem \ref{thm:comb}.
\end{proof}

\begin{remark}
\label{rem:cohompreserve}
More generally, truncation in the usual $t$-structure preserves $ \dgsh_{\Lambda^+}(M; k)$ but
need not preserve $\dgsh_{\Lambda}(M; k)$.  
If $k$ is a field, $\mathrm{Sh}_{\Lambda^+}$ is an Artinian abelian category whose simple objects are the sheaves supported on a single region of the front diagram, with one-dimensional fibers.  For general $k$, any object of $\mathrm{Sh}_{\Lambda^+}$, or even of $\dgsh_{\Lambda^+}$, may be given a filtration whose  subquotients are cohomologically supported on a single region $R$.  In general there is no similar convenient class of ``generators'' for the smaller category $\dgsh_{\Lambda}$.

\end{remark}

To condense diagrams and shorten arguments, we would like to collapse the quasi-isomorphisms forced by the 
singular support condition.  For actual isomorphisms, this is possible, as follows: 

\begin{lemma}
\label{lem:simpler}
Let $\cT$ be any category, and suppose for every object $t \in \cT$ we fix some arrow $t \to \underline{t}$.  
Let $f: \cT \to \cC$ be a functor carrying the $t \to \underline{t}$ to isomorphisms.  
Then the functor $\underline{f}: \cT \to \cC$ defined by 
\begin{eqnarray*} \underline{f} (x) & := & f(\underline{x}) \\
\underline{f}(x \to y) & := & f(y \to \underline{y}) \circ f(x \to y) \circ f(x \to \underline{x})^{-1}
\end{eqnarray*} 
is naturally isomorphic to $f$, where the map $f(x) \to \underline{f}(x) = f(\underline{x})$ is the
isomorphism $f(x \to \underline{x})$.  Note that $\underline{f}$ takes all the arrows $t \to \underline{t}$ to identity maps. 
\end{lemma}

\begin{remark} Equivalently, the localization  $\cT[\{t \to \underline{t}\}^{-1}]$ is equivalent
to its full subcategory on the objects $\underline{t}$. \end{remark}

\begin{corollary} \label{cor:straight}
Let $\cT$ be a stratification of $\bR^2$ in which no one-dimensional stratum has vertical tangents.  
Then every object in $\Fun_{\Lambda^+}(\cT, k)$ is isomorphic to an object in which all downward arrows are 
identity morphisms. 
\end{corollary} 
\begin{proof}
In Lemma \ref{lem:simpler} above, take the arrows $t \to \underline{t}$ 
to be the arrows from a stratum to the two dimensional cell below it, which must be sent to isomorphisms
by any functor in
$\Fun_{\Lambda^+}(\cT, k)$.  In the resulting functor, all downward generization maps are sent to equalities, since
these are the composition of a downward generization map to a top dimensional cell, and the inverse of such a map. 
\end{proof}

\subsection{Legible objects}
\label{subsec:legible}

Let $M$ be a front surface, $\Phi \subset M$ a front diagram, and $\Lambda \subset T^{\infty,-}$ the associated Legendrian knot.  The combinatorial presentation of \S\ref{subsec:acm} makes use of a regular cell decomposition $\cS$ of $M$ refining $\Phi$.  By Theorem \ref{thm:comb}, the categories $\dgfun_{\Lambda}(\cS,k)$ and $\dgfun_{\Lambda^+}(\cS,k)$ are not sensitive to $\cS$ (and of course $\dgfun_{\Lambda}(\cS,k)$ is even a Hamiltonian invariant of $\Lambda$, which is our purpose in studying it and what we will prove in \S\ref{sec:invariance}.)  

In this section we describe a way to produce objects of $\dgfun_{\Lambda}(\cS,k)$, where $\cS$ is not foregrounded as much.  For sufficiently complicated $\Phi$, not all objects can be produced this way.  

\begin{definition}
\label{def:legible}
Let $\Phi \subset \bR^2$ be a front diagram in the plane.  A \emph{legible diagram} $F$ on $\Phi$ is the following data:
\begin{enumerate}
\item A chain complex $F^{\bullet}(R)$ of perfect $k$-modules for each region $R \subset \Phi$.
\item A chain map $F(s):F^{\bullet}(R_1) \to F^{\bullet}(R_2)$ for each arc $s$ separating the region $R_1$ below $s$ from $R_2$ above $s$
\end{enumerate}
subject to the following conditions:
\begin{enumerate}
\item[(3)] If $s_1$ and $s_2$ meet at a cusp, with $R_1$ outside and $R_2$ inside, so that $F(s_1):F^{\bullet}(R_1) \to F^{\bullet}(R_2)$ and $F(s_2):F^{\bullet}(R_2) \to F^{\bullet}(R_1)$, then the composition
\[
F^{\bullet}(R_1) \to F^{\bullet}(R_2) \to F^{\bullet}(R_1)
\]
is the {\em identity} of $F^{\bullet}(R_1)$.  
\item[(4)] If $N,S,E,W$ are the regions surrounding a crossing, (named for the cardinal directions--- since
as stated in Section \ref{sec:frontprojection}, a front diagram can have no vertical tangent line, $N$ and $S$, and therefore also $E$ and $W$, are unambiguous) then the square
\[
\xymatrix{
& F^{\bullet}(N)\\
F^{\bullet}(W) \ar[ur] & & F^{\bullet}(E) \ar[ul] &\\
& F^{\bullet}(S) \ar[ul] \ar[ur]
}
\]
commutes.
\end{enumerate}

We say that a legible diagram obeys the \emph{crossing condition} if furthermore:

\begin{enumerate}

\item[(5)] At each crossing, the total complex of
\[
F^{\bullet}(S) \to F^{\bullet}(W) \oplus F^{\bullet}(E) \to F^{\bullet}(N)
\]
is acyclic.
\end{enumerate}
\end{definition}

A stratum $w$ of $\cS$ is incident either with exactly two regions, or (only when $w$ is a crossing of $\Phi$) four.  In either case there is a unique region incident with and ``below'' $w$, which we denote by $\rho(w)$.  If $w_1$ is in the closure of $w_2$, then
\begin{enumerate}
\item either $\rho(w_1) = \rho(w_2)$,
\item or $\rho(w_1)$ is separated from $\rho(w_2)$ by an arc $s$, with $\rho(w_1)$ below and $\rho(w_2)$ above $s$,
\item or $w_1$ is a crossing of $\Phi$, and $w_2 = N$ is the region above it.  In this case let us denote the other three regions around $w_1$ by $S, W, E$ as in \ref{def:legible}(4)
\end{enumerate}
If $F^{\bullet}$ is a legible diagram on $\Phi$, we define
an object $F^{\bullet}_{\cS} \in \dgfun_{\Lambda}(\cS,k)$ by putting $F^{\bullet}_{\cS}(w) := F^{\bullet}(\rho(w))$, and, for $w_1$
in the closure of $w_2$,
\begin{equation}
\label{eq:serviustullius}
\begin{array}{rcll}
F^{\bullet}_{\cS}(w_1 \to w_2) &=& \mathrm{id} & \text{ in case (1)} \\
F^{\bullet}_{\cS}(w_1 \to w_2) &=& F^{\bullet}(s) & \text{ in case (2)}\\
F^{\bullet}_{\cS}(w_1 \to w_2) & = &\left[F^{\bullet}(S) \to F^{\bullet}(W) \to F^{\bullet}(N)\right]  \\
& = & \left[F^{\bullet}(S) \to F^{\bullet}(\,E\,) \to F^{\bullet}(N)\right]
 & \text{ in case (3)} 
 \end{array}
\end{equation}
Note the third case is well-defined by \ref{def:legible}(4). 
By Definition \ref{def:romulus}, $F^{\bullet}_{\cS}$ is an object of $\dgfun_{\Lambda^+}(\cS,k)$, and of $\dgfun_{\Lambda}(\cS,k)$ if $F^{\bullet}$ obeys the crossing condition. 

\begin{definition}
\label{def:legibleobject}
An object of $\dgfun_{\Lambda^+}(\cS,k)$ (resp.~$\dgfun_{\Lambda}(\cS,k)$)
is said to be legible if it is quasi-isomorphic
to an object defined by a legible diagram (resp.~a legible diagram obeying the
crossing condition).
\end{definition}

\begin{remark}
\label{rem:ancusmarcius}
A reasonable notion of morphism between two legible diagrams $F_1^{\bullet}$ and $F_2^{\bullet}$ is a family of maps $n_R:F_1^{\bullet}(R) \to F_2^{\bullet}(R)$ making all of the squares
\[
\xymatrix{
F_1^{\bullet}(R_2) \ar[r] &F_2^{\bullet}(R_2) \\
F_1^{\bullet}(R_1) \ar[u] \ar[r] & F_2^{\bullet}(R_1) \ar[u] 
}
\]
commute.  The above construction induces a functor from this category to the homotopy category of $\dgfun(\cS;k)$ (which could be promoted to a dg functor if necessary). We will not directly work with this category, preferring to consider legible diagrams as a source of objects for the category $\dgfun \cong \dgsh$ that we have already defined.  However, let us say that $F_1^{\bullet}$ and $F_2^{\bullet}$ are ``equivalent as legible diagrams'' if there is a map $\{n_R\}$ as above where each $n_R$ is a quasi-isomorphism.
\end{remark}

\subsubsection{Examples}

Here are typical legible objects on some local front diagrams (near a strand, a cusp, and a crossing)

\begin{center}
\includegraphics[scale = .3]{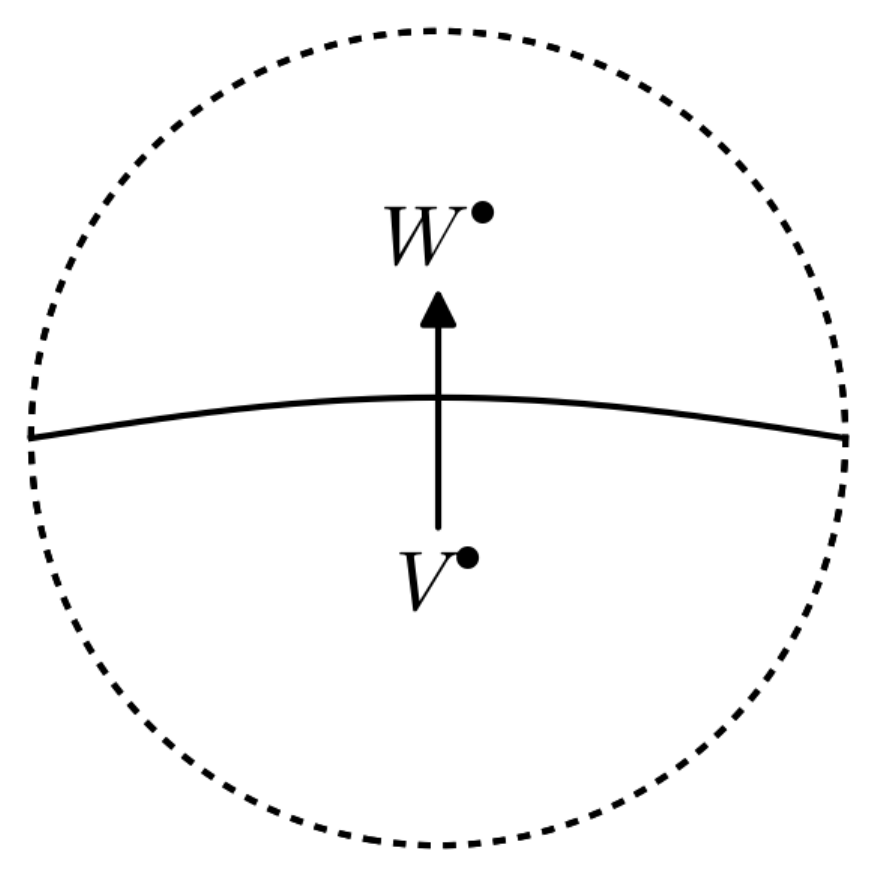}
\includegraphics[scale=.3]{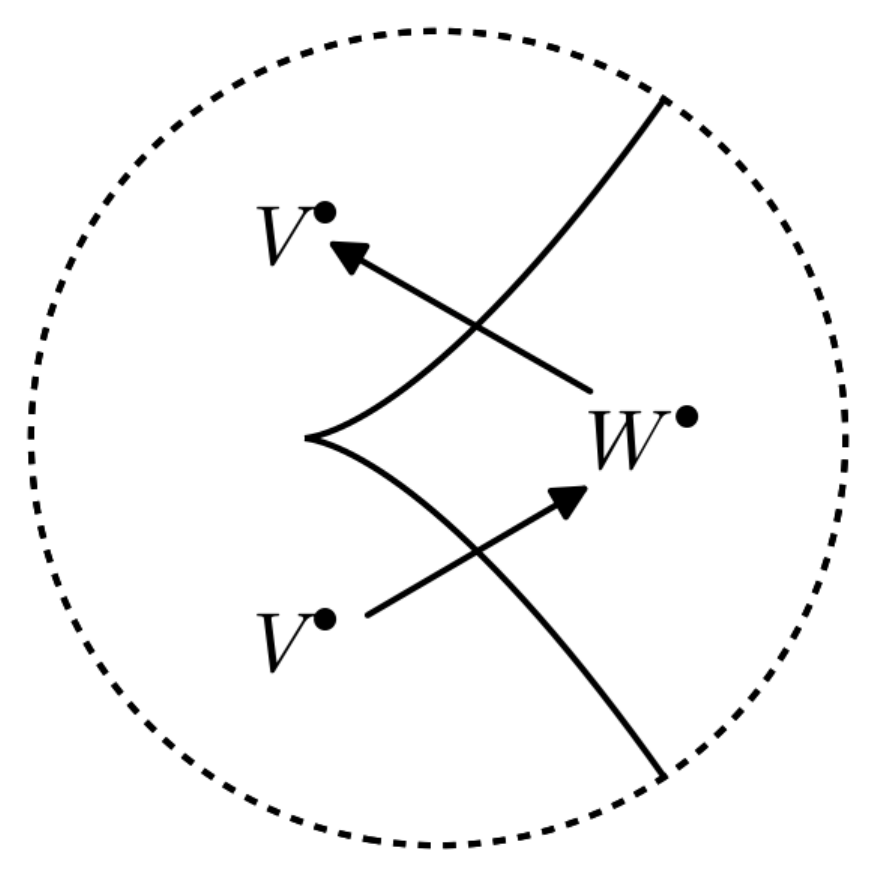}
\includegraphics[scale = .3]{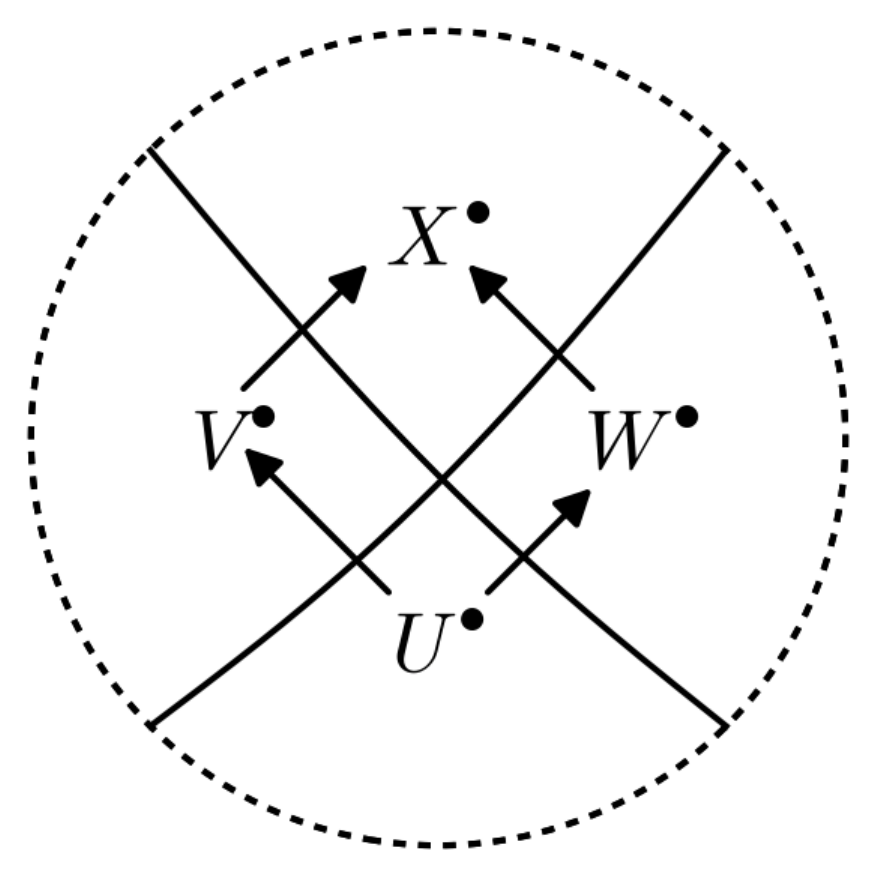}
\end{center}
The composition of the maps on the cusps is required to be the identity map of $V^\bullet$, and the square around the crossing must commute (or both commute and have acyclic total complex, if it is to obey the crossing condition). 
We use the word ``legible'' to convey that it is easier to draw the data of a legible diagram directly on the front diagram, than it is to draw the data of an object of $\Fun^{\bullet}_{\Lambda}(\cS,k)$.  For example the right-hand diagram above is more readable than \eqref{eq:tullushostilius}.

\subsubsection{Legible fronts}

Positive braids make a large class of front diagrams on which every object is legible.  Here by a positive braid we mean a front diagram in a convex open $M \subset \bR^2$ without cusps.  (For instance, $\Phi$ might be one of local front diagrams of Figure \ref{fig:r3}.)  On such a front diagram, the stratification by $\Phi$ is regular.  Moreover, the relation ``$R_1$ and $R_2$ are separated by an arc, with $R_1$ below and $R_2$ above'' generates a partial order on the set of regions, which we denote by $\cR(\Phi)$.
Then the data of a legible diagram (not necessarily obeying the crossing condition) is an object of $\dgfun(\cR(\Phi),k)$ in the sense of Section \ref{subsec:acm}.  The construction of Section \ref{subsec:legible} (see 
Equation \ref{eq:serviustullius}) describes precomposition with the map of posets $\rho:  \cS \to \cR(\Phi)$ taking $w$ to $\rho(w)$, the unique region incident with and below $w$.
Let us call this functor of precomposition $\rho^*$.

\begin{proposition}
\label{prop:braidlegible}
Suppose that $M \subset \bR^2$ is a convex open set and  $\Phi \subset M$ is a front diagram with no cusps, so that the stratification $\cS$ by $\Phi$ is a regular cell complex.  Then $\rho^*:\dgfun(\cR(\Phi),k) \to \dgfun_{\Lambda^+}(\cS,k)$ is a quasi-equivalence. 
In particular, every object of $\dgfun_{\Lambda^+}(\cS,k)$ is quasi-isomorphic to a  legible object.
\end{proposition}

\begin{proof}
We will use the right adjoint functor to $\rho^*$, which we denote
\[
\rho_*:\dgfun(\cS,k) \to \dgfun(\cR(\Phi),k)
\] 
and show that the adjunction map $\rho^* \rho_* F^\bullet \to F^\bullet$ is a quasi-isomorphism when $F^\bullet \in \dgfun_{\Lambda^+}(\cS,k)$.  For $R \in \cR(\Phi)$, define a subset (order ideal) $U_R \subset \cR(\Phi)$ by
\[
U_R = \{R' \mid R' \geq R\}
\]
The union of the cells in $\rho^{-1}(U_R)$ is an open subset of $M$ which we denote by $M_R$.  We have $M_{R'} \subset M_R$ if and only if $R' \geq R$.

Set $\cF = \Gamma_{\cS}^{-1}(F^{\bullet})$ as in Proposition \ref{prop:star};
so $F^{\bullet}(w)$ is quasi-isomorphic to $\Gamma(\text{star of $w$};\cF),$
where $\Gamma:\dgsh(M_R) \to \dgsh(\mathit{point})$ denotes the derived global sections functor.
Let us put
\[
\rho_* F^{\bullet} (R) = \Gamma(M_R,\cF)
\]
Then for $w \in \cS$, we have by definition
\[
\rho^* \rho_* F^{\bullet} (w) = \Gamma(M_{\rho(w)};\cF)
\]
As each $M_R$ is open, it contains the star of $w$ whenever $\rho(w)\geq R$, so we have a restriction of sections, i.e.~a map $\rho^* \rho_* F^{\bullet}(w) \to F^\bullet(w)$.  Now suppose that $F^{\bullet} \in \dgfun_{\Lambda^+}(M,k)$, or equivalently by Theorem \ref{thm:comb} that $\cF \in \dgsh_{\Lambda^+}(M,k)$, and let us show that the map
\begin{equation}
\label{eq:prove-this-is-iso}
\Gamma(M_{\rho(w)};\cF) \to \Gamma(\text{star of $w$};\cF)
\end{equation}
is a quasi-isomorphism for every $w$.  By Remark \ref{rem:cohompreserve}, we may find a filtration $\cF_1 \to \cF_2 \to \cdots \to \cF_m$ of $\cF$ each of whose graded pieces is supported on a single region, and belongs to $\dgsh_{\Lambda^+}$, so we may further reduce to the case where $m = 1$ and $\cF$ is supported on a single region $R$.  Such a sheaf is constant on the interior of $R$ and vanishes on the lower boundary of $R$, as in Example \ref{ex:31U}.  

The only nontrivial case is when $R$ is contained in $M_{\rho(w)}$ but not in the star of $w$.  In this case it is immediate that the codomain of \eqref{eq:prove-this-is-iso} vanishes, and we wish to conclude that the domain also vanishes.  As $R \neq \rho(w)$, $M_{\rho(w)}$ contains the lower boundary of $R$; let us denote it by $B$.  The domain of \eqref{eq:prove-this-is-iso} is identified with the relative cohomology of the pair $(\overline{R}\cap M_{\rho(w)}, B)$, which vanishes as both $\overline{R}\cap M_{\rho(w)}$ and $B$ are contractible.
\end{proof}

The Proposition applies to local front diagrams, in particular braids.  It is harder for $\Phi$ to be globally legible, even if we restrict to objects of $\dgfun_{\Lambda}(\cS,k)_0$.  For instance, there is no nonzero legible object for the following front diagram
of a Reidemeister-one-shifted unknot:
\begin{center}
\includegraphics[scale = .3]{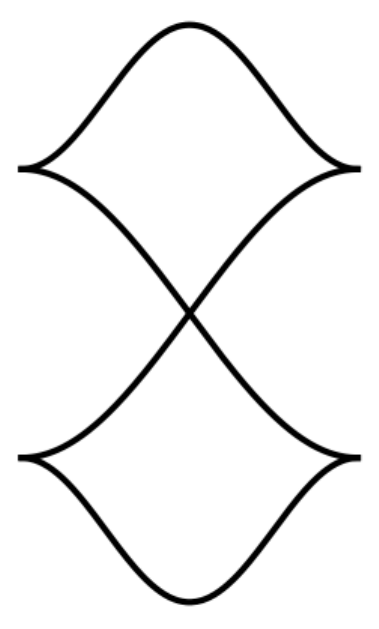}
\end{center}
On the other hand, there is a nonzero object of $\dgfun_{\Lambda}(\cS,k)_0 \cong \dgsh_{\Lambda}(\bR^2,k)_0$, for instance the sheaf of Example \ref{ex:R1unknotsheaf}.

We now describe some legible objects of $\dgsh(\Lambda, k)_0$ on some global examples.

\subsubsection{The unknot}
A  front diagram for a Legendrian unknot is shown in Figure \ref{fig:unknot}. 
\begin{figure}[H]
\includegraphics[scale=.2]{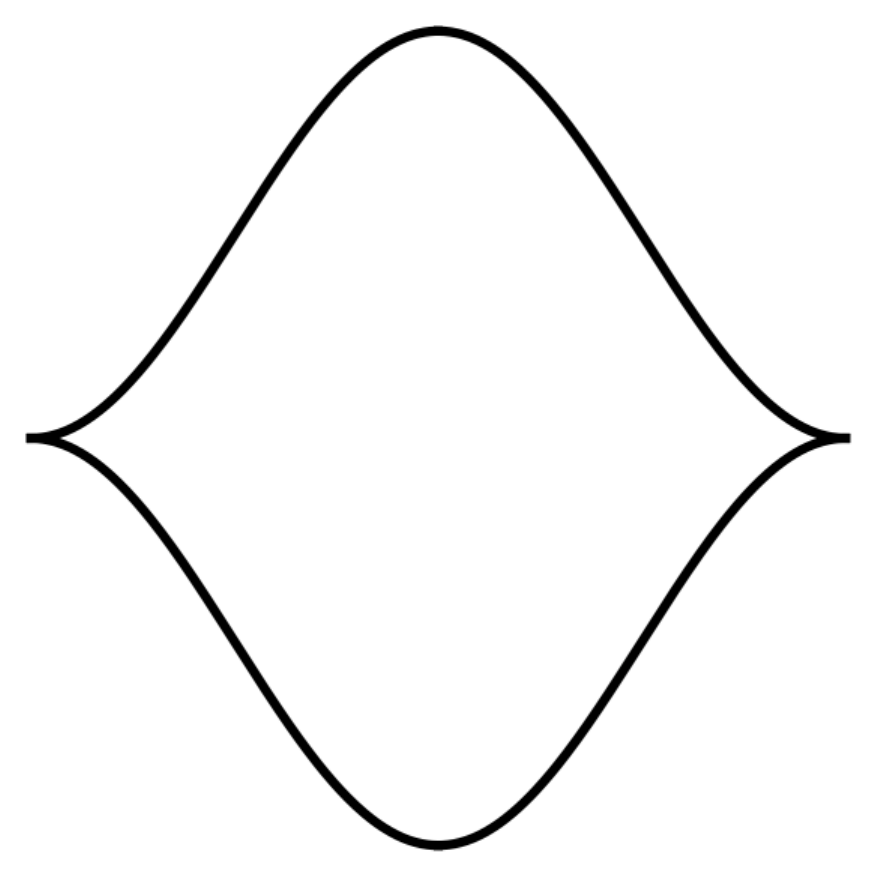}
\caption{The unknot}
\label{fig:unknot}
\end{figure}

Denote the bounded region by $R$ and the unbounded region by $O$, the upper arc by $u$ and the lower arc by $\ell$.  To give a legible diagram is to give two chain complexes $F^{\bullet}(R)$ and $F^{\bullet}(O)$, along with maps $F^{\bullet}(u):F^{\bullet}(R) \to F^{\bullet}(O)$ and $F^{\bullet}(\ell):F^{\bullet}(O) \to F^{\bullet}(R)$, subject to condition (3), the cusp condition, of Definition \ref{def:legible} (conditions (4) and (5) being vacuous for this front diagram).    The cusp condition for the left cusp and for the right cusp of Figure \ref{fig:unknot} make the same requirement: the composition $F^{\bullet}(u) \circ F^{\bullet}(\ell)$ must be equal to the identity map of $F^{\bullet}(O)$.  It follows that we can write $F^{\bullet}$ as a direct sum of the constant legible diagram (that takes the value $F^{\bullet}(O)$ on both $O$ and $R$) and a legible diagram $G^{\bullet}$ with $G^{\bullet}(O) = 0$.  

If we moreover require that $F^{\bullet}(O)$ is acyclic, i.e. that it represents an object of $\dgsh(\Lambda,k)_0$, then $F^{\bullet}$ and $G^{\bullet}$ are quasi-isomorphic --- i.e. up to quasi-isomorphism a legible diagram is just the data of a 
chain complex of $k$-modules assigned to this region.  It can be shown 
that all objects in this case are legible, and thus that 
the category attached to the unknot is quasi-equivalent to the derived category $k$-modules.

\subsubsection{The horizontal Hopf link}
\label{sec:hhl}

Figure \ref{fig:horizhopf} shows the horizontal Hopf link. 

\begin{figure}[H]
\includegraphics[scale = .3]{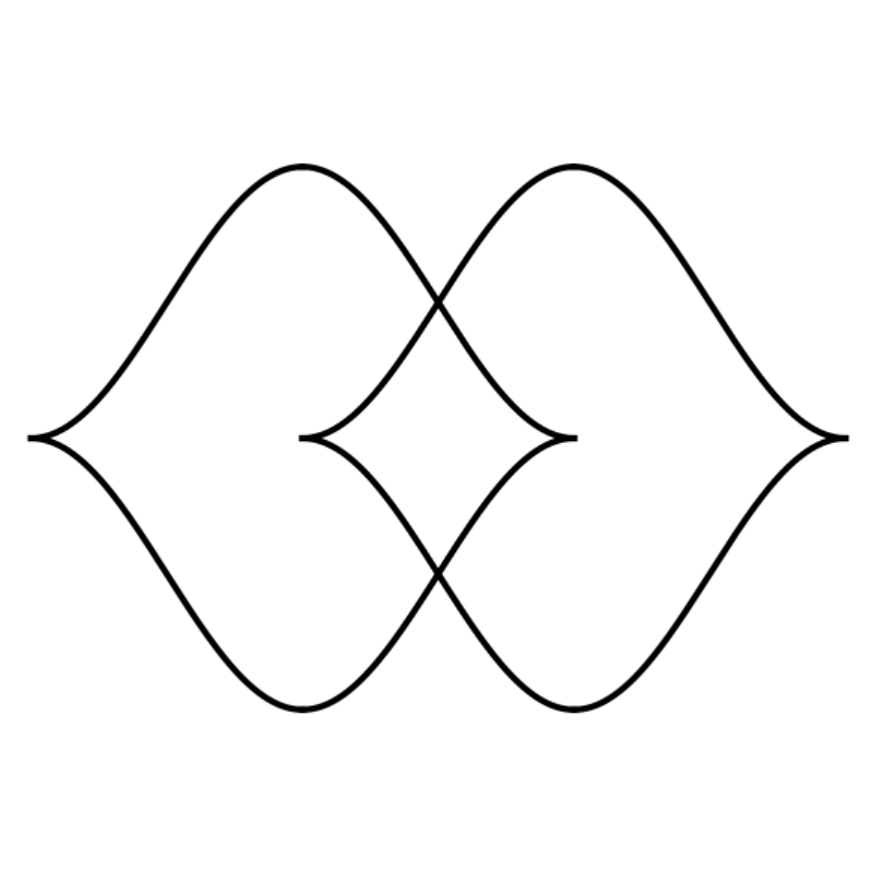}
\caption{The horizontal Hopf link}
\label{fig:horizhopf}
\end{figure}

There are three compact regions, call them ``left,'' ``middle,'' and ``right.''  
One family of legible objects on this front diagram has $F($left$) = F($right$) = k$ and $F($middle$) = k^2$, and the maps across arcs named in the diagram
\[
\xymatrix{
k & & k \\
 & \ar[ul]_p \ar[ur]^q k^2 & \\
k \ar@{=}[uu] \ar[ur]_i & & k \ar[ul]^j \ar@{=}[uu]
}
\]
The conditions at the cusps are $pi =  1$ and $qj = 1$; in particular $i$ and $j$ are injective, and $p$ and $q$ are surjective.  The crossing condition at the bottom is that $i$ and $j$ map $k \oplus k$ isomorphically onto $k^2$, and at the top similarly $p$ and $q$ map $k^2$ isomorphically onto $k \oplus k$.

We will later introduce the notion of the `microlocal rank' of an object; which depends on a Maslov potential.  It can be shown that the objects described above comprise all the objects of microlocal rank one with respect to the Maslov potential which takes the value zero on the bottom strands of the two component unknots.

\subsubsection{The vertical Hopf link}

Figure \ref{fig:verthopf}, shows the vertical Hopf link. 
\begin{figure}[H]
\includegraphics[scale = .3]{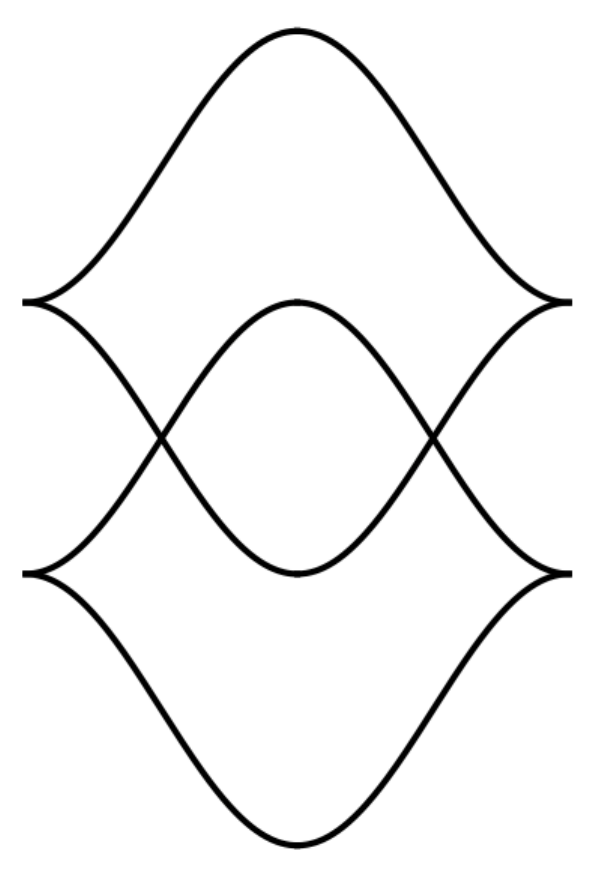} \qquad \includegraphics[scale = .3]{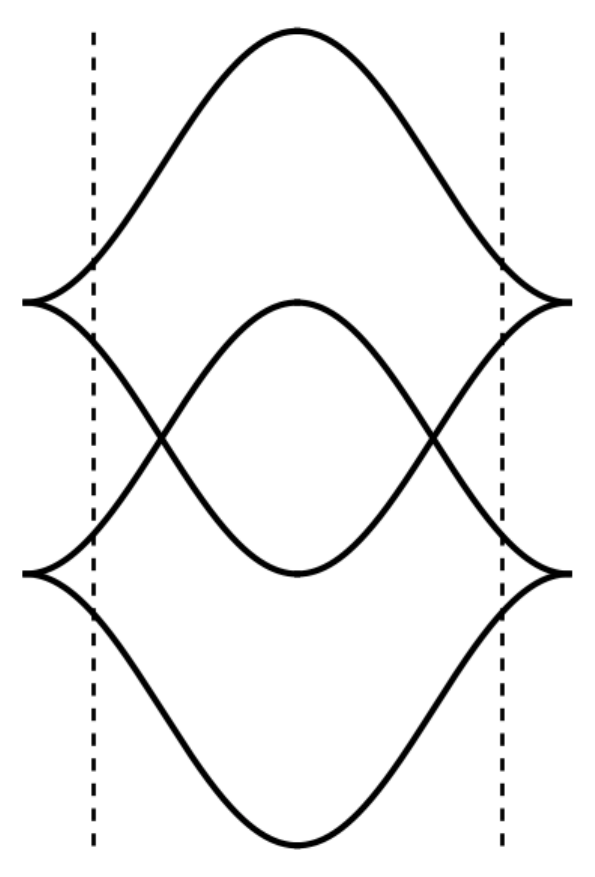}
\caption{The vertical Hopf link (left), with the cusps cut off (right)}
\label{fig:verthopf}
\end{figure}

Here we indicate the difficulties in finding globally legible objects.  We describe a family of sheaves on the vertical Hopf link that we suspect cannot be represented by a legible diagram.  
By Proposition \ref{prop:braidlegible}, if we restrict attention to the region between the two dashed lines and
view it as a front diagram, then all objects are legible.  However, their descriptions may take the following form: 
\[
\xymatrix{
& [k \stackrel{d_1}\longrightarrow k^3 \stackrel{d_2}\longrightarrow k] & \\
[0 \to k \stackrel{\alpha}{\to} k]\ar[ur]  & [0 \to k \stackrel{\beta}{\to} k]  \ar[u] & [0 \to k \stackrel{\gamma}{\to} k] \ar[ul] \\
& [0 \to 0 \to k] \ar[ul] \ar[u] \ar[ur]
}
\]
where $[U \to V \to W]$ denotes a three-term chain complex, $\alpha, \beta,\gamma \in k$, $d_2$ is the row vector $(\alpha,\beta,\gamma)$, and $d_1$ is in the kernel of $d_2$.  The maps 
are identities where possible and zero elsewhere, except for the three maps of the form $k \to k^3$ which are $(1,0,0)$ at the left, $(0,1,0)$ in the middle, and $(0,0,1)$ at the right.  
The crossing condition amounts to the statement that $\alpha$ and $\gamma$ are nonzero, and $d_1$ is neither of the form $(x,y,0)$ nor $(0,y,z)$.

Note that although the top region is labelled by a three-term chain complex, its cohomology is concentrated in only two degrees.  In fact, almost every object of this form is quasi-isomorphic to one where every chain complex is concentrated in two degrees.  The exception is when $\beta = 0$ and $d_1$ is the column vector $(1,1,-1)$.

The above description cannot extend to a globally legible diagram, because we would be forced to put the complex 
$[0 \to k \stackrel{\alpha}{\to} k]$ in the noncompact region, and then at the bottom right cusp, the identity map on this complex
would have to factor through $[0 \to k \stackrel{\alpha}{\to} k] \to [0\to 0 \to k]$.  In fact we suspect that this issue cannot be repaired by increasing the complexity of the legible diagram between the dashed lines (without changing the quasi-isomorphism type) --- i.e.~that not all objects of this
form admit legible diagrams.

\subsection{A Category of $k[x,y]$-Modules (Pixelation)}
\label{sec:pixelation}

A \emph{pixelation} of a front diagram $\Phi$ is a homeomorphism $\bR^2 \to \bR^2$ that carries the cusps, crossings, peaks and valleys of $\Phi$ to the lattice points of a square grid, and any part of a strand between two cusps, or a cusp and a crossing, or a crossing and a peak etc. to a line segment contained in the grid lines.

\begin{equation}
\label{fig:pixelation}
\includegraphics[scale=.3]{hhl.pdf}\quad \includegraphics[scale=.3]{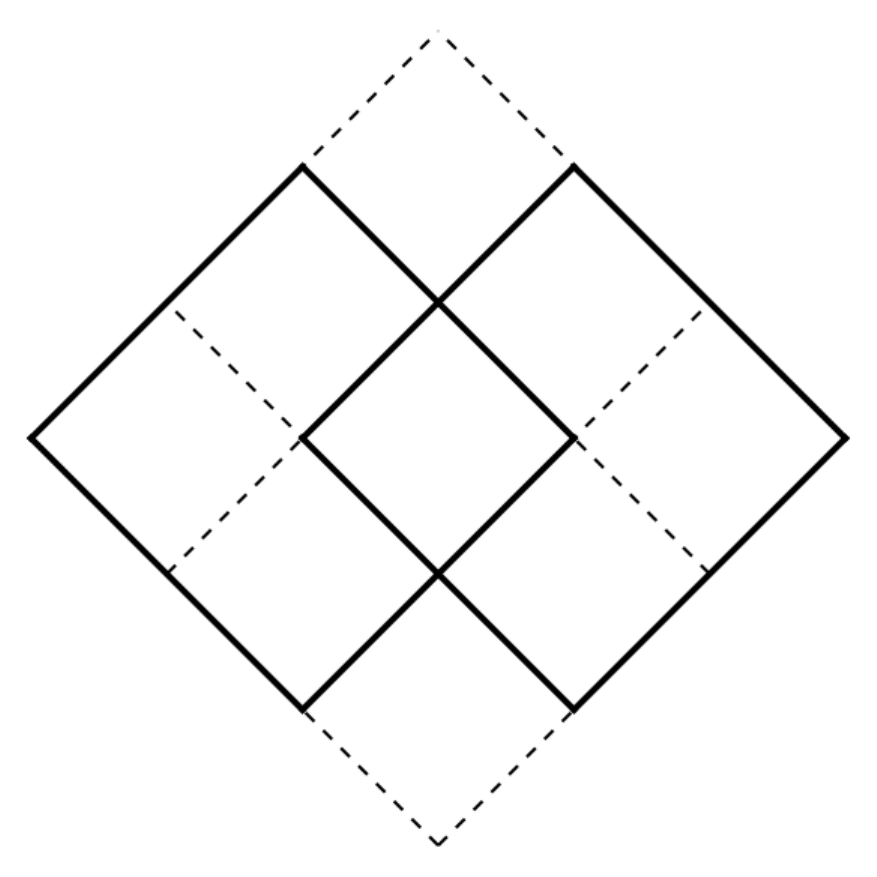}
\end{equation}

We can push $\cF \in \dgsh_{\Lambda}(\bR^2, k)$ forward along the pixelation to obtain a sheaf on $\bR^2$ constructible with respect to the square grid.  In fact, such a sheaf lands in the full subcategory spanned by objects that are constant on half-open grid squares (in the grid displayed, those that are closed on the top two sides and open on the bottom two sides), and that moreover have compact support and perfect fibers.  Let us denote this triangulated dg category by $\dgsh_{\mathrm{grid}}(\bR^2)$.

\begin{proposition}
The category $\dgsh_{\mathrm{grid}}(\bR^2)$ is equivalent to a full subcategory of the derived category of $\bZ^2$-graded modules over the ring $k[x,y]$, by a functor which takes the sheaf of rank one supported on the grid square with coordinates $(i,j)$ to the one-dimensional $k[x,y]$-module with bigrading $(i,j)$.
\end{proposition}

\begin{proof}
By \cite[Theorem 3.4]{FLTZ}, the derived category of bigraded $k[x,y]$-modules has a full embedding into $\dgsh(\bR^2)$, whose image is generated
(under sums and shifts and cones)
by sheaves of the form $\Theta(i,j):= f_! k$ where $f$ is the inclusion of an open set of the form
\[
\{(\alpha,\beta) \mid \alpha > i, \beta > j\}
\]
where $\alpha$ and $\beta$ denote the grid coordinates.  We need to check that each $F \in \dgsh_{\mathrm{grid}}(\bR^2)$ is in the image of this functor.  By induction on the number of grid squares in the support of $F$, it suffices to check that the sheaf with fiber $k$ and supported on a single grid square belongs to the image of $F$, but in fact this sheaf is quasi-isomorphic to the complex 
\[
\Theta(i+1,j+1) \to \Theta(i+1,j) \oplus \Theta(i,j+1) \to \Theta(i,j) 
\]
\end{proof}

\begin{example}
Suppose $F$ is one of the sheaves on the horizontal Hopf link of \S\ref{sec:hhl}.  Its pushforward under the pixelation displayed above is 
\[
\begin{array}{ccc}
& & 
\xymatrix{
& & 0\\
& k \ar[ur] & & k \ar[ul]\\
k \ar[ur]^= & & k^2 \ar[ur]^q \ar[ul]_p & & k \ar[ul]_=\\
& k \ar[ul]^= \ar[ur]_i & & k \ar[ul]^j \ar[ur]_= \\ 
& & 0 \ar[ul] \ar[ur]
}
\end{array} 
\]
The arrows pointing northwest (resp. northeast) assemble to an 8-by-8 nilpotent matrix acting with bidegree (0,1) (resp. (1,0)) on the bigraded vector space $k_{(0,1)} \oplus k_{(1,0)} \oplus k_{(0,2)} \oplus k^2_{(1,1)} \oplus k_{(2,0)} \oplus k_{(1,2)} \oplus k_{(2,1)}$.  These operators commute, defining an 8-dimensional bigraded $k[x,y]$-module.
\end{example}

In Example \ref{ex:34pixels}, we analyze the Legendrian (3,4)-torus knot using a pixelation.

\section{Invariance}
\label{sec:invariance}

To a Legendrian knot $\Lambda \subset \bR^3 \cong T^{\infty,-}\bR^2$ we have associated a
category $\dgsh_\Lambda(M)$.  In this section, we explain how this category is invariant under Legendrian isotopies of $\Lambda$, Theorem \ref{thm:4.1}.  This invariance theorem is a special case of the results of Guillermou-Kashiwara-Schapira \cite{GKS}, which we review in Sections \ref{subsec:convolution}--\ref{subsec:GKStheorem}.  In Section \ref{subsec:gksRmoves}, we give the explicit local equivalences for each of the Legendrian Reidemeister moves.

\begin{theorem}
\label{thm:4.1}
Let $M$ be a manifold, and let $\Omega \subset T^{\infty} M$ be an open subset of the cosphere bundle over $M$.  Suppose $\Lambda_1$ and $\Lambda_2$ are compact Legendrians in $\Omega$ that differ by a Legendrian isotopy of $\Omega$.  Then the categories $\dgsh_{\Lambda_1}(M)$ and $\dgsh_{\Lambda_2}(M)$ are equivalent.
\end{theorem}

\begin{proof}

Let us indicate here how this follows from the work of \cite{GKS}.

As $\Lambda_1$ and $\Lambda_2$ are compact, the isotopy between them can be chosen to be compactly supported (i.e. to leave fixed the complement of a compact set.)  Such an isotopy can be extended from $\Omega$ to $T^{\infty} M$, so to prove the Theorem it suffices to construct an equivalence between $\dgsh_{\Lambda_1}(M)$ and $\dgsh_{\Lambda_2}(M)$ out of a compactly supported isotopy of $T^{\infty} M$.   A compactly supported isotopy of $T^{\infty} M$ induces a ``homogeneous Hamiltonian isotopy'' of the complement the zero section in $T^* M$ with an appropriate support condition (``compact horizontal support''), and that this induces an equivalence is
Corollary 3.13 of \cite{GKS}, stated here as Theorem \ref{thm:where-it-happens}.  
\end{proof}

\begin{remark}
In the statement of the Theorem we have in mind the case where $\Lambda_1$ and $\Lambda_2$ are Legendrian submanifolds, but in fact they can be arbitrary compact subsets of $\Omega$.
\end{remark}

\begin{remark}
Here are two variants of the Theorem:
\begin{enumerate}
\item If $\Lambda_1$ and $\Lambda_2$ are noncompact, but isotopic by a compactly supported Legendrian isotopy, then $\dgsh_{\Lambda_1}(M) \cong \dgsh_{\Lambda_2}(M)$.  
\item Suppose $M$ is the interior of a manifold with boundary, and that one of the boundary components is distinguished.  Then let $\dgsh_{\Lambda}(M)_0 \subset \dgsh_{\Lambda}(M)$ denote the full subcategory of sheaves that vanish in a neighborhood of the distinguished boundary component.  As any compactly supported isotopy of $T^{\infty} M$ leaves this neighborhood invariant, $\dgsh_{\Lambda}(M)_0$ is also a Legendrian invariant.
\end{enumerate}
\end{remark}

\subsection{Convolution}
\label{subsec:convolution}
A sheaf on $N \times M$ determines a functor called convolution; such a sheaf is called a kernel for the functor.  For $F \in \dgsh(M)$ and $K \in \dgsh(N \times M)$, define the convolution $K \circ F \in \dgsh(N)$ by
\begin{equation}
\label{eq:shconv}
F \mapsto q_{1!}(K \otimes q_{2}^* F)
\end{equation}
where $q_1:N \times M \to N$ and $q_2:N \times M \to M$ are the natural projections.

We also define the convolutions of singular supports, as follows:
\begin{equation}
\label{eq:SSconv}
\SS(K) \circ \SS(F) := \left\{(y,\eta)  \in T^* N \mid \exists (x,\xi) \in \SS(F) \text{ such that } (y,\eta,x,-\xi) \in \SS(K) \right\}
\end{equation}

\begin{example}
If $K$ is the constant sheaf on the diagonal in $M \times M$, then $\SS(K)$ is the conormal of the diagonal in $M \times M$.  Convolution acts as the identity: $K \circ F \cong F$ and $\SS(K) \circ \SS(F) = \SS(F)$.
\end{example}

See Remark \ref{rem:h-image} for another example.  Under a technical condition, convolution of sheaves and of singular supports are compatible:

\begin{lemma}
\label{lem:tech-ass}
Suppose that $\SS(F)$ and $\SS(K)$ obey the following conditions:
\begin{enumerate}
\item $\left(N \times \mathrm{supp}(F)\right) \cap \mathrm{supp}(K)$ is proper over $N$.
\item For any $y \in N$ and $(x,\xi) \in \SS(F)$, if $\xi \neq 0$ then $(y,0,x,\xi) \notin \SS(K)$
\end{enumerate}
Then $\SS(K \circ F) \subset \SS(K) \circ \SS(F)$.
\end{lemma}

\begin{proof}
This is the special case of \cite[\S 1.6]{GKS} with $M_1 = N$, $M_2 = M$, and $M_3 =$ point.
\end{proof}

\subsection{Hamiltonian isotopies}

Let $M$ be a smooth manifold, and let $\dot{T}^* M$ denote the complement of the zero section in $T^*M$.  Let $I \subset \bR$ be an open interval containing $0$.  A ``homogeneous Hamiltonian isotopy'' of $\dot{T}^* M$ is a smooth map $t \mapsto h_t$ from $I$ to the symplectomorphism group of $\dot{T}^* M$, with the following properties:
\begin{enumerate}
\item $h_0$ is the identity
\item for each $t \in I$, the vector field $\frac{d}{dt} h_t$ is Hamiltonian
\item $h_t(x,a \xi) = a h_t(x,\xi)$ for all $a \in \bR_{>0}$.
\end{enumerate}

\begin{remark}
The third condition implies that each $h_t$ is exact, i.e. $h_t^* \alpha_M = \alpha_M$ for all $t$, where $\alpha_M$ is the standard primitive on $T^* M$.  An arbitrary isotopy with $h_t^* \alpha_M = \alpha_M$  is automatically Hamiltonian \cite[Corollary 9.19]{SalamonMcDuff}.
\end{remark}

Given such a $h$, consider its ``modified graph'' in $I \times \dot{T}^* M \times \dot{T}^* M$, defined by

\begin{equation}
\label{eq:modifiedgraph}
\Gamma(h) := \left\{\left(t,h_t\left(x,-\xi \right),\left(x,\xi\right)\right) \mid t \in I, (x,\xi) \in \dot{T}^* M\right\}
\end{equation}
Denote the restriction of $\Gamma(h)$ to $\{t\} \times \dot{T}^* M \times \dot{T}^* M$ by $L(h_t)$.  As each $h_t$ is a homogeneous symplectomorphism, $L(h_t)$ is a conic Lagrangian submanifold for each $t$.  The Hamiltonian condition is equivalent \cite[Lemma A.2]{GKS} to the existence of a conic Lagrangian lift of $\Gamma(h)$ to $T^* I \times \dot{T}^* M \times \dot{T}^* M$.  This lift is unique, and its union with the zero section is closed in $T^* (I \times M \times M)$.  We denote the
lift to $T^* I \times \dot{T}^* M \times \dot{T}^* M$ by $L(h)$, and the union of $L(h)$ with the zero section in $T^*(I \times M \times M)$ (resp. of $L(h_t)$ with the union of the zero section in $T^*(M \times M)$ by $\underline{L(h)}$ (resp. $\underline{L(h_t)}$).

Summarizing:
\begin{proposition}
\label{prop:characteristic-Lagrangian}
Let $h$ be a homogeneous  Hamiltonian isotopy of $\dot{T}^* M$.
There is a unique conic Lagrangian $L(h) \subset T^* I \times \dot{T}^*(M \times M)$ whose projection to $I \times \dot{T}^*(M \times M)$ is \eqref{eq:modifiedgraph}.  The union of $L(h)$ and the zero section is closed in $T^*(I \times M \times M)$.
\end{proposition}

\begin{remark}
\label{rem:h-image}
The definition \eqref{eq:SSconv} makes sense with $\SS(K)$ and $\SS(F)$ replaced by any closed conic sets in $T^* N \times T^* M$ and $T^* M$ respectively.  If $Z \subset T^* M$ contains the zero section then $\underline{L(h_t)} \circ Z$ is the union of the zero section with $h_t(Z \cap \dot{T}^* M)$.
\end{remark}

\subsection{The GKS theorem}
\label{subsec:GKStheorem}

\begin{theorem}[{\cite[Theorem 3.7]{GKS}}]
\label{thm:GKS}
Suppose $h$ and $L(h)$ are as in Proposition \ref{prop:characteristic-Lagrangian}.  Then there is a $K = K(h) \in \dgsh(I \times M \times M)$, unique up to isomorphism, with the following properties:
\begin{enumerate}
\item $K$ is locally bounded (has bounded restriction to any relatively compact open set),
\item the singular support of $K$, away from the zero section, is $L(h)$,
\item the restriction of $K$ to $\{0\} \times M \times M$ is the constant sheaf on the diagonal.
\end{enumerate}
\end{theorem}

We denote the restriction of $K(h)$ to  $\{t\} \times M \times M$ by $K(h_t)$, its singular support is $\Lambda(h_t)$.  Let us say that $h$ has \emph{compact horizontal support} if there is an open set $U \subset M$ with compact closure such that $h(t,x,\xi) = (x,\xi)$ for $x \notin U$.  (This is the condition 3.3 in \cite[p. 216]{GKS}.)  In that case each $K(h_t)$ is bounded (not just locally bounded) --- in fact, $K\vert_{J \times M \times M}$ is bounded for any relatively compact subinterval $J \subset I$ \cite[Remark 3.8]{GKS}.

\begin{theorem}
\label{thm:where-it-happens}
Suppose that $h$ has compact horizontal support, and let $K(h_t)$ be as above.  Then convolution by  $K(h_t)$ gives an equivalence
\[
K(h_t)\circ:\dgsh(M) \stackrel{\sim}{\to} \dgsh(M)
\]
If $\Lambda \subset T^{\infty}(M)$ is a closed subset, then $K(h_t)$ induces an equivalence
\[
K(h_t) \circ \dgsh_{\Lambda}(M) \stackrel{\sim}{\to} \dgsh_{h_t(\Lambda)}(M)
\]
\end{theorem}

\begin{proof}
The first assertion is \cite[Proposition 3.2(ii)]{GKS}.  The second assertion follows from Remark \ref{rem:h-image}, so long as the assumptions of Lemma \ref{lem:tech-ass} are satisfied with $K = K_t$ and any $F \in \dgsh_{\underline{\Lambda}}(M)$.  In fact as $h$ has compact horizontal support, $\mathrm{supp}(K(h_t))$ agrees with the diagonal of $M \times M$ outside of a compact set, so $\mathrm{supp}(K(h_t))$ is proper over $N$, and a fortiori (as $N \times \mathrm{supp}(F)$ is closed) so is $N \times \mathrm{supp}(F) \cap \mathrm{supp}(K)$.  This establishes (1).  As the codomain of $h_t$ is the complement of the zero section in $\dot{T}^*M$, the singular support of $K(h_t)$ contains no element of the form $(y,0,x,\xi)$ with $\xi \neq 0$.  This establishes (2) and completes the proof.\end{proof}

\begin{example}[Example 3.10 of \cite{GKS}]
\label{ex:geodesicflow}
Let $\bR^n$ be Euclidean $n$-space.  Normalized geodesic flow on $\dot{T}\bR^2\cong \dot{T}^*\bR^2$
is a homogeneous Hamiltonian isotopy defined by the time-independent Hamiltonian $|\xi|.$
Then $$L(h) = \{ (t,|v|; x,v ; x + t\hat{v}, v ) \} \subset T^*I \times \dot{T}^*\bR^2 \times \dot{T}^*\bR^2,$$
where $\hat{v} = v/|v|.$  There is a unique sheaf $K$ creating the distinguished triangle $$j_!k_{\{t>|y-x|\}}[n]\to K \to i_*k_{\{t < |y-x|\}},$$
where $i$ and $j$ are the inclusions of the indicated open sets.  Then $K$ is the GKS kernel for geodesic (Reeb) flow.
\end{example}

\subsection{Reidemeister moves}
\label{subsec:gksRmoves}

Let $\Phi_1$ and $\Phi_2$ be the local diagrams of Figure \ref{fig:r1}, \ref{fig:r2}, or \ref{fig:r3}.  Let $\Lambda_1$ and $\Lambda_2$ be the corresponding Legendrians in $T^{\infty,-}$.  By Theorem \ref{thm:4.1}, the categories $\dgsh_{\Lambda_1}$ and $\dgsh_{\Lambda_2}$ are equivalent.  
That equivalence is given by a kernel, i.e. a sheaf on $\bR^2\times \bR^2$, which is uniquely characterized by the GKS
theorem.  In fact, up to planar isotopy, each Reidemeister move can by place into the framework of
Example \ref{ex:geodesicflow}.  As a demonstration in the most difficult case, the Reidemeister-1 related red and blue curves in the
diagram below
\begin{center}
\includegraphics[scale=.4]{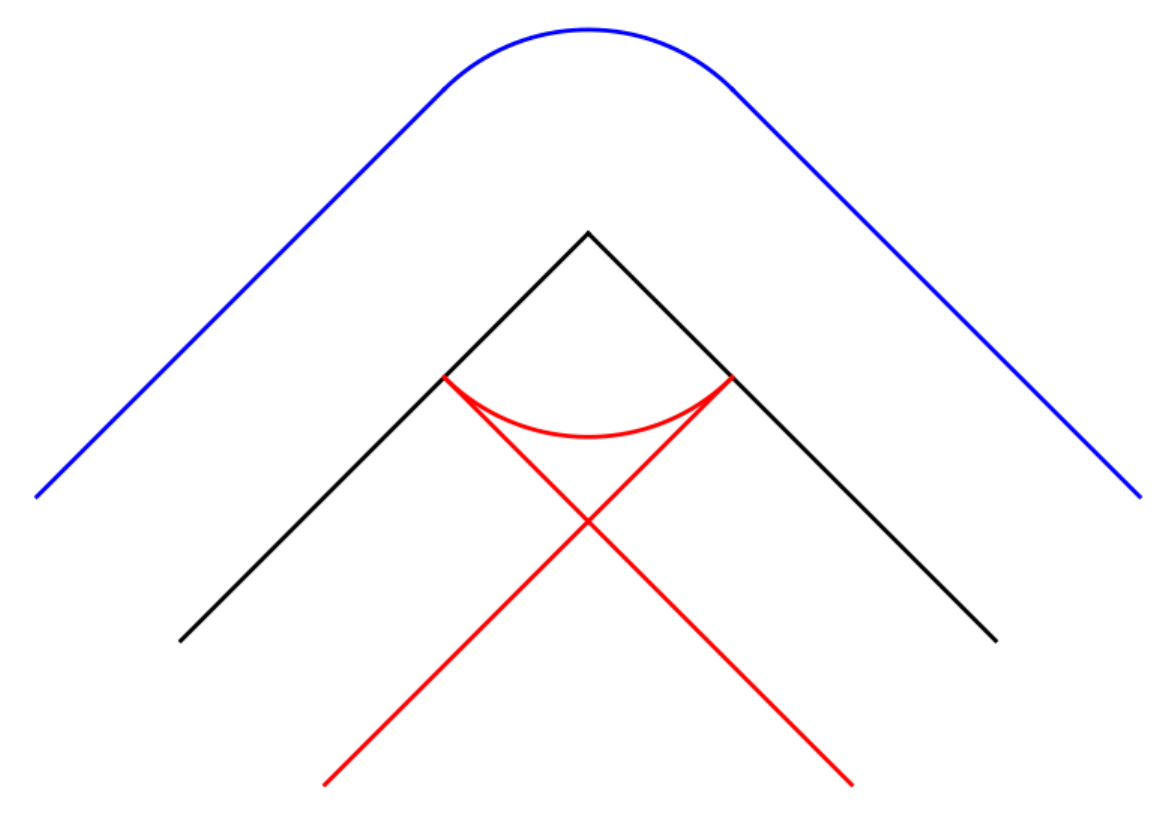}
\end{center}
are fronts of positive (red) and negative (blue) Reeb flows from the inward conormal of the quadrant surrounded by the black front.
With the kernel $K$ in hand, invariance under Reidemeister moves can be derived explicitly.
In this section we describe what happens at the level of objects, using the language of legible diagrams of \S\ref{subsec:legible}.  In each of the examples below,
all objects of the local categories are legible, either from applying 
Proposition \ref{prop:braidlegible} to one side or from a direct argument.

\subsubsection{Reidemeister 1}

Typical legible objects on $\Phi_1$ and $\Phi_2$ are displayed:

\begin{figure}[H]
\includegraphics[scale =.3]{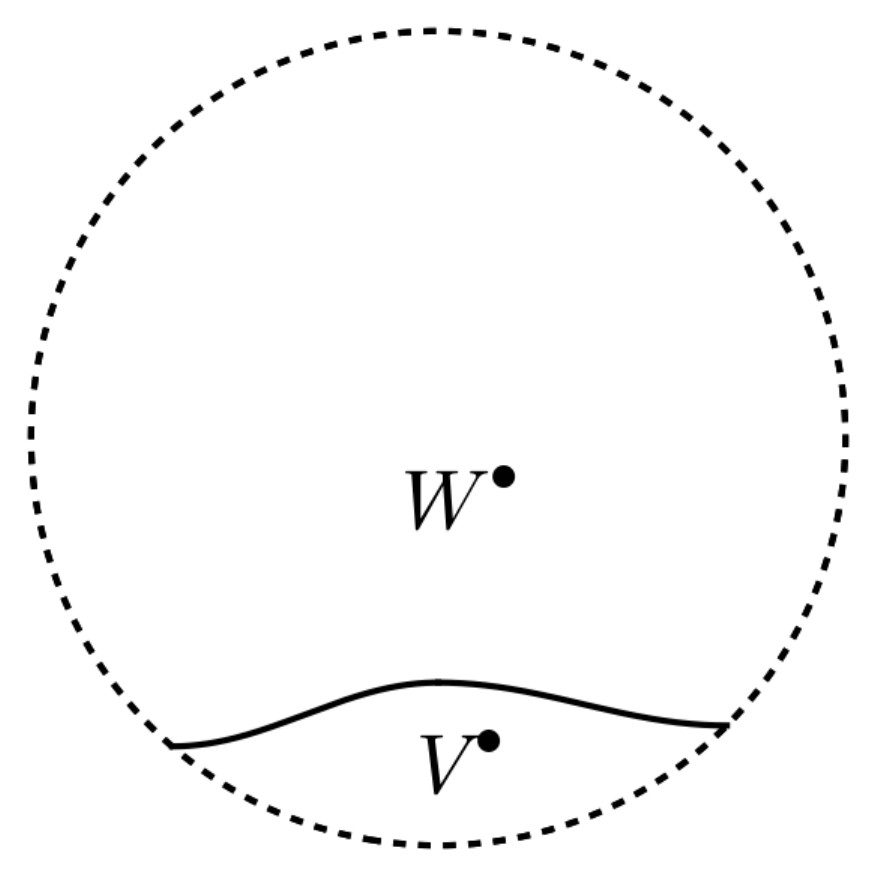} \quad \includegraphics[scale=.3]{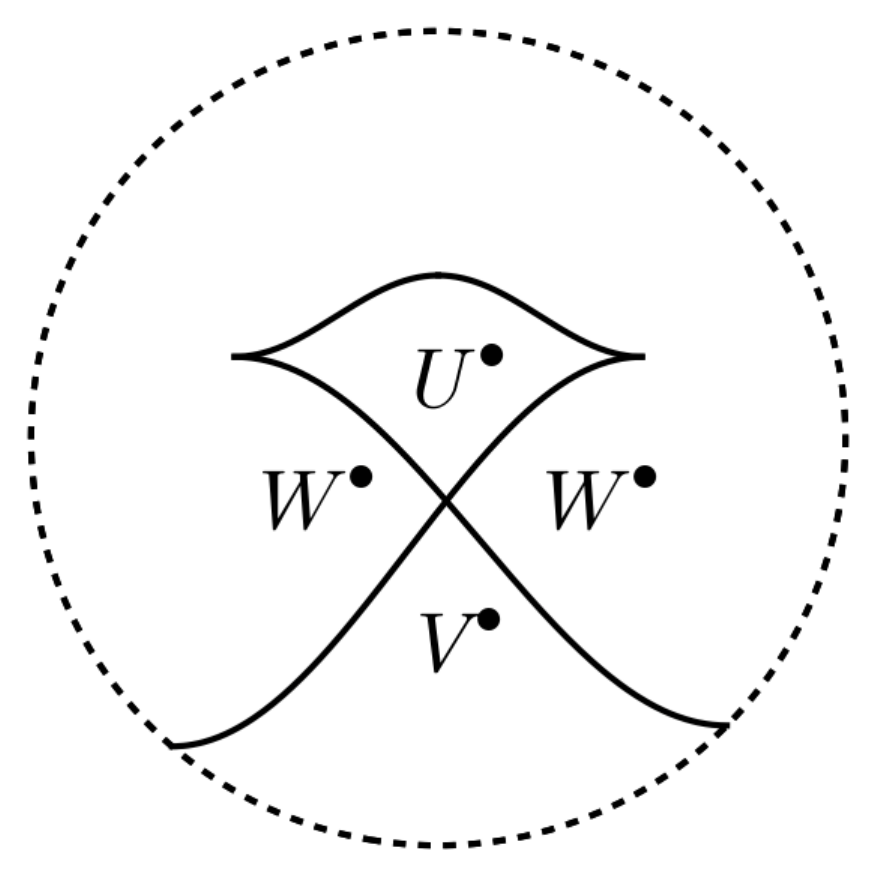}
\caption{Local description of the Reidemeister 1 move.}\label{fig:r1equiv}
\end{figure}
We have not included the maps in the diagram, let us name them
\[
\begin{array}{cc}
\xymatrix{
\\
\\
W^{\bullet} & \\
\ar[u]_f
V^{\bullet}
} &
\xymatrix{
& W^{\bullet} \\
& U^{\bullet} \ar[u]_{p} \\
W^{\bullet} \ar[ur]^{g_1} & & W^{\bullet} \ar[ul]_{g_2} \\
& V^{\bullet} \ar[ul]^{f_1} \ar[ur]_{f_2} 
}
\end{array}
\]
Note that, as $pg_1 = pg_2 = 1_{W^{\bullet}}$, we have $f_1 = f_2$.  Thus, to go from $\Phi_2$ to $\Phi_1$, we simply take $f = f_1 = f_2$.  

To go in the other direction, we assume that $f$ is an injective map of chain complexes.  (This is no restriction: replacing $f$ with the inclusion of $V^{\bullet}$ into the mapping cylinder of $f$ gives an equivalent legible diagram, in the sense of Remark \ref{rem:ancusmarcius}).  We take $V^{\bullet} = V^{\bullet}$,  $W^{\bullet} = W^{\bullet}$, $f_1 = f$ and $f_2 = f$,
\[
U^{\bullet} = \mathrm{coker}(V^{\bullet} \xrightarrow{(f,-f)} W^{\bullet} \oplus W^{\bullet})
\]
and $p:U^{\bullet} \to W^{\bullet}$ induced by the addition map $W^{\bullet} \oplus W^\bullet \to W^{\bullet}$.

\subsubsection{Reidemeister 2}  Typical legible object on $\Phi_1$ and $\Phi_2$ are displayed:

\begin{figure}[H]
\includegraphics[scale =.3]{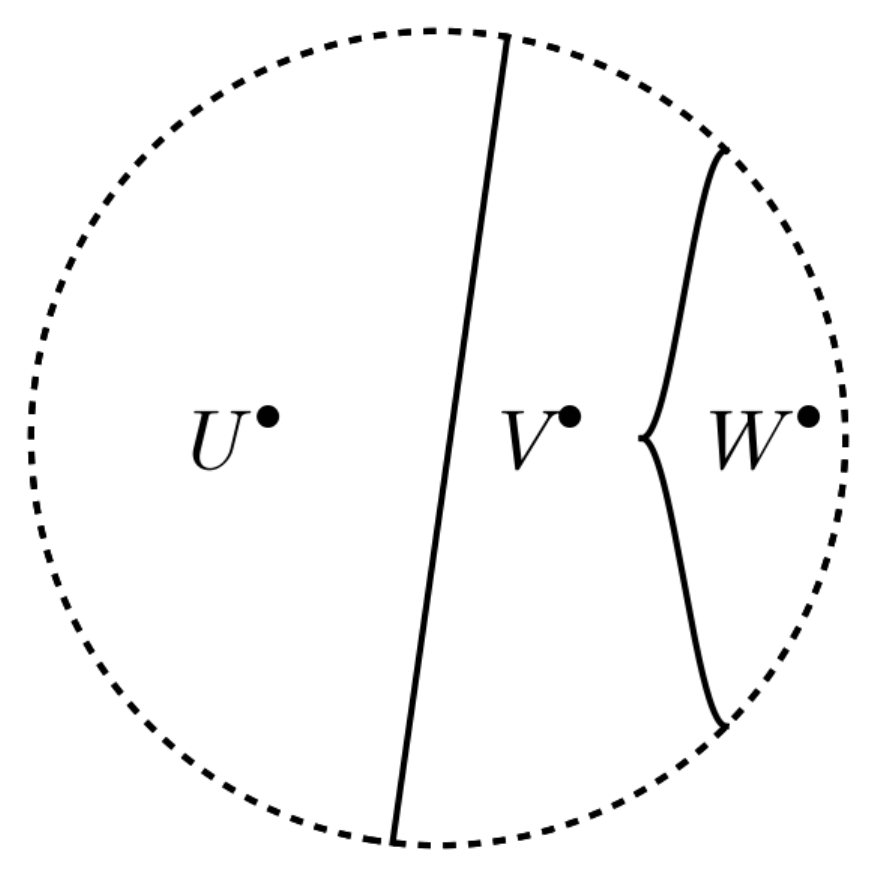} \quad \includegraphics[scale=.3]{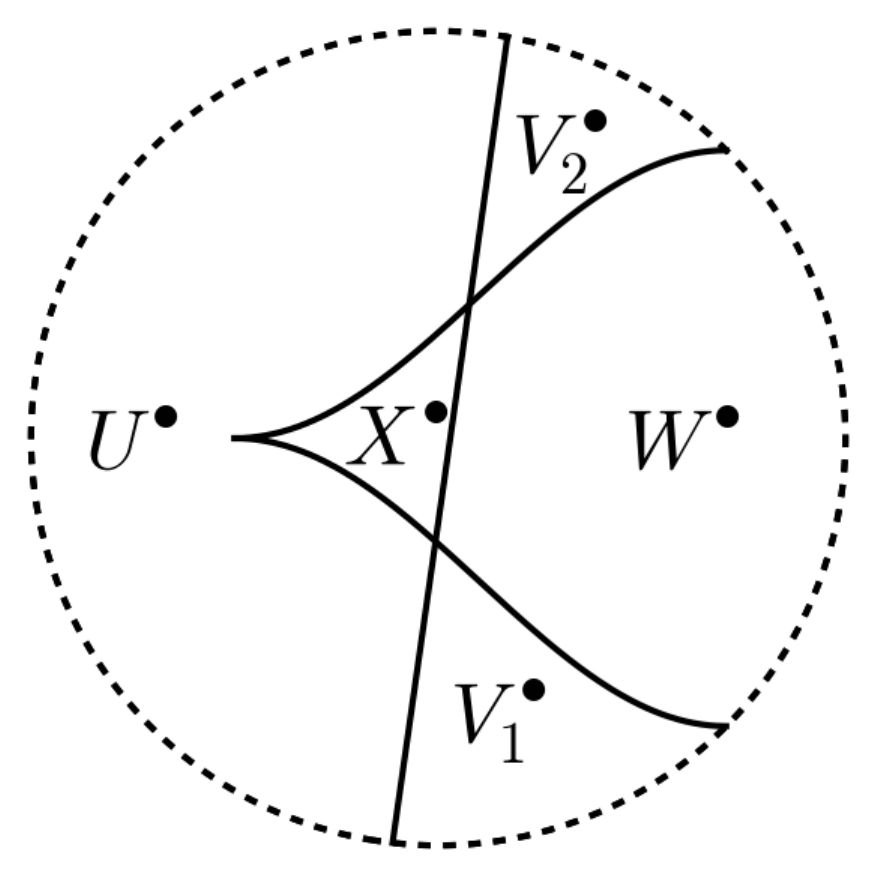}
\caption{Local description of the Reidemeister 2 move}\label{fig:r2equiv}
\end{figure}

We label the maps in the two diagrams as follows:
\[
\begin{array}{ccc}
\xymatrix{
\\
U^{\bullet} & \ar[l]_{f} \ar@/_/[r]_{g} V^{\bullet} & \ar@/_/[l]_{p} W^{\bullet}
}&
&
\xymatrix{
& & V_2^{\bullet} \ar@/_1pc/[dll]_{f_2} \\
U^{\bullet} \ar@/_/[r]_{\iota} & X^{\bullet} \ar@/_/[l]_{\pi} & & W^{\bullet}\ar[ll]_{\alpha}  \ar[ul]_{p_2} \\
& & V_1^{\bullet} \ar@/^1pc/[ull]^{f_1} \ar[ur]_{g_1}
}
\end{array}
\]
To go from $\Phi_1$ to $\Phi_2$, put $U^{\bullet} = U^{\bullet}$, $V_1^{\bullet} = V_2^{\bullet} = V^{\bullet}$, $W^{\bullet} = W^{\bullet}$, $p_2 = p$, $g_1 = g$, $f_1 = f_2 = f$, and 
\begin{itemize}
\item[-] $X^{\bullet} = \mathrm{coker}(V^{\bullet} \stackrel{(f,-g)}{\longrightarrow} U^{\bullet} \oplus W^{\bullet})$. Note that since $pg =1$, $(f,-g)$ is injective as a map of chain complexes.  It follows that the total complex of $V^{\bullet} \to U^{\bullet} \oplus W^{\bullet} \to X^{\bullet}$ is acyclic.

\item[-] The map $\alpha$ is the composition $W^{\bullet} \stackrel{(0,1)}{\longrightarrow} U^{\bullet} \oplus W^{\bullet} \to X^{\bullet}$.
\item[-] The map $\iota$ is the composition $U^{\bullet} \stackrel{(1,0)}{\longrightarrow} U^{\bullet} \oplus W^{\bullet} \to X^{\bullet}$
\item[-] The map $\pi$ is induced by
$
\xymatrix{
U^{\bullet} \oplus W^{\bullet} \ar[rr]^{\quad \, ( 1_U + f \circ p)}  & & U^{\bullet}
}
$
\end{itemize}

It is a slightly nontrivial lemma in homological algebra that any object on $\Phi_2$ come from $\Phi_1$ in this manner, let us indicate why.  It is easy to see that the following sorts of objects comes from $\Phi_1$:
\begin{itemize}
\item[(1)] Objects where $U^{\bullet}$, $V_1^{\bullet}$, and $V_2^{\bullet}$ are all acyclic, in which case the map $\alpha$ is a quasi-isomorphism.
\item[(2)] Objects where $U^{\bullet} \leftrightarrows X^{\bullet}$ are inverse isomorphisms, in which case the maps $V_1^{\bullet} \to W^{\bullet}$ and $W^{\bullet} \to V_2^{\bullet}$ must be quasi-isomorphisms.
\end{itemize}
We may conclude that every other object arises in this say so long as every legible diagram on $\Phi_2$ fits into an exact triangle $[F' \to F \to F'' \to ]$ with $F''$ of type (1) and $F'$ of type (2).  If the chain complexes and maps of $F = \{V_1^{\bullet}, W^{\bullet},\ldots\}$ are named as in the right-hand diagram above, then for $F''$ we may take the object
\[
\xymatrix{
& & 0 \ar@/_1pc/[dll] \\
0 \ar@/_/[r]  & \mathrm{im}(\iota) \ar@/_/[l]  & & \mathrm{Cone}(V_1^{\bullet} \to W^{\bullet})\ar[ll]_{\iota \circ \pi \circ \alpha}  \ar[ul] \\
& & 0 \ar@/^1pc/[ull] \ar[ur]
}
\]
There is evidently a map $F \to F''$, and the cone on this map is of type (1).

\subsubsection{Reidemeister 3}

Let $\Phi_1$ and $\Phi_2$ be the local diagrams of Figure \ref{fig:r3}.
Recall from \ref{prop:braidlegible} that every object of the associated categories is legible,
so we consider legible diagrams only.
Let $D_4$ be the graph with four nodes $e,f_1,f_2,f_3$, with an edge between $e$ and $f_i$ for $i = 1,2,3$.  Let $D_4^+$ denote this quiver oriented so that $e$ is a source, and let $D_4^-$ denote this quiver oriented so that $e$ is a sink.  Let $\mathrm{Rep}(D_4^+)$ and $\mathrm{Rep}(D_4^-)$ denote the dg derived category of representations of each of these quivers.  There are functors
$\dgfun(\Phi_1,k) \to \mathrm{Rep}(D_4^+)$ and $\dgfun(\Phi_2,k) \to \mathrm{Rep}(D_4^-)$, which take the
legible diagrams of Figure \ref{fig:r3equiv}
\begin{figure}[H]
\includegraphics[scale = .3]{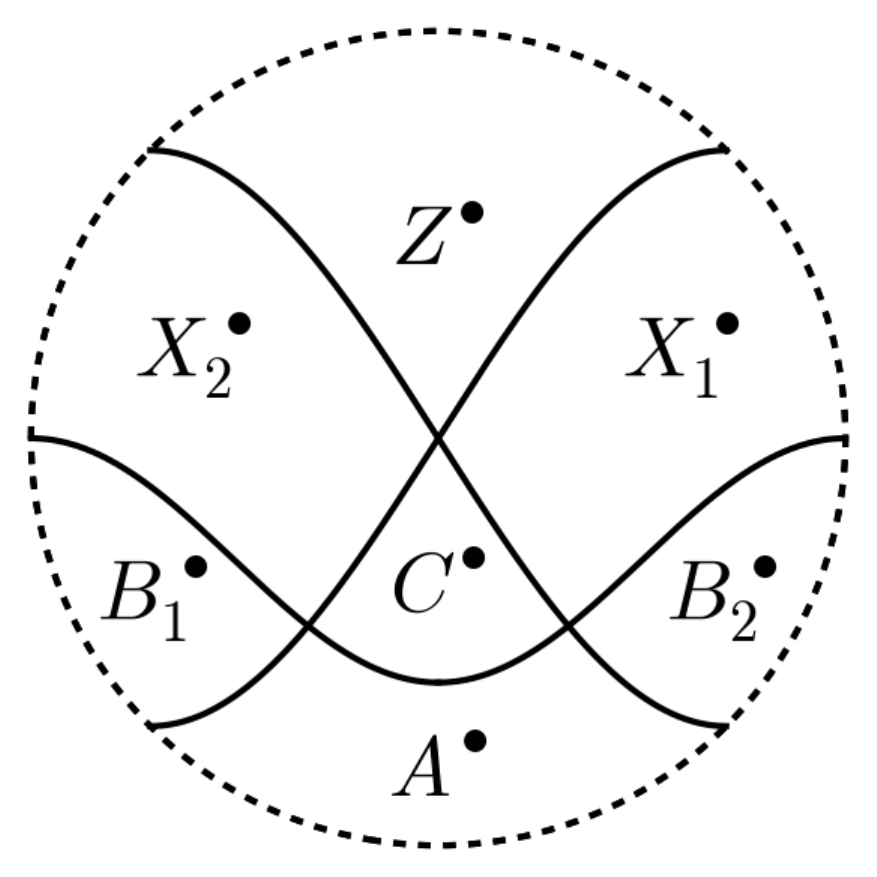} \quad \includegraphics[scale = .3]{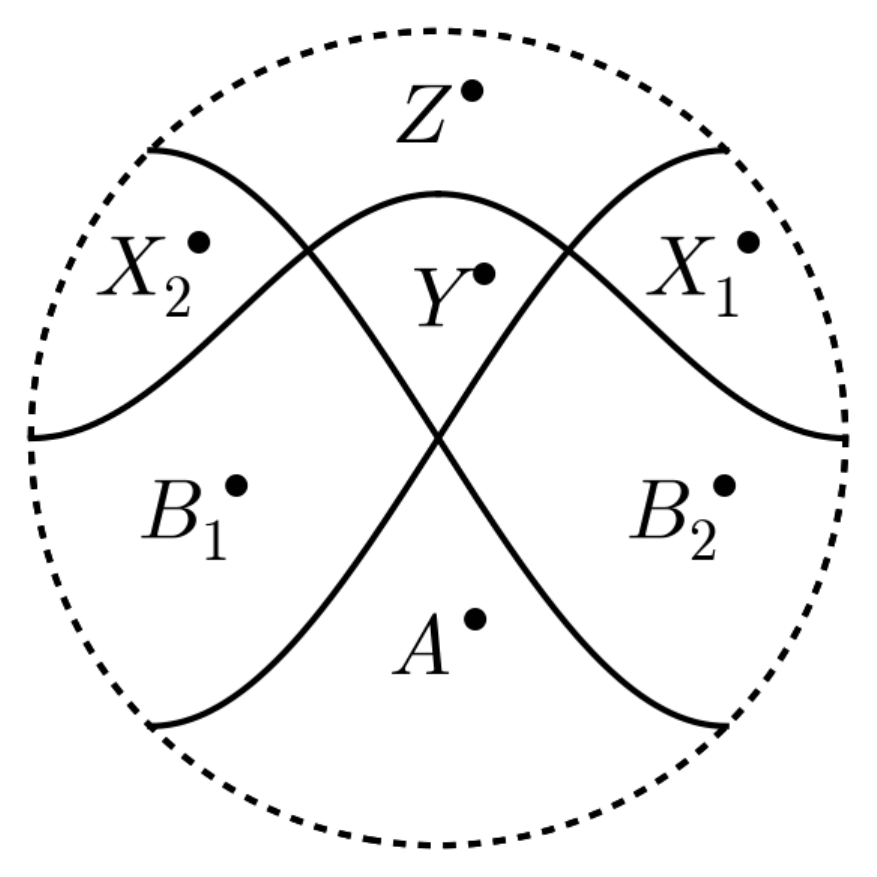}
\caption{Local categories for the Reidemeister 3 move}
\label{fig:r3equiv}
\end{figure}
\noindent to the following objects:
\[
\begin{array}{ccc}
\xymatrix{
B_1^{\bullet} & C^{\bullet} & B_2^{\bullet} \\
& A^{\bullet} \ar[ul] \ar[u] \ar[ur]
}
&
\quad
&
\xymatrix{
 & Z^{\bullet} &  \\
X_2^{\bullet}\ar[ur] & Y^{\bullet}  \ar[u] & X_1^{\bullet} \ar[ul]
}
\end{array}
\]
These functors are quasi-equivalences; for instance, the quasi-inverse $\mathrm{Rep}(D_4^-) \to \dgfun(\Phi_2,k)$ takes 
\[
\begin{array}{c}
B_1^{\bullet} = Y^{\bullet} \times_{Z^{\bullet}} X_2^{\bullet} \\
B_2^{\bullet} = X_1^{\bullet} \times_{Z^{\bullet}} Y^{\bullet} \\
A^{\bullet} = B_1^{\bullet} \times_{Y^{\bullet}} B_2^{\bullet}
\end{array}
\]
Let $W$ denote the Weyl group of type $D_4$, and $s_e,s_{f_1},s_{f_2},s_{f_3}$ its generators corresponding to the nodes $e,f_1,f_2,f_3$.  There are two standard equivalences \cite{BGP} between $\mathrm{Rep}(D_4^+)$ and $\mathrm{Rep}(D_4^-)$, naturally labeled by the Weyl group elements $s_e$ and $s_{f_1} s_{f_2} s_{f_3}$.  The Reidemeister 3 equivalence is given by the composition
\[
\xymatrix{
\dgfun(\Phi_1,k) \ar[r] &  \mathrm{Rep}(D_4^+) \ar[rrrr]^{s_e s_{f_1} s_{f_2} s_{f_3} s_es_{f_1} s_{f_2} s_{f_3} s_e} & &   & & \mathrm{Rep}(D_4^-) & \dgfun(\Phi_2,k) \ar[l]
}
\]

\section{Further constructions}
\label{sec:mishmash}

In this section we explore some additional structures on our category.  Given a knot $\Lambda\subset \bR^3\cong T^{\infty,-}\bR^2$
equipped with a Maslov
potential, the \emph{microlocal monodromy} constructs from a sheaf in $\dgsh_\Lambda(\bR^2)$ a local system on $\Lambda$ itself.
An object whose microlocal monodromy is a rank-$r$ local system in cohomological degree zero is said to have microlocal rank-$r$.
Next we define the notion of a \emph{ruling filtration} of an object $F$, which is a presentation of $F$ as an iterated cone over basic
sheaves called ``eyes.'' The notion of a ruling filtration relates to the combinatorial construction called a \emph{ruling} in the Legendrian knot literature,
and allows for the construction of an abstract Seifert surface for the knot $\Lambda.$ 
For objects with Maslov potential taking only two values, the microlocal rank-$r$ objects form a moduli space which is representable
by an Artin stack of finite type.

\subsection{Microlocal monodromy}
\label{sec:mmon}
Roughly speaking, a constructible sheaf on a manifold determines a ``local system'' on the smooth part of its singular support.  To make this precise in general is somewhat intricate, for example the monodromies of the ``local system'' may act by homological shifts.  This sort of behavior is very concrete in the setting of this paper, i.e. when the sheaf has singular support in a knot, and we describe in in this section.  A much more general version of this story is essentially the theory of the $\mu\mathit{hom}$-functor of \cite[Ch. IV]{KS} --- we will not invoke $\mu\mathit{hom}$ directly here but see Remark \ref{rem:quote-really-unquote}.

Let $\Lambda$ be a Legendrian knot in general position, and fix a
regular refinement of the stratification of its front diagram.  Let
$\cS$ be the resulting stratification.  Fix a parameterization $S^1 \to \Lambda$ of the knot such that 
the preimage of a cusp is a closed interval of non-zero length.  Pulling
back the stratification into arcs, cusps, and crossings gives a stratification of $S^1$
which is a regular cell complex.
We denote the poset of this stratification by $\Delta$.  An arc in $\cS$ has a unique preimage in $\Delta$,
a crossing has two points as preimages, and a cusp has two zero dimensional strata and a one-dimensional stratum (altogether, a closed interval)
as preimages.  If $c \in \cS$ is a cusp, we write these preimages and the maps relating them in 
$\Delta$ as $c_{\preccurlyeq} \to c_\prec \leftarrow c_{\curlyeqprec}$.

\begin{definition} 
Given $F \in  \dgfun(\cS, k)$, its unnormalized microlocal monodromy $\mmon'(F)$ is the assignment 
of a chain complex of $k$-modules to every object in $\Delta$, degree zero chain maps to all 
arrows in $\Delta$ except at cusps, where the arrow $c_{\preccurlyeq} \to c_\prec$ is assigned a 
degree one chain map (i.e. it shifts the degree by one, but commutes with the differential).  In detail,
$\mmon'(F)$ is defined 
as follows:
\begin{itemize}
\item Let $a \in \cS$ be an arc, and denote also by $a$ its preimage in $\Delta$. 
There's a region
$N$ above $a$ and a map $a \to N$ in $\cS$.  We take 
$$\mmon'(F)(a) = \mathrm{Cone}(F(a\to N))$$
\item Let $c \in \cS$ be a crossing; as in \eqref{eq:tullushostilius}
 we use the following notation for the star of $c$:
\[
\xymatrix{
& & N\\
& nw \ar[ur] \ar[dl] & & ne \ar[ul] \ar[dr] \\
W  & & c \ar[ul] \ar[ur] \ar[dl] \ar[dr] \ar[uu] \ar[dd] \ar[ll] \ar[rr] & & E\\
& sw \ar[ul] \ar[dr] & & se \ar[ur]   \ar[dl] \\
& & S
}
\]
There are two 
elements of $\Delta$ which map to $c$; we denote them $c_{\diagup}$ and $c_{\diagdown}$, where the subscript indicates the 
slope of the image of a neighborhood of $c_{\cdot}$ in the front diagram.   That is, in $\Delta$, there are maps
$nw \leftarrow c_\diagdown \rightarrow se$ and $ne \leftarrow c_\diagup \rightarrow sw$.   We define
\begin{eqnarray*} \mmon'(F)(c_{\diagup}) & = & \mathrm{Cone}(F(c \to nw)) \\ 
\mmon'(F)(c_{\diagdown}) & = & \mathrm{Cone}(F(c \to ne)) 
 \end{eqnarray*}
(Note it is the cone on the direction going off the branch of legendrian on which the point lives.)  The maps
$\mmon'(F)(ne) \leftarrow \mmon'(F)(c_\diagdown) \rightarrow \mmon'(F)(sw)$, and similarly on the other branch, 
are defined by functoriality of cones, as in the following diagram
\[
\xymatrix{
F(c) \ar[r] \ar[d]_{F(c \to ne)} & F(nw) \ar[r] \ar[d]_{F(nw \to N)} & \mathrm{Cone}(F(c \to nw)) = \mmon'(F)(c_\diagup) 
\ar[d]_{\mathrm{Cone}(F(c \to ne) \to F(nw \to N))}  \\
F(ne) \ar[r]  & F(N) \ar[r]  & \mathrm{Cone}(F(ne \to N)) =  \mmon'(F)(ne) 
}
\]
Note the following are equivalent:
\begin{itemize}
\item
 $\mathrm{Tot}\left(F(c) \to F(ne) \oplus F(nw) \to F(N)\right)$ is acyclic
 \item
  $\mmon'(F)(c_\diagup \to ne): \mmon'(F)(c_\diagup) \to \mmon'(F)(ne)$ is a quasi-isomorphism
  \item  $\mmon'(F) (c_\diagdown \to nw): \mmon'(F)(c_\diagdown) \to \mmon'(F)(nw)$ is a quasi-isomorphism
\end{itemize}

\item Let $c \in \cS$ be a cusp.  As in \eqref{eq:luciustarquiniuscollatinus}
we adopt the following notation for the star of $c$:

\[ 
\xymatrix{
& & a \ar[dll] \ar[d]  \\
O &  c \ar[dr] \ar[ur]  \ar[l] \ar[r] & I \\
& & b \ar[ull] \ar[u]
}
\]

The preimage in $\Delta$ of $c$ is an interval $c_{\preccurlyeq} \to c_\prec \leftarrow c_{\curlyeqprec}$
and the preimage of the star of $c$ is 
$$b \leftarrow c_{\preccurlyeq} \to c_\prec \leftarrow c_{\curlyeqprec} \to a$$. 

We define 
$$\mmon'(F)( c_{\preccurlyeq}) = \mmon'(F)(c_\prec) := \mathrm{Cone}(F(c \to a))$$
$$\mmon'(F)( c_{\curlyeqprec}) := \mmon'(F)(a) = \mathrm{Cone}(F(a \to O))$$

We should provide maps $\mmon'(F)(c_{\preccurlyeq}) \to \mmon'(F)(b)$ and 
$\mmon'(F)(c_{\curlyeqprec}) \to \mmon'(F)(c_\prec)$, the other maps in the preimage of the star of $c$ just 
being sent to equalities.   
The map $\mmon'(F)(c_{\preccurlyeq}) \to \mmon'(F)(b)$ is provided by 
\[
\xymatrix{
F(c) \ar[r] \ar[d]_{F(c \to b)} & F(a) \ar[r] \ar[d]_{F(a \to I)} & \mathrm{Cone}(F(c \to a)) = \mmon'(F)(c_{\preccurlyeq}) 
\ar[d]_{\mathrm{Cone}(F(c \to b) \to F(a \to I))}  \\
F(b) \ar[r]  & F(I) \ar[r]  & \mathrm{Cone}(F(b \to I)) =  \mmon'(F)(b) 
}
\]
The map $\mmon'(F)(c_{\curlyeqprec}) \to \mmon'(F)(c_\prec)$ is subtler.  Applying $F$ to 
$c \to a \to O$, and then using the octahedral axiom gives a distinguished
triangle 
$$ \mathrm{Cone}(F(c \to a)) \to \mathrm{Cone} (F(c \to O)) \to 
\mathrm{Cone}(F(a \to O)) \xrightarrow{[1]} $$
or in other words a morphism
$$ \mmon'(F)(c_{\curlyeqprec}) \to \mmon'(F)(c_\prec) [ 1] $$
which we view as a degree 1 map $\mmon'(F)(c_{\curlyeqprec}) \to \mmon'(F)(c_\prec)$.

Note the following are equivalent:
\begin{itemize}
\item $F(c \to O)$ is a quasi-isomorphism
\item $ \mmon'(F)(c_{\curlyeqprec}) \to \mmon'(F)(c_\prec) [ 1] $ is a quasi-isomorphism
\end{itemize}
\end{itemize}
\end{definition}

If $F \in \dgfun_{\Lambda}(\cS, k)$, then every arrow in $\Delta$ is sent by $\mmon'$ either to a quasi-isomorphism,
or, at a cusp, to a shifted quasi-isomorphism.  Counting the shifts, we have the following.

\begin{proposition}
\label{prop:mumonshift}
If $F \in \dgfun_{\Lambda}(\cS, k)$ and
$a$ is an arc on one component of $\Lambda$, then traveling around the component gives a sequence of 
quasi-isomorphisms
\[
\mmon'(F)(a) \xleftarrow{\sim} \cdots \xrightarrow{\sim} \mmon'(F)(a)[\#\mbox{down cusps} - \#\mbox{up cusps}] = \mmon'(F)(a)[-2r]
\]
where $r$ is the rotation number of the component.  
In particular, on any component of $\Lambda$ that has nonzero rotation number, 
$\mmon'(F)(a)$ must be either unbounded in both directions, or acyclic.  
\end{proposition}
\begin{remark}
In this case it is natural to consider a variant of $\dgsh_\Lambda(M, k)$, 
where the objects are periodic complexes with period dividing $2r$.
\end{remark}

We can use a Maslov potential ({\em cf.} Section \ref{sec:classicalinvariants}) to correct for the shifts.

\begin{definition}
\label{def:mmon}
Fix a Maslov potential $p$ on the front diagram of $\Lambda$.  The normalized microlocal monodromy 
$\mmon: \dgfun(\cS, k) \to \dgfun(\Delta, k )$ is defined on arcs and preimages of crossings  by 
$\mmon(F)(x) = \mmon'(F)(x)[p(x)]$.  In the notation above, we set at a cusp $c$
$$\mmon(F)( c_{\curlyeqprec}) := \mmon(F)(a) = \mathrm{Cone}(F(a \to O))[-p(a)]$$
$$\mmon(F)( c_{\preccurlyeq}) = \mmon(F)(c_\prec) := \mathrm{Cone}(F(c \to a))[-p(b)]$$
\end{definition}

\begin{proposition} \label{prop:mmon}
For $F \in \dgfun_\Lambda(\cS, k)$, the microlocal monodromy
$\mmon(F)$ sends every arrow in $\Delta$ to a quasi-isomorphism.
\end{proposition}

\begin{definition}
$F \in \dgfun_\Lambda(\cS, k)$ is said to have microlocal rank $r$ with respect a fixed Maslov potential $\mu$ 
if 
$\mmon(F)(x)$ is quasi-isomorphic to a locally free $k$-module of rank $r$ placed in degree zero; similarly
for $\cF \in \dgsh_\Lambda(\bR^2, k)$.   We write
$\cC_r(\Lambda)$ for the full subcategory of microlocal rank $r$ objects. 
\end{definition}

The definition of $\cC_r(\Lambda)$ is a specialization to two dimensions of the notion of \emph{pure} complexes in the sense of \cite[\S 7.5]{KS}.  When $r = 1$ and $k$ is a field, such objects are called \emph{simple} in loc. cit.  This notion of purity is not to be confused with the weight-theoretic version of purity we invoke in the discussion of Section \ref{subsubsec:weight-formalism}

\begin{remark}
\label{rem:quote-really-unquote}
We regard the microlocal monodromy
as ``really'' being a functor from sheaves constructible with respect to $\cS$ to 
sheaves on $\Lambda$; the above proposition would say it carries sheaves with singular support in $\Lambda$ to local systems on $\Lambda$.  Following a suggestion of the referee, let us briefly indicate how such a functor can be directly constructed by means of the theory of $\mu\mathit{hom}$ of \cite[\S 4.4]{KS}.  If $F_0$ belongs to $\cC_1(\Lambda)$ (with respect to a given Maslov potential), then the functor
$F \mapsto \mu\mathit{hom}(F_0,F)$, which is defined in great generality, carries an object of $\cC(\Lambda)$ to a complex of sheaves on $T^* \bR^2$ that is cohomologically supported on $\underline{\Lambda}$, and locally constant away from the zero section.
\end{remark}

\begin{proposition}
\label{prop:zigzag}
Let $\Lambda$ be a Legendrian containing a `stabilized' component, i.e., one whose front diagram contains 
a zig-zag of cusps, as in the following picture. 
\begin{center}
\includegraphics[scale=.3]{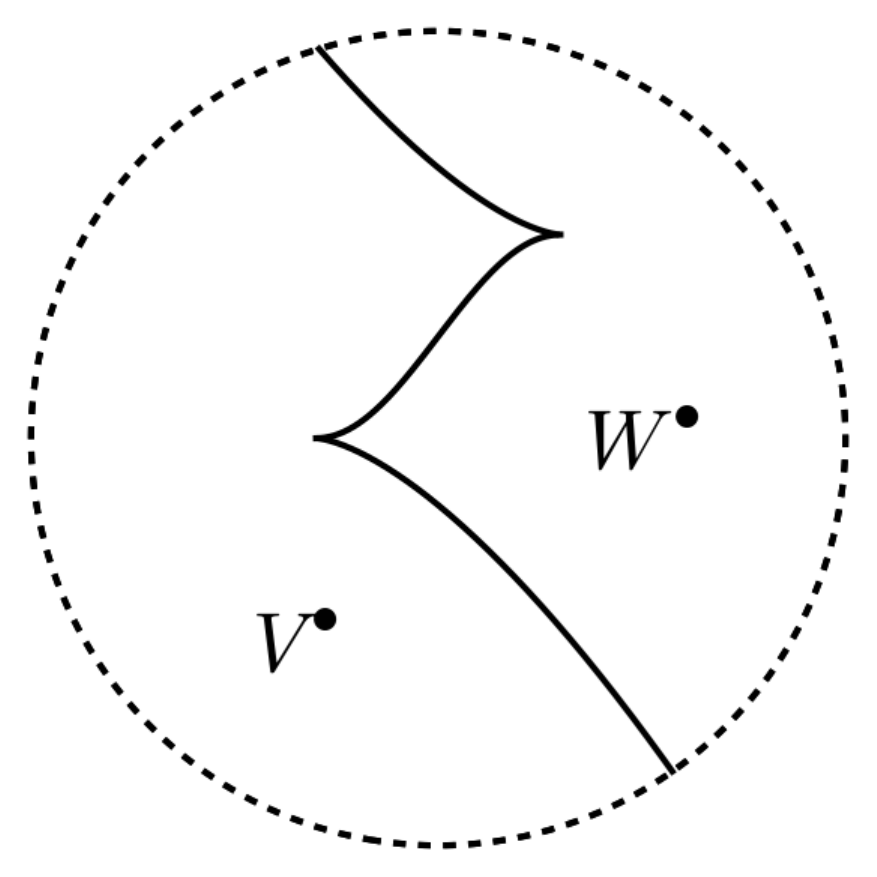}
\end{center}
Then the microlocal monodromy of any object in $\dgsh_{\Lambda}(\bR^2, k)$ is acyclic
along this component.  In particular, if $\Lambda$ has only one component, the category 
$\dgsh_{\Lambda}(\bR^2, k)_0$ is zero. 
\end{proposition}
\begin{proof}
We work locally, and in the combinatorial model.  
The stratification $\cS$ induced by the depicted front diagram is regular; 
down from the top, we label the one-dimensional
strata $x, y, z$, and the cusps $a, b$.   Then zero- and one-dimensional cells of the stratification poset look like 
$x \leftarrow a \rightarrow y \leftarrow b \rightarrow z$, all of which map to the regions, which we'll label $L, R$ on the
left and right.  By Theorem \ref{thm:comb}, elements of $\dgfun_{\Lambda}(\cS, k) \subset
\dgfun_{\Lambda^+}(\cS, k)$ are characterized by requiring the maps 
$x \to L, y \to R, z\to L, a \to y, b \to z, a \to R, b \to L$ to be sent to quasi-isomorphisms.  Our claim that the microlocal monodromy is acyclic
amounts to the statement that all the other maps are sent to quasi-isomorphisms as well. 

It's enough to check on cohomology objects, that is, to show that any element of $\Hom_{\Lambda}(\cS, k)$ 
sends all arrows to isomorphisms.  Invoking Corollary \ref{cor:straight} and Proposition \ref{prop:cohompreserve}, we pass to an element all of whose downward maps
(as listed above) are sent to identity maps.  Fix some such $F$.  Then as in the picture we let 
$V = F(b) = F(x) = F(L) = F(z)$ and $W = F(a) = F(y) = F(R)$, and:
$$F(b) \to F(y) \xleftarrow{=} F(a) \to F(x) \to F(R)$$
which, collapsing the equality, gives
$V \to W \to V \to W$
with both compositions of consecutive arrows equal to the identity.  It follows that the central map $
F(a) = W \to V = F(x)$ is an isomorphism, from which we conclude the microlocal monodromy vanishes. 
\end{proof}

\subsection{Ruling filtrations}
\label{sec:rulingsheaf}

The microlocal monodromy of an object defines a local system on the knot.
Thinking of our object as living in the Fukaya category, if it were a geometric Lagrangian
with a local system, microlocal monodromy would correspond to the restriction of the
local system to infinity in the cotangent fibers.  However, not all objects
are geometric Lagrangians:  any object $C$ which is equivalent to a geometric brane must have
$\Hom^*(C,C)$ isomorphic to the cohomology of a (not necessarily connected) surface,
and not all objects do.  Yet any object {\em is} an iterated cone of
sums of geometric objects.  Now one might expect ``cone'' to translate geometrically into
symplectic surgery,
but sometimes this surgery cannot be performed in the ambient space.
Topologically, it is no problem, and accordingly one might 
expect that a filtration of our object whose associated graded objects
come from Lagrangians would give rise to a local system on some topological
surface.  In this subsection, we give such a construction and relate it
to the more familiar notion of rulings of front diagrams.

Let $\Phi$ be a front diagram.  A
 ruling of $\Phi$ is a partition of the arcs of $\Phi$ into paths which
travel from left to right, beginning at a left cusp, ending at a right cusp, meeting no cusps in-between, and
having the property that the two paths which begin at a given left cusp end at the same right cusp and do not meet in-between.  
The simply connected region between
these two paths is called an ``eye''.  For instance, the trefoil of Example \ref{ex:trefoil} has a ruling whose two eyes are pictured in Figure \ref{fig:trefoileyes}.
\begin{figure}[H]
\includegraphics[scale=.3]{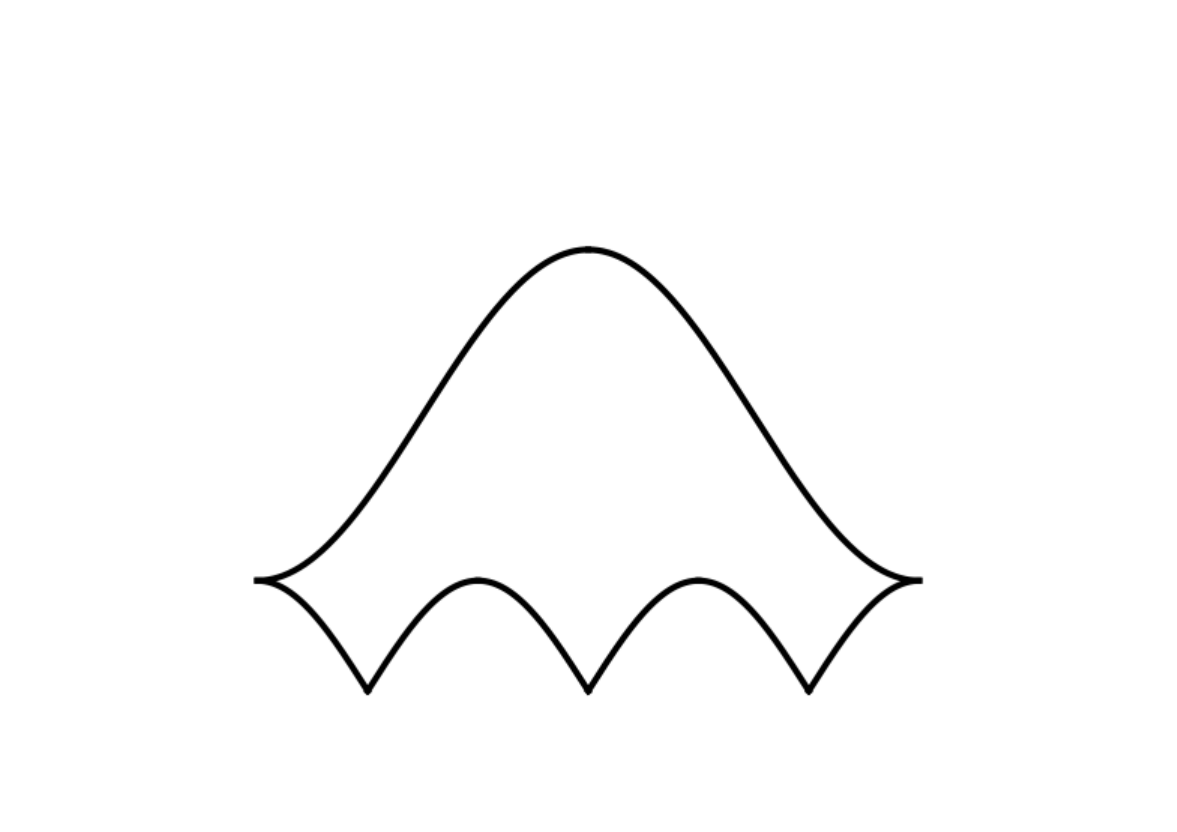}
\includegraphics[scale=.3]{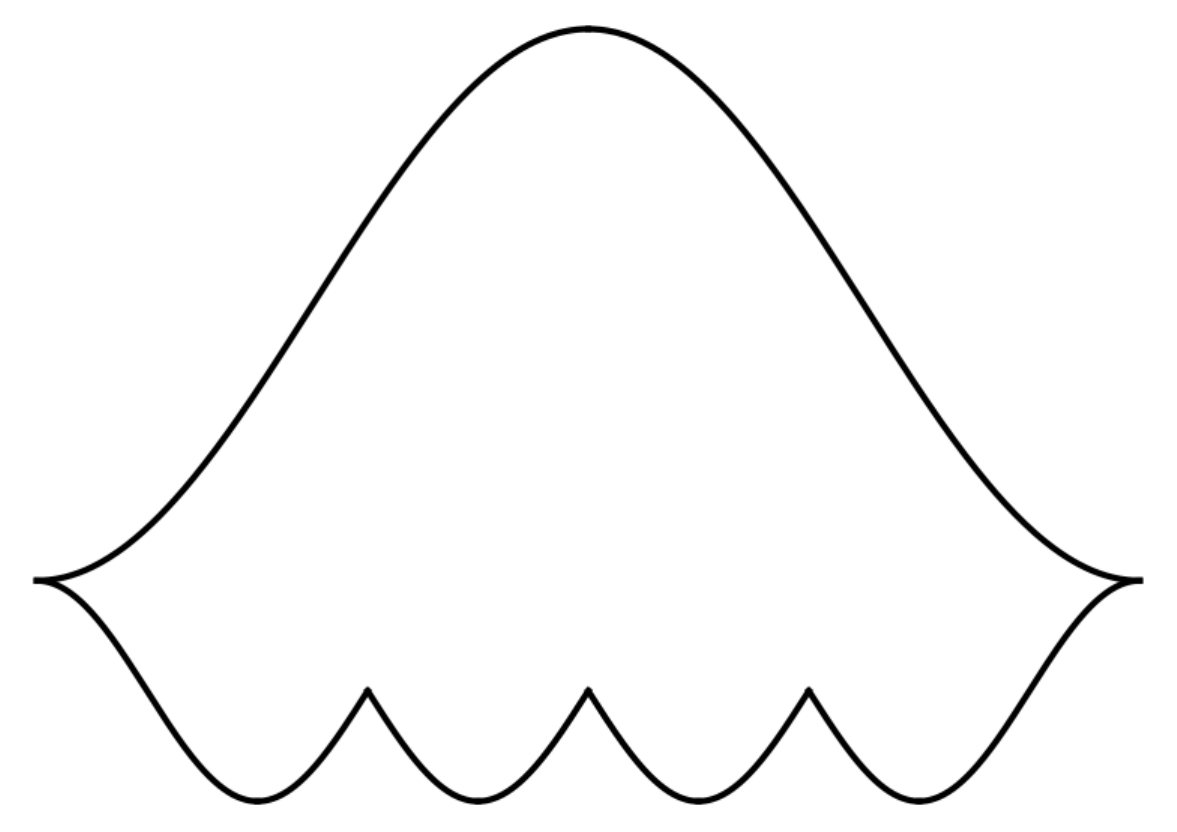}
\caption{ \label{fig:trefoileyes} Two eyes of a ruling of the Legendrian trefoil.}
\end{figure}

A ruling determines, and is determined by, a subset of the set of crossings of $\Phi$, called \emph{switches}.  Whenever a path of the ruling meets a crossing, it can either continue along the strand it came in on, or not --- in the second case we call the crossing a switch.  
 In the ruling of Figure \ref{fig:trefoileyes}, every crossing is a switch.

If two paths of a ruling meet at a switch $\mathbf{x} = (x_0,z_0)$, then near $\mathbf{x}$, one
path always has a higher $z$ coordinate
than the other at each fixed $x$.  We say the one with the higher $z$ coordinate is ``above'' and the other ``below'' $\mathbf{x}$, 
although it need not be true that the $z$ coordinate of the path ``above'' $\mathbf{x}$ is greater than $z_0$. 
  In the presence of a Maslov potential $\mu$, we say a ruling is {\em graded} if $\mu$ takes the same value on
the two strands which meet at any switched crossing.  

Let $\Lambda$ be a Legendrian knot with front diagram $\Phi \subset \bR^2$. 
Any eye of a ruling of $\Phi$ gives rise to a {\em eye sheaf} $\cE \in \dgsh_{\Lambda^+}(\Phi)$.  
As a sheaf, it is determined by the condition that its stalk is an invertible $k$-module (that is, when $k$ is a field, a one-dimensional vector space) in degree zero inside the eye and along its upper boundary, and stalk
zero elsewhere.  We sometimes refer to $\cE[d]$ as an eye sheaf as well. 
If $\cE$ is an eye sheaf, we write $\underline{\cE}$ for the underlying eye.

We recall that in a triangulated category, a filtration $R_\bullet$ of $F$ is just a diagram
\begin{equation}
\label{eq:filtration}
\xymatrix{
0= R_0 F\ar[rr] & & R_1 F\ar[dl]\ar[rr] & & R_2 F\ar[r]\ar[dl] & \dots\ar[r] & 
R_n F=F \\ 
&\Gr^R_1 F \ar[ul]_{[1]}& & \Gr^R_2 F \ar[ul]_{[1]} 
& &  & 
}
\end{equation}

\begin{definition} Let the front diagram of $\Lambda$ have $c$ left (or equivalently right) cusps, and 
let $\cF \in \dgsh_{\Lambda}(\bR^2, k)$ have microlocal rank one. 
A ruling filtration $R_\bullet \cF$ is a 
filtration such that
$\Gr^R_i F$ is an eye sheaf for $i \in [1,c]$, and zero for any other $i$.  \end{definition}

\begin{lemma} 
Let $\cF \in \dgsh_{\Lambda}(\bR^2, k)$ have microlocal rank one.  
Let $R_\bullet \cF$ be a ruling filtration.  Then the eyes 
$\underline{\Gr^R_i F}$ determine a ruling of the front diagram of $\Lambda$; in particular,
for any arc $a$, there is exactly one $i(a)$ such that 
$a$ lies on the boundary of $\underline{\Gr^R_{i(a)} \cF}$.  
\end{lemma}
\begin{proof}
We have $\bR_{>0} \Lambda \subset SS(\cF) \subset \bigcup SS(\Gr^R_i F)$, and
we further know that there are only as many steps in the filtration as left cusps of the front diagram. 
\end{proof}

If $R = R_\bullet \cF$ is a ruling filtration as above, we write $\underline{R}$ for the corresponding ruling
and call it the ruling
underlying the ruling filtration $R$. 

\begin{lemma} \label{lem:fillmmon}
Let $\cF \in \dgsh_{\Lambda}(\bR^2, k)$ have microlocal rank one.  
Let $R_\bullet \cF$ be a ruling filtration.  Let 
$a \to N$ be an upward
generization map from an arc.  Then the maps $\Gr^R_{i(a)} \cF \leftarrow R_{i(a)} \cF \to \cF$ 
induce quasi-isomorphisms
$$\mathrm{Cone}(\Gr^R_{i(a)} \cF(a)  \to  \Gr^R_{i(a)} \cF(N) )
\xleftarrow{\sim} \mathrm{Cone}(R_{i(a)} \cF(a)  \to  R_{i(a)} \cF(N) ) \xrightarrow{\sim}
  \mathrm{Cone}(\cF(a)  \to  \cF(N) ) 
$$
\end{lemma}
\begin{proof}
By the nine lemma for triangulated categories
(see, e.g., Lemma 2.6 of \cite{May}), we have a diagram
(in which all the squares commute except the non-displayed square at the bottom right, which anti-commutes)

$$
\xymatrix{
R_{j-1} \cF(a) \ar[r] \ar[d] & R_{j-1} \cF(N) \ar[r] \ar[d] & \mathrm{Cone}(R_{j-1} \cF(a) \to R_{j-1} \cF(N) ) \ar[r] \ar[d] & \\
R_j \cF(a)  \ar[r] \ar[d] & R_j \cF(N) \ar[r] \ar[d] & \mathrm{Cone}(R_{j} \cF(a) \to R_{j} \cF(N)) \ar[r] \ar[d] & \\
\Gr^R_j \cF (a) \ar[r] \ar[d] & \Gr^R_j \cF(N) \ar[r] \ar[d] & \mathrm{Cone}(\Gr^R_j \cF(a) \to \Gr^R_j \cF(N)) \ar[r] \ar[d] & \\
 & & & 
}
$$
The complex $\mathrm{Cone}(\Gr^R_j \cF(a) \to \Gr^R_j \cF(N))$ is acyclic except for $j = i(a)$.  From
the rightmost column, we see that 
$\mathrm{Cone}(R_{j-1} \cF(a) \to R_{j-1} \cF(N) ) \to \mathrm{Cone}(R_{j} \cF(a) \to R_{j} \cF(N))$ is a quasi-isomorphism
except when $j = i(a)$.  

Thus for $j < i(a)$ we have that $\mathrm{Cone}(R_{j} \cF(a) \to R_{j} \cF(N))$ is acyclic.  
When $j = i(a)$, we see that $\mathrm{Cone}(R_{j} \cF(a) \to R_{j} \cF(N)) \to  \mathrm{Cone}(\Gr^R_j \cF(a) \to \Gr^R_j \cF(N))$
is a quasi-isomorphism.  Finally, composing  quasi-isomorphisms
for $j > i(a)$, we get
$\mathrm{Cone}(R_{j} \cF(a) \to R_{j} \cF(N)) \xrightarrow{\sim} \mathrm{Cone}( \cF(a) \to  \cF(N))$. 
\end{proof}

\begin{corollary}
The underlying ruling $\underline{R}$ of a ruling filtration $R_\bullet \cF$ is graded. 
\end{corollary}
\begin{proof}
Being graded is a condition at the switched crossings of a ruling; let $c$ be one such and 
let us as usual pass to the following local model.
\[
\xymatrix{
& & N\\
& nw \ar[ur] \ar[dl] & & ne \ar[ul] \ar[dr] \\
W  & & c \ar[ul] \ar[ur] \ar[dl] \ar[dr] \ar[uu] \ar[dd] \ar[ll] \ar[rr] & & E\\
& sw \ar[ul] \ar[dr] & & se \ar[ur]   \ar[dl] \\
& & S
}
\]
By definition, for $c$ to be a switched crossing, we must have $i(nw) = i(ne)$ and $i(sw) = i(se)$.  
In particular the generization maps induced by $\cF(nw) \to \cF(N)$ and $\cF(ne) \to \cF(N)$ have 
cones which are one-dimensional vector spaces in the same degree, determined only by the degree
in which the eye sheaf
$\Gr^R_{i(nw)} \cF$ sits, and on whether $nw, ne$ are on the top or bottom of the corresponding eye. 
On the other hand, we have from the fact that $\cF$ has microlocal rank one and
from Definition \ref{def:mmon} that
this degree also encodes the Maslov potential at $ne, nw$, which are therefore equal. 
\end{proof}

Following \cite{K}, to each ruling of the front diagram
we associate a topological surface whose boundary can
be identified with the Legendrian.
Let $ \coprod \{E_c\} $ be the disjoint union of the eyes.
Orient in one way all the eyes with bottom strand of even Maslov index, and in the other
way all the eyes with bottom strand of odd Maslov index.  At each switched crossing, glue the incident eyes with a half-twisted strip.  The resulting
surface is orientable iff the ruling is 2-graded, and is in particular orientable if the ruling is graded. 
It is called the filling surface of the ruling. 

\begin{figure}[H]
\includegraphics[scale=.3]{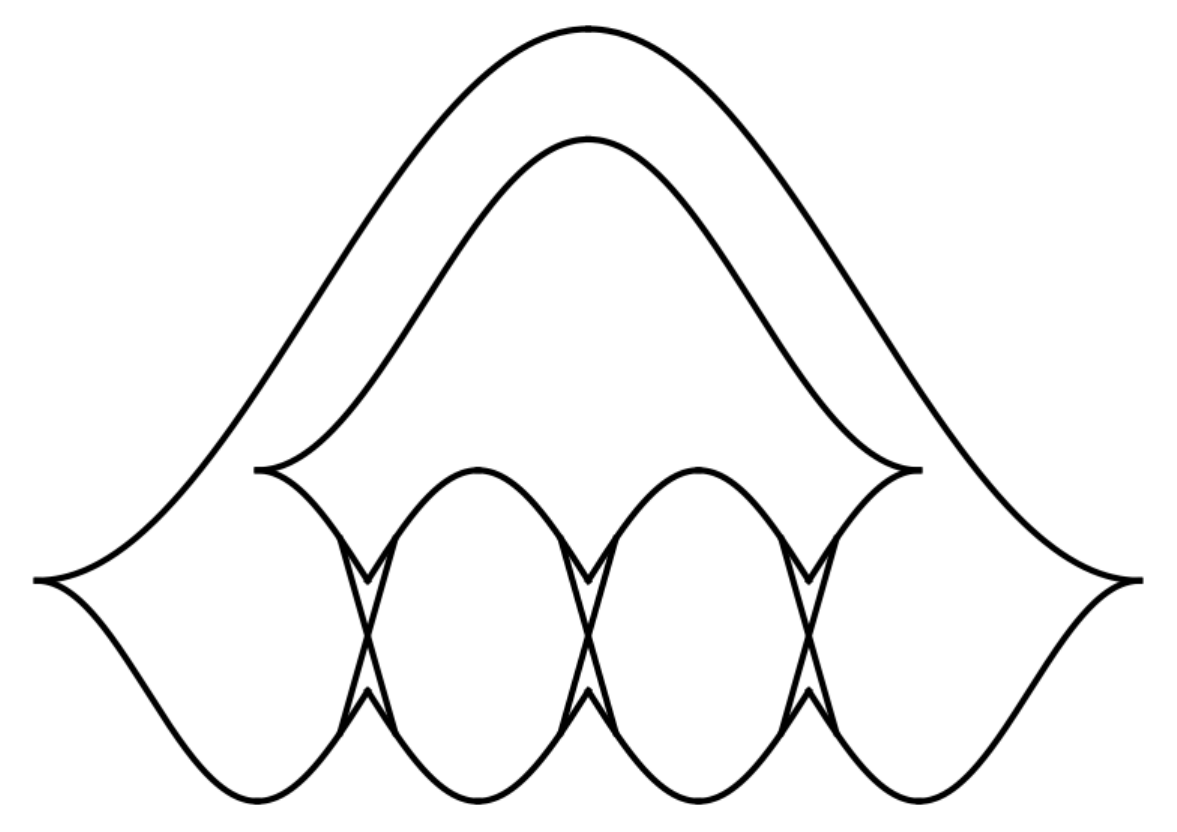}
\caption{The  filling surface associated to the ruling of Figure
\ref{fig:trefoileyes}}
\label{fig:fillingsurface}
\end{figure}

A ruling filtration of a rank one object $\cF$ gives a rank one local system on this surface.  We specify it by
trivializing on each eye, and giving gluing isomorphisms at the switched crossings.  An eye sheaf has a unique nonvanishing
homology sheaf, whose stalks are everywhere canonically identified since the eye is contractible;  we fix an isomorphism to a
one dimensional vector space.  The stalk of the eye sheaf is zero along the lower boundary of the eye, but on 
the filling surface we extend the constant local system here as well.  Then, by 
Lemma \ref{lem:fillmmon}, along the boundary of the coordinate
chart on the filling surface, the stalk of the constant local system is everywhere identified with the microlocal
monodromy of $\cF$ along that boundary.  Now at a switched crossing, the northwest ($nw$) and southeast ($se$) arcs
on the one hand belong to different eyes, and on the other the microlocal monodromies along them are canonically identified
by Proposition \ref{prop:mmon}.  We summarize this construction in: 

\begin{proposition} \label{prop:rf} Let $\cF \in \dgsh_{\Lambda}(\bR^2 , k)$ have microlocal rank one.  Then
a ruling filtration $R_\bullet \cF$ gives rise to a rank one local system on the filling surface of the ruling
$\underline{R}$. \end{proposition}

\begin{remark}
If instead of using the northwest and southeast arcs, we had used the northeast and southwest arcs, the isomorphism
across each crossing would change by a sign.
\end{remark}

The following observation is useful for inductive arguments:

\begin{definition}
Let $\Phi$ be the front diagram of $\Lambda$, and let $\underline{E}$ be an eye of $\Phi$.  We write 
$\Phi \setminus \underline{E}$ for the front diagram given by removing the arcs and 
cusps of $\underline{E}$, and $\Lambda \setminus \underline{E}$ the corresponding Legendrian.  (See Remark \ref{rem:C0}.)  
Note that a Maslov potential on $\Phi$ restricts to a Maslov potential on $\Phi \setminus \underline{E}$. 
\end{definition}

\begin{lemma} \label{splotch} {\bf (Enucleation.)}
Let $\Phi$ be the front diagram of $\Lambda$, and let $\cF \in \dgsh_{\Lambda}(\bR^2, k)$
be an object of microlocal rank one. 
Let $\cE \in \dgsh_{\Lambda^+}(\bR^2, k)$ be an eye sheaf.  

Suppose given a morphism $\cF \to \cE$.  
Assume the microsupport of the cone
projects to $\overline{\Lambda \setminus \underline{\cE}}$. 
Then in fact the microsupport is contained in the closure
of the conormal lift of $\Lambda \setminus \underline{\cE}$:
i.e., it does not contain the conormals to any crossing in
the front diagram.   
Moreover, in this case the appropriate shift
$\mathrm{Cone}(\cF \to \cE)[-1] \in \dgsh_{(\Lambda \setminus \underline{\cE})}(\bR^2, k)$ has microlocal rank one. 
\end{lemma}

One mystery regarding ruling filtrations is how the notion of `normal ruling' will appear.  We recall this briefly. 
Suppose $\Phi \subset \bR^2$ is a front diagram whose cusps and crossings all have different $x$-coordinates; i.e., suppose a sufficiently 
thin vertical strip around a given crossing contains no cusps or other crossings.  Consider two eyes which meet at a switched crossing
$\mathbf{x}$.  We write $t, b$ for the top and bottom strand of one eye, and $t', b'$ for the top and bottom of the other.  
Fix a sufficiently thin vertical strip containing $x$ and no other crossings, and without loss of generality let $t$ be above $t'$ in this strip. 
Then the crossing is said to be {\em normal} so long as the vertical order in which $t, b, t', b'$ meet the left or equivalently
right boundary of the strip is {\em not} $t t' b b'$, i.e. it is either $tt'b'b$ or $tbt'b'$.  Or in other words, the vertical strip meets the two eyes in
a way that a vertical strip might meet the two eyes of a two-eye unlink.

If a plane isotopy does not preserve the relative positions of the $x$-coordinates of the cusps and crossings, it will not necessarily carry a normal ruling to a normal ruling.  However the number of normal rulings is in fact a Legendrian invariant, and indeed so is the `ruling polynomial,' which is the generating polynomial in the variable $z$ such
that the coefficient of $z^i$ is the number of normal rulings whose filling surface has Euler characteristic $i$ \cite{ChP, Ru, K}.

It is {\em not} generally true that the ruling underlying a ruling filtration is always normal, even for the trefoil.  For example, the object with $l_1 = l_4$ in Example \ref{ex:rainbowtrefoil} has a large eye sub-sheaf
corresponding to the non-normal ruling whose unique switch is the middle crossing.
However, for positive
braid closures in the plane, we will see in Proposition \ref{prop:oneruling} there is a natural condition on ruling filtrations which can be imposed to ensure
that (1) every rank one object admits a unique ruling filtration of this kind and (2) the underlying ruling is normal.

We expect the same can be said for an arbitrary Legendrian knot.  This is because rank one objects  
conjecturally correspond to augmentations of the Chekanov-Eliashberg dga, to which
normal rulings may be associated \cite{NS, HR}.
Determining how to express this condition
for a general Legendrian knot is an important open problem.

\subsection{Binary Maslov potentials}
\label{subsec:bimaslov}
We say a Maslov potential is {\em binary} when it takes only two values.  
Such a Maslov potential on a diagram is unique up to a shift, which we choose so that the value is $0$ on any strand that is below a cusp and $1$ on any strand that is above a cusp.  

Certain sheaves of microlocal rank $r$ on $(\bR^2, \Lambda)$ 
with respect to a binary Maslov potential can be described by linear algebra rather than homological algebra.  

Recall that  $\dgsh_{\cS}(M, k)_0$ denotes the $\cS$-constructible sheaves which have acyclic stalks as $z \to -\infty$, 
and similarly $\dgsh_{\Lambda}(M, k)_0$, $\dgfun_{\Lambda}(\cS, k)_0$, etc. 

\begin{proposition} \label{prop:bimaslov}
Assume $\Lambda \subset T^{\infty, -} \bR^2$ 
carries a binary Maslov potential $p$.  Then any $\cF \in \dgsh_{\Lambda}(\cS, k)_0$ with microlocal rank $r$ is
quasi-isomorphic to its zeroth cohomology sheaf.  Similarly, 
any $F \in \dgfun_{\Lambda}(\cS, k)_0$  
with microlocal rank $r$ is quasi-isomorphic to its zeroth cohomology. 
\end{proposition}
\begin{proof}
We treat the case of $F \in \dgfun_{\Lambda}(\cS, k)_0$.
Since every downward map in $\cS$ is sent by $F$ to a quasi-isomorphism, it's enough to study what happens along upward maps. 
Moreover we can ignore cusps and crossings, since the stalks of these are anyway quasi-isomorphic to stalks on something below them. 

So consider an upward map $a \to N$.   We'll write $F^i(x):=H^i(F(x))$.  
If $p(a) = 0$, then by hypothesis we have an exact sequence
$$ 0 \to F^0(a) \to F^0(N) \to \kappa_r \to F^1(a) \to F^1(N) \to 0$$
where $\kappa_r$ is a projective $k$-module of rank $r$.  
Moreover $F^i(a) \cong F^i(N)$ for all $i \notin [0,1]$. 

On the other hand, if $p(a) = 1$, then 
we have an exact sequence
$$ 0 \to F^{-1}(a) \to F^{-1}(N) \to \kappa_r \to F^0(a) \to F^0(N) \to 0$$
and $F^i(a) \cong F^i(N)$ for all $i \notin [-1,0]$. 

In particular, in both cases, $\dim F^{> 0 }(a) \ge \dim F^{> 0}(N)$, and so we see the rank of $F^{>0}$ never decreases when
traversing the diagram downwards.  So if any region has positive dimensional $F^{>0}$, so will the region outside the knot,
which is a contradiction.  Likewise, $\dim F^{< 0}$ never decreases while traversing the diagram upwards, so 
if any region has positive dimensional $F^{< 0}$, then so will the region outside the knot; contradiction. 
\end{proof} 

\begin{remark}
The horizontal Hopf link and vertical Hopf link are Legendrian isotopic.  They also both admit binary Maslov potentials.  However, 
the Legendrian isotopy {\em does not} identify these Maslov potentials.  
\end{remark}

\begin{proposition} \label{prop:bimapfiltrations}
Let $\Lambda$ admit a binary Maslov potential, and assume $\cF \in \dgsh_{\Lambda}(\bR^2,k)$ has microlocal
rank one.  Let $R_\bullet \cF$ be a ruling filtration.  Then all the terms $R_\bullet \cF$ are concentrated in degree zero. 
\end{proposition}
\begin{proof}
Follows from Lemma \ref{splotch} and induction. 
\end{proof}

Moreover, in the binary Maslov setting, we have nice moduli spaces.

\begin{proposition} \label{prop:moduli}
Assume the front diagram of $\Lambda$ carries a binary Maslov potential.  Let 
$\cM_r(\Lambda)$ be the functor from $k$-algebras to groupoids by letting $\cM_r(\Lambda)(k')$ be the groupoid
whose objects are the rank $r$ sheaves in $\dgsh_{\Lambda}(\bR^2;k')_0$,
and whose morphisms are the quasi-isomorphisms.  Then $\cM_r(\Lambda)$ is representable by an Artin stack
of finite type.  
\end{proposition}
\begin{proof}
We fix a regular stratification $\cS$ and work in $\dgfun_{\Lambda}(\cS,k')$; by
Proposition \ref{prop:bimaslov}
this is equivalent to $\Fun_{\Lambda}(\cS, k')$.  Recall any element of this has a representative
with zero stalks in the connected component of infinity; we always take such a representative.  Moreover by
Corollary \ref{cor:straight}, we  take a representative with all downward generization maps identities.  It suffices
therefore to discuss the two dimensional cells: any object is determined by the upward generization maps
from stalks on arcs (which may be identified with the stalks of the region below the arc) to stalks on the region above
the arc; since all downward generization maps are identities, any map between objects is
characterized by its action on regions.  

For each two dimensional cell $R$ of $\cS$, let $d(R)$ denote the number of strands below $R$ with Maslov potential zero.   
Then one sees the $k'$-module $F(R)$ is projective of rank $r \cdot d(R)$.  The data of $F$ amounts to a morphism
$F(R) \to F(S)$ for every arc $a$ separating a region $S$ above from a region $R$ below; 
in this case let $d'(a) = d(R)$ and $d''(a) = d(S)$.  
The parameter space of 
such data is the affine space $\cA:= \prod_a \Hom(\bA^{r \cdot d'(a)},\bA^{r \cdot d''(a)})$, where $\Hom$ means linear
homomorphisms; isomorphisms of such data are controlled by the algebraic group $\cG := \prod_{R} \GL_{d(R)}$.  

The conditions that such data actually give an object $F \in \Hom_{\Lambda}(\cS, k')$ are locally closed in $\cA$; let
$\cB$ be the subscheme cut out by them.  
Then $\cG \backslash \cB = \cM_r(\Lambda)$. 
\end{proof}

\begin{remark}
The moduli space of special Lagrangians in a compact Calabi-Yau manifold has been considered in \cite{H}.  When the sLags are decorated with unitary local systems, the resulting space
has the structure of a Kahler manifold.  An analogous object in our setting is the moduli of sLag surfaces in $\bR^4 \cong \bC^2$ that end on the Legendrian knot $\Lambda$.  Most of these sLags are not exact, therefore they do not directly determine objects of $\dgsh_{\Lambda}(\bR^2;k)$.  Nevertheless, we expect the $\cM_r(\Lambda)$ to resemble the moduli of sLags.  
\end{remark}

When $\Lambda$ does not carry a binary Maslov potential, there can be objects $\cF \in \cC_1(\Lambda)$ 
with $\Ext^i(F,F) \neq 0$ for negative $i$.  This happens, for instance, on one of the Chekanov knots, see
\S \ref{sec:conj}.  In this setting it may be more natural to incorporate such structures into the moduli space
by viewing $\cM_r(\Lambda)$ as a derived stack
(see for instance the survey \cite{Toen}).

\section{Braids}
\label{sec:braids}

Sheaves restrict: for an open subset $Y \subset X$, we have 
$\dgsh_\Lambda(X; k) \to \dgsh_{\Lambda \cap T^\infty Y}(Y; k)$. 
Sheaves glue, so $\dgsh_\Lambda(X; k)$ can be reconstructed from the categories
associated to an open cover and the appropriate restriction data.  
This induces restriction morphisms between moduli spaces, and 
descriptions of global moduli spaces as limits of appropriate diagrams.  

The above general notions lead to remarkably intricate structures, 
already in the special case where $X = \mathbb{R}^2$ 
and $\Lambda$ is a positive braid. 
In this case, the moduli space
admits a description in terms of iterated affine bundles over a flag variety, and
in fact is naturally identified with the open Bott-Samelson variety of the braid.  
From the main invariance theorem \cite{GKS}, we deduce that this association gives
a categorical representation of the positive braid monoid, recovering 
a construction of Deligne \cite{D-braids}. 

The connection to categorical braid representations leads, via the work of
Williamson and Webster, to a connection between our constructions and 
the triply graded Khovanov-Rozansky homology.  This connection is 
mediated by closing the positive braid in the solid, corresponding to closing
its front projection on a cylinder by gluing the ends of the horizontal strip. 

A different closure --- the ``rainbow closure'' in the front plane $\mathbb{R}^2$ --- 
gives rise to a different connection to topological knot invariants.  Indeed,
we show that each rank one object on the rainbow closure of a positive braid, 
carries a unique ruling filtration.  This provides
a decomposition of the moduli space of objects labeled by rulings, 
allowing us
to exploit the work of Rutherford \cite{Ru} to relate the orbifold cardinality of the 
category over finite fields to part of the HOMFLY polynomial.
We explain all this in the subsections below.

Throughout this section we assume $k$ is a field, in order to apply 
Proposition \ref{prop:moduli}.

\subsection{Braids and their closures}

We write $\Br_n$ for the Artin braid group on $n$ strands, i.e. 
\[
\Br_n = \langle s_1^{\pm},\ldots, s_{n-1}^{\pm} \rangle / \{s_i s_j = s_j s_i \mbox{ for } |i-j| \ne 1,\,\,\,\,  s_i s_{i+1} s_i = s_{i+1} s_i s_{i+1}\}
\]
Geometrically we take an element of $\Br_n$ to be an isotopy class of the following sort of object: an $n$-tuple of disjoint smooth sections of the projection  $[0,1]_x \times \RR^2_{y,z} \to [0,1]_x$ such that the $i$th section has $z$-coordinate some $1,\ldots,n$ in a neighborhood of $x=0$ and of  $x = 1$.

An isotopy allows one to ensure that the images of the sections under projection to the $xz$-plane are 
immersed and coincide only transversely and in pairs.  We may take 
the sections to have constant $y$ and $z$ coordinates --- $y = 0$ and $z \in \{1, 2, \ldots, n\}$
---  except for $x$ in a neighborhood of these coincidences; this gives
the standard generators of the braid group by half twists.   We write $s_i$ for the positive (counterclockwise) half twist between the 
strands with $z$ coordinate $i$ and $i+1$, and $s_i^{-1}$ for its inverse.

We say a braid $\beta \in \Br_n$ is positive if it can be expressed 
as a product of the $s_i$, and write $\Br^+_n$ for the set of positive braids.  
Taking a geometric description as above of a positive braid, 
view the projection to the $xz$ plane as a front diagram  (there are no vertical tangents since the projection to the $xz$ plane
is an immersion).   Then the associated Legendrian is isotopic to the original braid, and since the braid relation is a Legendrian Reidemeister
III move, this construction gives a well defined Hamiltonian isotopy class of Legendrian.  

For example, Figure \ref{fig:braidproj} shows
the front projection corresponding to the expression $s_2 s_1 s_1 s_3$.
\begin{figure}[H]
\includegraphics[scale = .3,angle = 180]{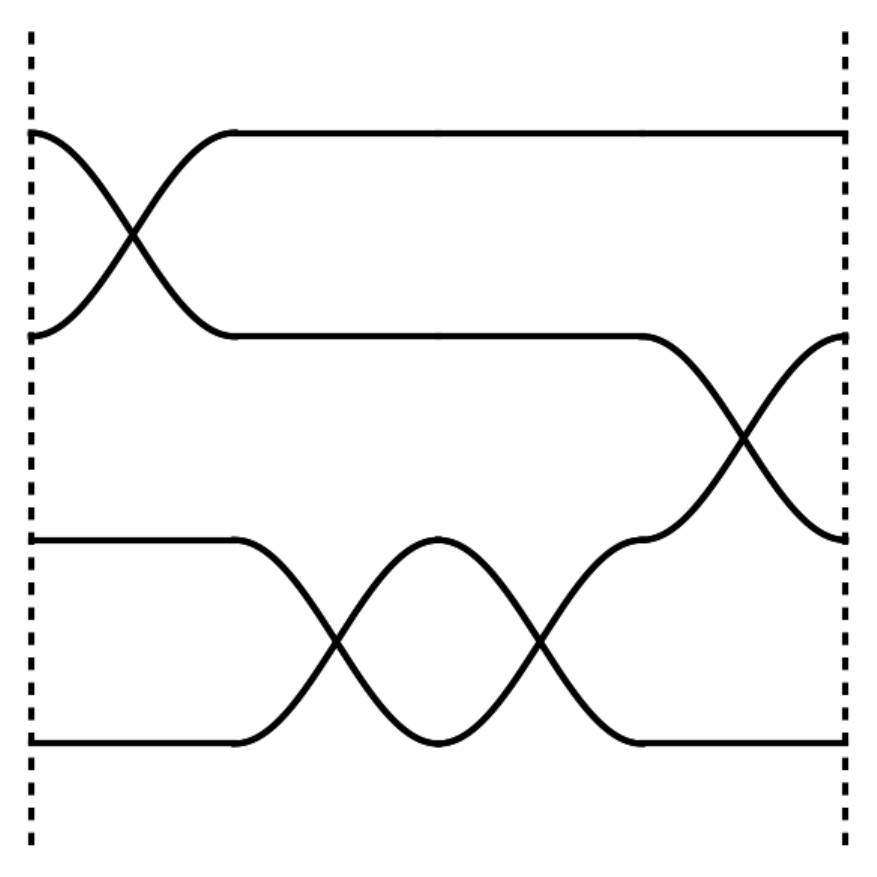} 
\caption{A braid projection}
\label{fig:braidproj}
\end{figure}

One obtains a knot from a braid by joining the ends in some way.  
There are several of these; one we will not consider is the plat closure --- 
where, at each end, the first and second; third and fourth; etc., strands are joined
together by cusps.  In fact, one can see using the Reidemeister II move that 
all Legendrian knots arise as plat closures of positive braids.  

We will instead consider the braid closure.  Topologically, 
the braid closure amounts to joining the highest strand on the right to the highest strand on
the left, and so on.  We will consider two different variants on this.  The first, 
which we call the cylindrical closure, simply identifies the right and left sides
of the front diagram, giving a front diagram in the cylinder  $S^1_x \times \bR_z$ and hence 
a legendrian knot in $T^{\infty,-} (S^1_x \times \bR_z)$.  The second, which we call the planar
or rainbow closure, gives a front diagram in the plane formed by `cusping off' the endpoints 
as in Figure \ref{fig:braidclosure} below.
\begin{figure}[H] 
\includegraphics[scale = .3,angle = 180]{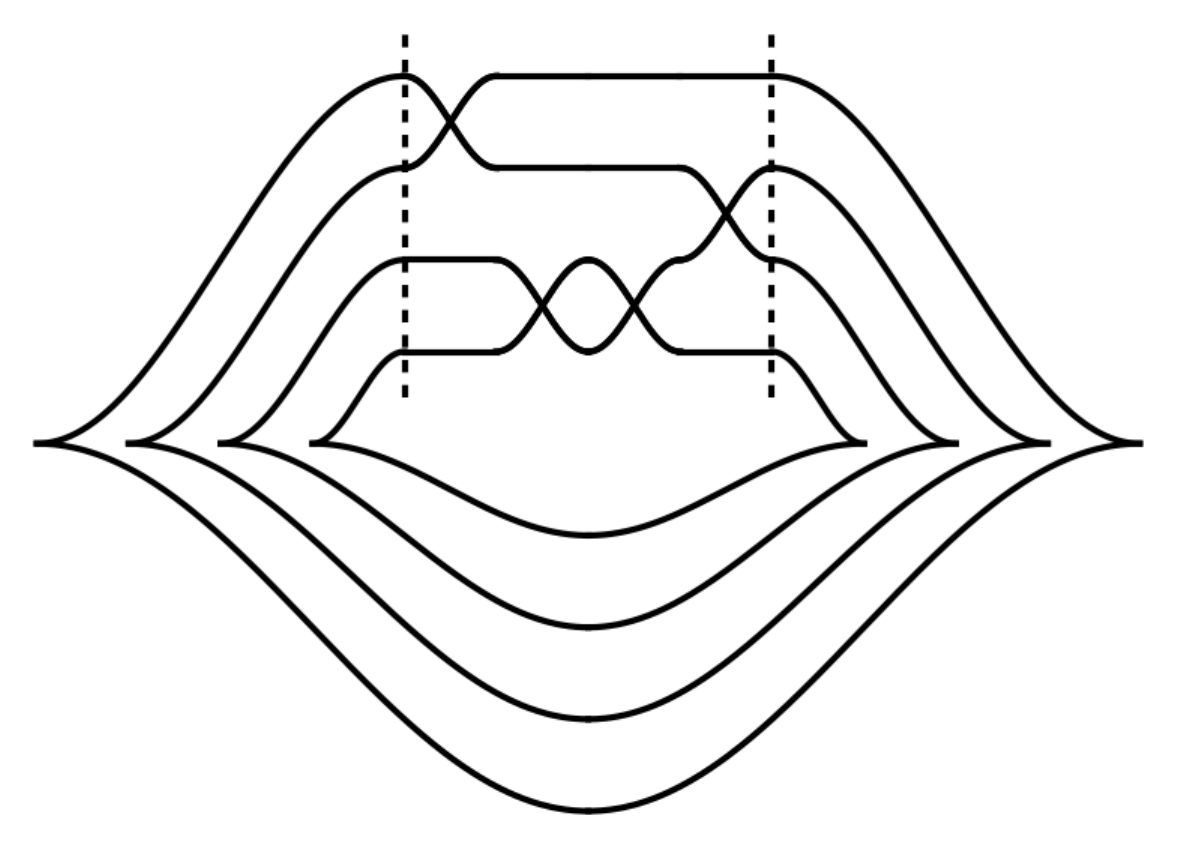} 
\caption{A braid closure ``cusped off,'' or closed ``over the top.''}
\label{fig:braidclosure}
\end{figure}

\subsection{Sheaves microsupported along braids} \label{subsec:braidmoduli}
We begin by studying the local picture: sheaves microsupported in braids, 
or equivalently, the restriction of a sheaf $\cF$ with singular support 
along the braid closure to a rectangle containing all crossings of the braid $\beta$, 
i.e., a picture as in  Figure \ref{fig:braidproj}.  

In this context we will be interested in sheaves with acyclic stalks 
in the connected component of $z \to -\infty$; we denote this full subcategory by $\dgsh_\beta(\bR^2, k)_0$, and in the Maslov potential which is identically zero
on the braid.

\begin{proposition}
Let $\beta$ be a Legendrian whose front diagram is a positive braid.  
Fix the Maslov potential which is everywhere zero. 
Let $\cF \in \dgsh_\beta(\bR^2, k)$ be such that $\mmon(\cF)$ is
concentrated in degree zero, and assume the same for the stalk of $\cF$ at a
point of $\bR^2$.  Then $\cF$ is quasi-isomorphic to its
zeroeth cohomology sheaf. 
\end{proposition}
\begin{proof}
Like Proposition \ref{prop:bimaslov}, but easier. 
\end{proof}

Combining this with Proposition \ref{prop:braidlegible}, we see that objects
of $\dgsh_\beta(\bR^2, k)_0$ can be described by legible diagrams (in the sense
of Section \ref{subsec:legible}), and moreover that
every region is assigned a $k$-module, rather than a complex of them.  
That is, 

\begin{proposition} \label{prop:bos} 
Let $\beta$ be a braid; fix the zero Maslov potential on its front diagram.  
Let $Q_\beta$ be the quiver with one vertex for each region in the front diagram,
and one arrow $S \to N$ for each arc $a$ separating a region $N$ above from a region $S$ below. 
Then $\dgsh_\beta(\bR^2, k)_0$ is equivalent to the full subcategory of representations of $Q$ in which
\begin{itemize}
\item The vertex corresponding to the connected component of $z \to - \infty$ is sent to zero.
\item All maps are injective.
\item If $N, E, S, W$ are the north, east, south, and west regions at a crossing, then the 
sequence $0 \to F(S) \to F(E) \oplus F(W) \to F(N) \to 0$ is exact. 
\end{itemize}
\end{proposition}

For a positive braid $\beta$, we write $\cC(\beta):=\dgsh_\beta(\bR^2, k)_0$, 
and $\overline{\cC}(\beta) = \dgsh_{\beta^+}(M, k)_0$. 
We write $\cC_r(\beta)$ and $\overline{\cC}_r(\beta)$ for the corresponding subcategories 
of objects of microlocal rank $r$ with respect to the zero Maslov potential.  

We write `$\equiv_n$' for the identity in $\Br_n$, and we omit the subscript when no confusion
will arise.  By cutting the front diagram into overlapping vertical strips,
each of which contains a single crossing, and such that the overlaps contain trivial braids, we find from the sheaf axiom that
\begin{equation}
\cC(s_{i_1} \ldots s_{i_w}) = \cC(s_{i_1}) \times_{\cC(\equiv)} \cC(s_{i_2}) \times_{\cC(\equiv)} \cdots \times_{\cC(\equiv)} \cC(s_{i_w})
\end{equation}
and similarly for $\overline{\cC}$, $\cC_r$, and $\overline{\cC}_r$. 
And likewise, for moduli spaces, 
\begin{equation} \label{modprod} \mM_r(s_{i_1} \ldots s_{i_w}) = \mM_r(s_{i_1}) \times_{\mM_r(\equiv)} \mM_r(s_{i_2}) \times_{\mM_r(\equiv)} \cdots \times_{\mM_r(\equiv)} \mM_r(s_{i_w})\end{equation}
and likewise for $\overline{\cM}_r$.

\begin{remark} The above moduli spaces are Artin stacks, and the fiber products should be understood in the sense of such stacks.  For a fixed $r$, it is possible to work equivariantly with schemes instead by framing appropriately.  
\end{remark}

To calculate the $\cC_r$ and $\cM_r$ in general, it now suffices to determine these for the trivial braid and one-crossing braids
(and to understand the maps between these).  As special cases of Proposition \ref{prop:bos}, we have: 

\begin{corollary}  $\cC_r(\equiv_n)$ is the subcategory of representations of the $A_n$ quiver $\bullet \to \bullet \to \cdots \to \bullet$
which take the $k$'th vertex to a $rk$ dimensional vector space, and all arrows to injections.  
Writing $\mathrm{P}_{r,n} \subset \mathrm{GL}_{rn}$ for the group of  $r \times r$ block upper triangular matrices, 
$\mM_r(\equiv_n) = \mathit{pt} / \mathrm{P}_{r,n}$. 
\end{corollary}

\begin{corollary} \label{cor:bs}
Let $s_i$ be the interchange of the $i$ and $i+1$ st strands in $Br_n$.  
An object of $\overline{\cC}_r(s_i)$ is determined by two flags 
$L_\bullet k^{\oplus rn}$ and $R_\bullet k^{\oplus rn}$ such that 
$\dim L_k / L_{k-1} = r = \dim R_k / R_{k-1}$ and $L_k = R_k$ except possibly
for $k = i$.  Moreover, for such an object, the following are equivalent: 
\begin{itemize}
\item $L_{i-1} = L_i \cap R_i = R_{i-1}$
\item $L_{i+1} = L_i + R_i = R_{i+1}$
\item $F \in \cC_r(s_i)$
\end{itemize}
Two such pairs $(L_\bullet, R_\bullet)$ and $(L'_\bullet, R'_\bullet)$ are isomorphic
if and only if there's a linear automorphism of $k^{rn}$ carrying one to the other,
and moreover all isomorphisms arise in this manner. 
\end{corollary} 
\begin{proof}
The correspondence with the description in Proposition \ref{prop:bos} comes from framing 
the stalk in the top-most region of the braid by fixing an isomorphism with
$k^{\oplus rn}$, and then identifying all other stalks with their image under the injections
into this vector space.  
\end{proof}

\begin{remark}
Our moduli spaces
$\cM_n(\beta)$ appear explicitly in the work of Brou\'e and Michel \cite{BrMi}, where they are called
$\cB(\beta)$.  They are sometimes called {\em open Bott-Samelson varieties}; the spaces
 $\overline{\mM}_r(\beta)$  are the corresponding closed Bott-Samelson varieties.  
 It was shown by Deligne \cite{D-braids} that the association
$\beta \mapsto \cB(\beta)$ gives a categorical representation of the positive braids.  

Here we have seen that, in type A, 
these spaces arise as moduli of objects in the Fukaya category which end on the given braid.
We note that in our presentation, all the {\em data} of a categorical positive braid representation 
(i.e. the higher homotopies etc.) are automatically present because isotopies of positive
braids can be chosen to be Legendrian isotopies, for which all desired categorical
data is furnished by (appropriate family versions of) Theorem \ref{thm:where-it-happens}. 
That is, for us, $\beta \mapsto \mM_r(\beta)$ was a categorical braid invariant {\em for a priori
geometric reasons},
which subsequently we checked combinatorially to be equivalent to a classical construction.
\end{remark}

We can similarly calculate the moduli space for the cylindrical closure of a braid: 

\begin{definition}
For $\beta$ a positive braid, we write $\beta^\circ$ for the legendrian knot in 
$T^\infty (S^1_x \times \RR_z)$ whose front diagram in the annulus 
is obtained by gluing the $x = 0$ and $x=1$ boundaries of the front plane. 
\end{definition}

Note this gluing creates no vertical (i.e., parallel to the $z$-axis)  tangents. 
The gluing data for taking a sheaf on strip to a sheaf on the cylinder is just a choice of isomorphism between the restriction to the right boundary
and the restriction to the left.  Thus

\begin{eqnarray} \label{circ} 
\cC_r(\beta^\circ) & = & \cC_r(\beta) \times_{\cC_r(\equiv) \times \cC_r(\equiv) } \cC_r(\equiv) \\
\mM_r(\beta^\circ) & = & \mM_r(\beta) \times_{\mM_r(\equiv) \times \mM_r(\equiv) } \mM_r(\equiv)
\end{eqnarray}

\subsection{Braid moduli as correspondences}
\label{sec:cylbraidclosure}

Our sheaves live on strips or cylinders decorated by braids.  Specifying a smaller strip or cylinder,
and taking the correspondingly restricted braid, induces a restriction morphism on moduli spaces. 
We are interested here in morphisms and diagrams of morphisms which can be constructed in this 
way, and especially in the use of these as convolution kernels. 

\vspace{2mm} \noindent {\bf Convention:}
We will be interested in the derived categories of constructible sheaves {\em on our moduli spaces
$\cM_r(\Lambda)$ of constructible sheaves}.  In the cases of interest here, where $\Lambda$ 
is a braid, the space 
$\cM_1(\Lambda)$ will always be an algebraic Artin stack over $k$, and we will be interested in 
its derived category of sheaves constructible with respect to {\em algebraic} stratifications.  
For an Artin stack $\mathfrak{A}$, we write $D(\mathfrak{A})$ for its derived category of constructible sheaves; 
the relevant foundations are worked out in \cite{LO2, LO3}.  In fact, all our stacks are global quotient stacks,
so the reader could use the equivalent notion of equivariant derived category \cite{BernsteinLunts} instead.  

We will sometimes want to use the mixed structure on the derived category; we will indicate as much
by writing $D_m(\mathfrak{A}).$
\vspace{2mm}

We first discuss a simpler moduli problem to which ours maps. 
Let $G$ be a semisimple algebraic group, and let $P \subset G$ be a parabolic subgroup.  Let $X$ be a compact one-manifold with boundary and carrying marked points, such that every point
on the boundary is marked.  We write $\mathrm{Bun}_{G, P}(X)$ for the moduli space of $G$ bundles on $X$, with reduction of structure
at the marked points to $P$.  There are natural maps given by cutting at or forgetting a marked point in the interior. 
$$  \mathrm{Bun}_{G, P}( \mathrm{Cut}_p(X))  \xleftarrow{cut} \mathrm{Bun}_{G, P} (X) \xrightarrow{forget}
\mathrm{Bun}_{G,P}(\mathrm{Forget}_p(X))$$ 
Here, if $S\subset X$ is the set of marked points and $p\in S$, then
$\mathrm{Forget}_p(X)$ 
is just $X$ as a space with $S \setminus p$ as the marked points.  By 
$\mathrm{Cut}_p(X)$ we mean the space where the interior point $p$ is replaced by two boundary points
$p', p''$.

We use the constant sheaf on $\mathrm{Bun}_{G, P} (X)$ as the kernel for an integral transform: 
\begin{eqnarray*}
D( \mathrm{Bun}_{G, P}( \mathrm{Cut}_p(X))) & \to & D(\mathrm{Bun}_{G,P}(\mathrm{Forget}_p(X))) \\
\cF & \mapsto & forget_! cut^* \cF
\end{eqnarray*} 
Note that $cut$ is a smooth map with fibre $P$, and $forget$ is a proper map with fibre $G/P$ (this is
why it is important that $P$ is parabolic), so this transform preserves purity.

The main cases of interest 
are when $X = \barbarbell$ or $X = \boing$.  We use the first

$$ \mathrm{Bun}_{G,P}(\barbell) \times \mathrm{Bun}_{G,P}(\barbell) \xleftarrow{cut}
\mathrm{Bun}_{G,P}(\barbarbell) \xrightarrow{forget} \mathrm{Bun}_{G,P}(\barbell)$$
to define a convolution product:  
\begin{eqnarray*}
\ast: D( \mathrm{Bun}_{G,P}(\barbell) ) 
\times D(\mathrm{Bun}_{G,P}(\barbell)) & \to & D(\mathrm{Bun}_{G,P}(\barbell)) \\
(\cF, \cG) & \mapsto & forget_! cut^* (\cF \otimes \cG) 
\end{eqnarray*}

In the second, it gives Lusztig's horocycle correspondence \cite{Lu}: 
$$ \mathrm{Bun}_{G,P}(\barbell) \xleftarrow{cut}
\mathrm{Bun}_{G,P}(\boing) \xrightarrow{forget} \mathrm{Bun}_{G,P}(\bigcirc)$$
which allows us to define the horocycle morphism
\begin{eqnarray*}
\mathscr{H}: D( \mathrm{Bun}_{G,P}(\barbell) )& \to & D( \mathrm{Bun}_{G,P}(\bigcirc) \\
\cF & \mapsto & forget_! cut^* \cF
\end{eqnarray*}

A point $p \in X$ determines a map to $pt/P$ which forgets everything except the structure at that point.  
Gluing is given by fibre products:
$$ \mathrm{Bun}_{G, P}( X/\{p = q\} ) = \mathrm{Bun}_{G,P}(X) \times_{pt/P \times pt/P} pt/P $$
$$ \mathrm{Bun}_{G,P}( (X, p) \cup_{p=q} (Y,q)) = \mathrm{Bun}_{G, P}(X) \times_{pt / P} \mathrm{Bun}_{G,P}(Y)$$

In particular, 

$$ \mathrm{Bun}_{G, P}( \boing ) = \mathrm{Bun}_{G,P}(\barbell) \times_{\bullet/P \times \bullet/P} \bullet/P $$
$$ \mathrm{Bun}_{G,P}( \barbarbell ) = \mathrm{Bun}_{G, P}(\barbell) \times_{\bullet/ P} \mathrm{Bun}_{G,P}(\barbell)$$

We return to the discussion of positive braids.  Fix a microlocal rank $r$, let $G = \mathrm{GL}_{rn}$ and $P = \mathrm{P}_{r,n}$, 
we write $\cM$ for $\cM_r$. 
Let $\beta$ be a positive braid in a strip; there is a natural map 
$\pi: \cM(\beta) \to \mathrm{Bun}_{G,P}(\barbell)$: the $G$ bundle structure is visible at the top of the strip,
and the reduction to $P$ bundles is imposed by the flags at the left and right.  Likewise if 
$\beta | \beta'$ is a diagram where we have in a strip $\beta$ on the left, $\beta'$ on the right, and we
imagine a line separating them (passing through no crossings), we have a map
$\cM(\beta | \beta') \to \mathrm{Bun}_{G,P}(\barbarbell)$.   We have similarly 
$\overline{\pi}_{\beta}: \overline{\cM}(\beta) \to  \mathrm{Bun}_{G,P}(\barbell)$ and
$\overline{\pi}_{\beta|\beta'}: \overline{\cM}(\beta|\beta') \to  \mathrm{Bun}_{G,P}(\barbarbell)$.

\begin{proposition} \label{prop:conv}
We have $\pi_{\beta*} \bQ \ast \pi_{\beta'*} \bQ = \pi_{\beta \beta' *} \bQ$, and similarly
$\overline{\pi}_{\beta*} \bQ \ast \overline{\pi}_{\beta'*} \bQ = \overline{\pi}_{\beta \beta'*} \bQ$
\end{proposition}
\begin{proof}
Equation (\ref{modprod}) is the assertion that
the
left hand square in the following diagram is Cartesian, and 
(our imaginary line had no power to change the moduli space) the upper right horizontal arrow is an isomorphism: 
$$ 
\xymatrix{
\cM(\beta) \times \cM(\beta') \ar[d]^{\pi \times \pi} & \cM(\beta | \beta') \ar[d]^{\pi} \ar[l] \ar[r]^{\sim} & \cM(\beta \beta') \ar[d]^{\pi} \\
\mathrm{Bun}_{G,P}(\barbell)  \times  \mathrm{Bun}_{G,P}(\barbell) & 
\mathrm{Bun}_{G,P}(\barbarbell) \ar[r]^{forget} \ar[l]_{\,\,\,\,\,\,\,\,\,\,\,\,\,\,\,\,\,\,\,\,\,cut} & \mathrm{Bun}_{G,P}(\barbell)
}
$$
Now use smooth base change on the left hand square and commutativity on the right.  The statement for
$\overline{\pi}$ follows by an identical argument. 
\end{proof}

Let $\beta$ be a braid, $\beta^\circ$ its cylindrical closure, and $\beta|^\circ$ its cylindrical closure with
an imaginary line drawn where the braid was closed up.  There are natural maps
$\pi_{\beta^\circ}: \mM(\beta^\circ) \to \mathrm{Bun}_{G,P}(\bigcirc)$ and
$\pi_{\beta|^\circ}: \mM(\beta|^\circ) \to \mathrm{Bun}_{G,P}(\boing)$, and similarly 
$\overline{\pi}_{\beta^\circ}: \overline{\mM}(\beta^\circ) \to \mathrm{Bun}_{G,P}(\bigcirc)$ and
$\overline{\pi}_{\beta|^\circ}: \overline{\mM}(\beta|^\circ) \to \mathrm{Bun}_{G,P}(\boing)$.  
We have:

\begin{proposition} \label{prop:hor}
We have $\mathscr{H}(\pi_{\beta*} \bQ)  = \pi_{\beta^\circ *} \bQ$, and similarly
$\mathscr{H}(\overline{\pi}_{\beta*} \bQ)  = \overline{\pi}_{\beta^\circ *} \bQ$
\end{proposition}
\begin{proof}
The
left hand square in the following diagram is Cartesian, and the upper right horizontal arrow is an isomorphism: 
$$ 
\xymatrix{
\cM(\beta)  \ar[d]^{\pi } & \cM(\beta |^\circ ) \ar[d]^{\pi} \ar[l] \ar[r]^{\sim} & \cM(\beta^\circ ) \ar[d]^{\pi} \\
 \mathrm{Bun}_{G,P}(\barbell) & 
\mathrm{Bun}_{G,P}(\boing) \ar[r]^{forget} \ar[l]_{\,\,\,\,cut} & \mathrm{Bun}_{G,P}(\bigcirc)
}
$$
Now use smooth base change on the left hand square and commutativity on the right.  The statement
for $\overline{\pi}$ follows by an identical argument. 
\end{proof}

There are (at least) two standard presentations of the spaces involved we have been using: 
$$ 
\xymatrix{
P \backslash G / P  \ar@{=}[d] & P^{ad} \backslash G \ar@{=}[d] \ar[l] \ar[r] & G \backslash G \ar@{=}[d]\\
 \mathrm{Bun}_{G,P}(\barbell) & 
\mathrm{Bun}_{G,P}(\boing) \ar[r]^{forget} \ar[l]_{\,\,\,\,cut} & \mathrm{Bun}_{G,P}(\bigcirc) \\
G\backslash (G/P \times G/P)  \ar@{=}[u] & G \backslash (G/P \times G) \ar@{=}[u] \ar[l] \ar[r] & G \backslash G \ar@{=}[u]\\
}
$$

In the above, $P$ acts on the left or right as described, except for the upper middle where we have written
$P^{ad}$ to indicate it acts by conjugation.  On the other hand $G$ acts on $G/P$ by left multiplication and on $G$
by conjugation.  The middle row gives the most ``coordinate free'' description, in which terms we restate Corollary
\ref{cor:bs}:

\begin{proposition}
\label{prop:modbs}
The map $\overline{\cM}(s_i) \to \mathrm{Bun}_{G,P}(\barbell) = G\backslash (G/P \times G/P)$ is injective.  Its image
is the locus of pairs of flags  $(F^{\ell},F^r)$ which are identical except possibly at the $i$-th position, i.e. if we write
\[
\overline{BS}_i =  \{ (0 \subset F^{\ell}_1 \subset \cdots \subset F^\ell_{n-1} \subset V, 0 \subset F^r_1 \subset \cdots \subset F^r_{n-1} \subset V) | 
 \, F^{\ell}_j = F^r_j \mbox{ for } j \ne i\}
 \]  
 then $\overline{\cM}(s_i) = G \backslash \overline{BS}_i$.  Likewise if we write 
  $BS_i$ of flags where $F^\ell_i \cap F^r_i = F_{i-1}$, then
 the image of $\cM(s_i) = G \backslash BS_i$.
\end{proposition}

$BS_i$ is what is called in \cite{WW-HOMFLY} a Bott-Samelson variety. 
Figure \ref{fig:bs} below illustrates the correspondence between points in the Bott-Samelson variety and sheaves on a front diagram:
\begin{figure}[H]
\includegraphics[scale = .3]{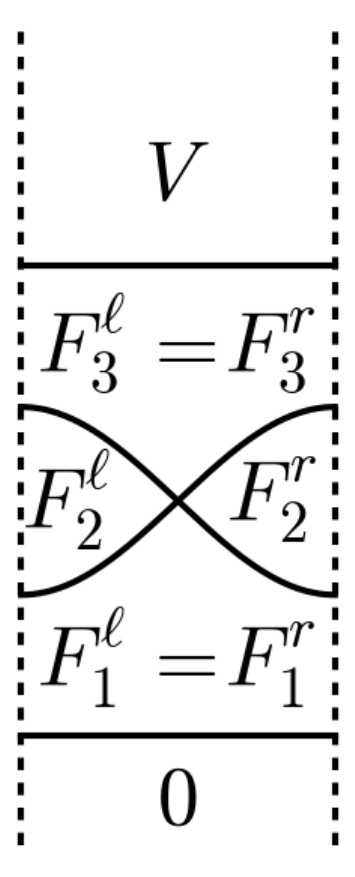}
\caption{A point of the Bott-Samelson lies over a pair of flags}
\label{fig:bs}
\end{figure}

\begin{remark}
\label{rmk:modbs}
Writing $s_{i_1} \ldots s_{i_w}$ as $| s_{i_1} | s_{i_2} | \ldots  | s_{i_w} |$ and taking the ``bottom row'' point of view leads to the description
$$\cM_r(s_{i_1} \ldots s_{i_w}) = G \backslash \{(F^1, F^2, \ldots, F^{w+1})  \in (G/P)^{w+1} | (F^j, F^{j+1}) \in BS_{i_j} \} $$
Similarly, writing $s_{i_1} \ldots s_{i_w}^\circ$ as $| s_{i_1} | \ldots | s_{i_w} |^\circ$ leads to the description
$$\cM_r(s_{i_1} \ldots s_{i_w}^\circ) = G \backslash \{ (F^1,F^2,\ldots, F^{w+1}, g) \in (G/P)^{w+1} \times G \, | \, (F^j, F^{j+1}) \in BS_{i_j} \, \& \, gF^{w+1} = F^1\}
$$
\end{remark}

\begin{example}
\label{ex:trivialbraid}
For the trivial braid and in rank 1, we have 
$$
\xymatrix{
\cM_1(\equiv_n^\circ) \ar[d]^{\pi} \ar@{=}[r] & G \backslash \{(F, g) \in G/B \times G\, | \, gF = F\}   \ar[d]^{\pi} \\
\cM_1(\bigcirc) \ar@{=}[r] &  G \backslash G 
}
$$
The right hand side is the Grothendieck simultaneous resolution of $G = \mathrm{GL}_n$.
\end{example}

\begin{remark}
The horocycle correspondence originates in the work of Lusztig on character sheaves \cite{Lu}, where the 
$\overline{\cM}_n(\beta)$ also appears explicitly. In particular, as
observed in \cite{WW-traces}, $\mathscr{H}(\overline{\pi}_{\beta*} \bQ) =  \overline{\pi}_{\overline{\beta} *} \bQ$ is a character
sheaf by definition, and in fact is a cuspidal character sheaf.  
\end{remark}

\subsection{HOMFLY homology}

There are many variants of Khovanov homology; 
we are interested here in the triply graded version constructed in \cite{KhR, Kh}, which categorifies
 the HOMFLY polynomial in the same sense that the original Khovanov homology categorifies the Jones
 polynomial.  This ``HOMFLY homology'' was predicted by the physicists
 Gukov, Schwarz, and Vafa \cite{GSV}; mathematically it is constructed from the
 braid group categorifications of Deligne \cite{D-braids} and Rouquier \cite{Ro}. 
 Williamson and Webster reinterpreted their construction in the language of geometric 
 representation theory, formulating it in terms of certain character sheaves arising from the 
 horocycle correspondence \cite{WW-Hochschild, WW-traces, WW-HOMFLY}.  

We have seen above that
all these ingredients emerge naturally from the consideration of moduli of sheaves
microsupported on braids.  This leads to a description of the HOMFLY 
homology in terms of these moduli.  The precise statement is somewhat complicated,
and requires the use of a formalism of mixed complexes of sheaves and their weight filtrations
(a notion we briefly recall and give references for in Section \ref{subsubsec:weight-formalism}).

The meaning of the theorem below should not be obscured by the proliferation of weights and 
spectral sequences and the 
combinatorics of this section.  Although we have used diagrams and constructible sheaves on the plane to compute them, 
neither the moduli space $\cM_n(\beta^\circ) $ nor the map $\cM_n(\beta^\circ) \to G \backslash G$ need ever mention
a knot diagram.  Indeed, the space may be defined as the moduli of `rank one' objects in the Fukaya category which have singular
support in the knot, and the map amounts to looking at the local system at infinity.  The significance of the theorem
is that it produces the HOMFLY homology from the geometric data of the knot itself, by a geometric construction.

\begin{theorem}  \label{thm:HHH} Let $\beta^\circ$ be the cylindrical closure of a positive braid.   
As in Proposition \ref{prop:hor}, consider the map induced by restriction to the top of the 
cylinder: 
$$\pi_{\beta^\circ}: \cM_n(\beta^\circ) \to \mathrm{Bun}_{G,P}(\bigcirc) = G\backslash G$$ 
Let $W$ be any weight filtration on
$\pi_{\beta^\circ *} \bQ$.
The $E_1$ hypercohomology spectral sequence 
$$E^{ij}_1 = H^{i+j}(G \backslash G, \mathrm{Gr}_{-i}^W \pi_{\beta^\circ *} \bQ) \Rightarrow H^{i+j}(G\backslash G, \pi_{\beta^\circ*} \bQ)$$ 
is, up to a shift of gradings, identical to the spectral sequence constructed in \cite{WW-HOMFLY}. 
In particular its $E_2$ page is the colored-by-$n$ triply graded HOMFLY homology.
\end{theorem}

If we take \cite{WW-HOMFLY} as the definition of triply graded homology, the proof
of this theorem is almost immediate from Propositions \ref{prop:hor} and \ref{prop:conv}. 
However, we will take \cite{Kh} as the definition of triply graded homology --- which we 
choose to do here because \cite{Kh} is by far the most elementary and explicit account, 
and even more importantly,
has explicit normalization conventions.  In this case, the proof of this Theorem can be 
extracted from \cite[\S 6.3]{WW-HOMFLY}, via Propositions \ref{prop:hor} and \ref{prop:conv}.  
However, \cite{WW-HOMFLY} can be difficult to navigate, and moreover has nowhere explicit
normalization conventions.  Thus we give in this subsection a complete account 
of the necessary ideas from \cite{WW-HOMFLY}, leading up to a proof of the above theorem.

 \label{sec:HOMFLY}

Our goal in Sections \ref{subsubsec:weight-formalism}--\ref{subsubsec:Khovanov} is to provide a crash course in triply graded knot homology: the construction of  Williamson and Webster, 
its relation to the claim we attribute to them in the body of the text,
its relation to the `Soergel bimodule' construction of Khovanov \cite{Kh}, and some sample computations.

\subsubsection{Weights}
\label{subsubsec:weight-formalism}

In algebraic geometry, the formalism of constructible sheaves has a richer variant.  Working either with mixed $\ell$-adic complexes on a variety defined over a finite field \cite{SGA45, Weil2}, or with mixed Hodge modules on a variety defined over the complex numbers \cite{Sai}, one has a triangulated category $D_m(X)$ of ``mixed complexes,'' along with a forgetful functor to the usual derived category of sheaves.

For each $k \in \ZZ$, there is a subcategory $P_k(X) \subset D_m(X)$ of pure complexes of weight $k$.  
  Shifting the homological degree changes the weight, $P_k(X)[1] = P_{k+1}(X)$;
there is a Tate twist denoted $(1)$ which 
 shifts weights without changing the homological degree: $P_k(X)(1) = P_{k-2}(X)$. 
We write $\langle 1 \rangle := [2](1)$ for the weight-preserving shift-twist.  
The constant sheaf on a smooth variety placed  in cohomological degree zero is
pure of weight zero --- we will denote it by $\bQ$, though the reader who prefers the theory of $\ell$-adic sheaves should replace this symbol with $\overline{\bQ}_{\ell}$.

The word `mixed' means that objects $K \in D_m(X)$ admit {\em weight filtrations}, i.e.  
$$ \cdots \to W_{i-1} K \to W_i K \to W_{i+1} K \to \cdots $$  
so that $\mathrm{Gr}_k^W:= \mathrm{Cone}(W_{k-1} \to W_k) \in P_k(X)$, $W_{\gg 0} = K$ and $W_{\ll 0} = 0$.  
When $K$ admits a weight filtration whose subquotients all have weight $\ge n$, one says
$K$ has weight $\ge n$.  Every map between complexes which comes from a map of (cohomologies of) algebraic
varieties is strictly compatible with weight filtrations.
 The basic fact about weight filtrations \cite{Weil2} is that star-pushforward
preserves the property of having weight $\ge n$.  

Weights interact well with perverse sheaves \cite{BBD}.  
When $K$ is (shifted) perverse, there is a unique weight filtration by (shifted) perverse subsheaves, which in this
case we prefer and term `the' weight filtration.  A fundamental result (the `decomposition theorem') 
is that objects in $\bigoplus P_k$ are all direct sums of shifted simple perverse sheaves.

For any complex of sheaves $K$ with any filtration $W$, there is a spectral sequence 
$$E^{ij}_1 = H^{j}(X, \Gr_{-i}^W K[i]) \Rightarrow H^{i+j}(X, K)$$
We recall the construction of the first differential.  There is an exact triangle 
\begin{equation} \label{eqn:chr} 
\Gr_i^W K \to W_{i+1} / W_{i-1} K \to \Gr_{i+1}^W K \xrightarrow{[1]}
\end{equation}
and in particular a boundary morphism $\Gr_{i+1}^W K \to \Gr_i^W K [ 1]$;
one can show 
that composing these morphisms forms a complex.
$$ \cdots \to \mathrm{Gr}_{-i+1}^W K [i-1] \to \mathrm{Gr}_{-i}^W K [i] \to \Gr_{-i-1}^W K [i+1] \to \cdots$$
The degree in the complex is the number in brackets. 
By definition the first differential in the above spectral sequence is induced by this map:
$$d_1: E^{i,j}_1 = H^{j}(X, \Gr_{-i}^W K[i]) \to H^{j}(X, \Gr_{-i-1}^W K [i+1]) = E^{i+1,j}_1$$

Webster and Williamson prove that different choices of the weight filtration $W$ yield homotopic complexes
of the form (\ref{eqn:chr}), thus defining the chromatography\footnote{
According to Wikipedia, ``Chromatography 
is the collective term for a set of laboratory techniques for the separation of mixtures.''}  map: 
$$\mathrm{Chr}: D_m(X) \to \mathrm{Kom} \,\, P_0(X)$$

\vspace{2mm} \noindent {\bf Notation}. We define $[i]$, $(i)$, $\langle i \rangle$ on 
$\mathrm{Kom}\, \, P_0(X)$ so that these symbols commute with Chr.  Explicitly,
for  $K^\bullet \in \mathrm{Kom} \,\, P_0(X)$, we write 
$K[i]$ for the complex whose terms are $K[i]^j = K^{i+j}$; i.e. the square bracket
shifts the external degree and does not affect the internal degree (which could anyway
not shift without having the terms leave $P_0(X)$).  We write
$K\langle i \rangle^\bullet$ for the complex whose terms are 
$K\langle i \rangle^j = K^i [2j](j) = K^i \langle j \rangle$; that is, it shift-twists interally and leaves
the external degree alone.  Finally we write
$K^\bullet(i) := K^\bullet \langle i \rangle [-2i]$.   \vspace{2mm}

\subsubsection{Representations of the braid group via convolution kernels}

As we recalled in \ref{sec:cylbraidclosure}, given the data of a semisimple group 
and a parabolic subgroup, there is a diagram 

$$ \mathrm{Bun}_{G,P}(\barbell) \times \mathrm{Bun}_{G,P}(\barbell) \xleftarrow{cut}
\mathrm{Bun}_{G,P}(\barbarbell) \xrightarrow{forget} \mathrm{Bun}_{G,P}(\barbell)$$

\noindent and a corresponding convolution product on the sheaf categories

\begin{eqnarray*}
\ast: D( \mathrm{Bun}_{G,P}(\barbell) ) 
\times D(\mathrm{Bun}_{G,P}(\barbell)) & \to & D(\mathrm{Bun}_{G,P}(\barbell)) \\
(\cF, \cG) & \mapsto & forget_! cut^* (\cF \otimes \cG) 
\end{eqnarray*}

We wish to be more explicit about the spaces 
and categories which arise.  For the spaces, we use the
identification $\mathrm{Bun}_{G,P}(\barbell)  \cong P \backslash G / P$.  
By $P \backslash G / P$, we really mean the quotient of $G$ by $P \times P$ 
acting on the left by 
$(b_1, b_2)\cdot g = b_1 g b_2^{-1}$ -- note the restriction to the diagonal $P$ gives
the conjugation action $b \cdot g = b g b^{-1}$.

Recall that $cut$ is smooth (with fibres $P$) and $forget$ is proper
(with fibres $G/P$), 
so if we are working with mixed sheaves, convolution preserves the pure sheaves, weights, etc.

Let us restrict attention to the case where $P= B$ is a Borel subgroup.  
Recall that the 
$B \times B$ orbits
on $G$ are enumerated by the Weyl group (in the case relevant here, $G = GL_n$ and the Weyl group is 
the group of permutations); we write $G_w$ for the orbit corresponding to $w \in W$.  
The inclusion of the orbit $G_w$ induces an inclusion $j_w: B\backslash G_w / B \to B \backslash G/ B$.
Let us then fix notation for certain sheaves:

$$ \mathcal{T}_w := j_{w *} \bQ \,\,\,\,\,\,\,\,\,\,\,\,\,\,\,\,\,\,\,\,\,\,\,\,\,\,\,\,\, \mathcal{T}_w^{-1} := j_{w !} \bQ \langle length(w) \rangle$$

Rouquier \cite[\S 5]{Ro} shows there is a representation of the braid group of $G$ (i.e., the usual
braid group when $G = GL_n$)  on $D(B \backslash G / B)$
determined by $w^{\pm} \mapsto \ast \mathcal{T}_w^{\pm}$.  
In particular, 
given a braid $\beta = s_{i_1}^\pm \ldots s_{i_k}^{\pm}$, the mixed complex
$\mathcal{T}_\beta := \mathcal{T}_{i_1}^{\pm} \ast \cdots \ast \mathcal{T}_{i_k}^{\pm}
\in D(B \backslash G / B)$ only depends on the braid.  

Henceforth we restrict ourselves to the case $G = \GL_n$ and $B$ the upper triangular matrices. 
Recall that we have defined a morphism $\pi_\beta: \cM_1(\beta) \to \mathrm{Bun}_{G,B}(\barbell)$
by restriction of the sheaf to the boundary of the front strip.  

\begin{proposition} \label{prop:rrr} 
We have $\mathcal{T}_\beta \cong \pi_{\beta*}\bQ$
\end{proposition}
\begin{proof}
It follows from the calculations 
in \ref{subsec:braidmoduli} that, for an elementary reflection $\beta = s$ we have a natural 
identification $\pi_{s} = j_s$, hence in particular, $j_{s*} \bQ \cong \pi_{s*} \bQ$.  
The desired assertion now follows from Proposition \ref{prop:conv}.   
\end{proof}

As we have previously remarked, the fact that $\beta \mapsto \pi_{\beta*}\bQ$ gives a representation
of the positive braid monoid follows {\em without calculation} from the invariance of the underlying categories
of sheaves microsupported on braids under Legendrian isotopy --- calculations are only required to then identify
it with the representation of Rouquier.  However, we do not see from this perspective
any reason why the representation should extend to the braid group.

\subsubsection{From geometric representation theory to algebra} 
Here we explain in detail how to pass from
the geometric representation theory formulation of HOMFLY homology in 
\cite{WW-Hochschild, WW-HOMFLY, WW-traces} to the algebraic formulation in terms of Soergel bimodules
of \cite{Kh}.  We {\em do not} use the shift conventions of \cite{WW-HOMFLY}, avoiding especially the `half Tate twist'.
Let us establish the following notation.

\vspace{2mm} \noindent {\bf Notation}.  For $A, B$ abelian categories and $T: A \to B$ a functor, 
we write $\underline{T}: \mathrm{Kom}(A) \to \mathrm{Kom}(B)$ for the induced functor on the categories of
complexes up to homotopy.  We use the same notation for complexes of Soergel bimodules in Section \ref{subsubsec:Khovanov}.
\vspace{2mm}

With this in hand, we will give an argument for the following assertion:

\begin{theorem} \label{thm:WWH} \cite{WW-HOMFLY} Let $\beta$ be a braid. 
Let $B \to B \times B$ be the inclusion of the diagonal
and $I: B^{ad} \backslash G \to B \backslash G / B$ the corresponding quotient map.
The $E_1$ hypercohomology spectral sequence 
$$E^{ij}_1 = H^{i+j}(B^{ad} \backslash G, \mathrm{Gr}_{-i}^W I^* \mathcal{T}_\beta) \Rightarrow H^{i+j}(B^{ad} \backslash G, I^* \mathcal{T}_\beta)$$ determined by the complex
$I^* \mathcal{T}_\beta$ and any weight filtration $W$ does not depend on the weight filtration
after the first page, and depends, up to an overall shift of the gradings, only on the 
topological knot type of the braid closure.  

Each term in the spectral sequence itself carries a weight filtration,
which we denote $\Omega$.  The triply graded vector space
$$\mathbb{H}^{ij}_k(\beta) := \mathrm{Gr}_{j+k} E_2^{ij} =
\mathrm{Gr}^\Omega_{j+k} H^i(\underline{H}^j(B^{ad} \backslash G, \mathrm{Chr}(I^* \mathcal{T}_\beta )))
$$ 
coincides with the triply graded HOMFLY homology of \cite{KhR, Kh}.
\end{theorem}

A proof of this statement is given in \cite[\S 6.3]{WW-HOMFLY}, using
\cite{WW-Hochschild}.  Here we give an expanded treatment, specialized to the case of 
the usual (uncolored) HOMFLY homology, for the purpose of 
spelling things out in enough detail to explicitly match gradings. 
\vspace{2mm}

The basic calculation is the following: 

\begin{proposition}
The pullback and Gysin pushforward for the inclusion $G_1 \to \overline{G_{s_i}}$ induce:
\begin{eqnarray*}
\mathrm{Chr}(\mathcal{T}_i) & = & \bQ_{G_1} \langle -1 \rangle \to \bQ_{\overline{G_{s_i}}} \\
\mathrm{Chr}(\mathcal{T}_i^{-1}) & = & \bQ_{\overline{G_{s_i}}} \langle 1 \rangle 
\to \bQ_{G_1}\langle 1 \rangle  
\end{eqnarray*}
where in each case $\bQ_{\overline{G_{s_i}}}$ is the degree zero term of the chromatographic complex.
\end{proposition}
\begin{proof}
In the category of (appropriately shifted) perverse sheaves, there are exact sequences
$$ 0 \to \bQ_{\overline{G_{s_i}}} \to j_{s_i *}  \bQ_{G_{s_i}}  \to \bQ_{G_1}[-1](-1) \to 0 $$
$$ 0 \to \bQ_{G_1}[-1]  \to j_{s_i !}  \bQ_{G_{s_i}} \to \bQ_{\overline{G_{s_i}}}\to 0 $$
The weight filtrations are visible in these sequences.  For $\mathcal{T}_i =j_{s_i *}  \bQ_{G_{s_i}}$ , we have 
$W_{-1}  \mathcal{T}_i = 0$, $W_0 \mathcal{T}_i = \bQ_{\overline{G_{s_i}}}$, and
$W_1 \mathcal{T}_i = \mathcal{T}_i$. 
Indeed, then $\mathrm{Gr}_0^W \mathcal{T}_i = \bQ_{\overline{G_{s_i}}}$
and $\mathrm{Gr}_1^W \mathcal{T}_i = \bQ_{\overline{G_1}}[-1](-1) $, which shows
the terms in Chr are as advertised. 

Similarly for $\mathcal{T}_i^{-1} =  j_{s_i !}  \bQ_{G_{s_i}} \langle 1 \rangle$, 
we have $W_{-2} \mathcal{T}_i^{-1} = 0 $, $W_{-1} \mathcal{T}_i^{-1} =
\bQ_{G_1}[-1] \langle 1 \rangle$, and $W_0 \mathcal{T}_i^{-1} = \mathcal{T}_i^{-1}$. 
Then $\mathrm{Gr}_{-1}^W  \mathcal{T}_i^{-1} = \bQ_{G_1}[-1] \langle 1 \rangle$
and $\mathrm{Gr}_0^W\mathcal{T}_i = \bQ_{\overline{G_{s_i}}} \langle 1 \rangle$.  This again
shows the terms in Chr are as advertised.

We leave
the statement about the maps in Chr to the reader.
\end{proof}

Both convolution and the smooth pullback $I^*$ preserve weights and thus commute with
Chr. 
That is, 

$$\mathrm{Chr}(I^* ( \mathcal{T}_{i_1}^\pm \ast \cdots \ast \mathcal{T}_{i_t}^\pm ) ) = 
\underline{I}^* \left( \mathrm{Chr}(\mathcal{T}_{i_1}^{\pm} ) \underline{\ast} \cdots \underline{\ast} 
 \mathrm{Chr}(\mathcal{T}_{i_t}^{\pm}) \right)
$$

As is our standing notation, $\underline{I}^*$ just means you pull back the complex term by term, and
$\underline{\ast}$ means the monoidal structure on $\mathrm{Kom}(P_0)$ induced by $\ast$,
i.e., make the double complex of every possible convolution of a term in one complex
by a term in the other, then totalize.  

Consider the fibre product diagram

$$\xymatrix{
B^{ad} \backslash G \ar[r]^{I} \ar[d]^{\pi_B} & B \backslash G / B \ar[d]^{\pi_{B \times B}}  \\
  B \backslash \bullet \ar[r]^{\Delta} &(B \times B) \backslash \bullet
}$$

By smooth base change, 
$H^*(B^{ad} \backslash G, I^* \cF) = 
H^*(B \backslash \bullet, \pi_{B*} I^* \cF) 
= H^*(B \backslash \bullet, \Delta^* \pi_{B \times B*}  \cF) $.  
Similarly,  $\pi_{B \times B} (\cF \ast \cG) 
= (\pi_{B \times B} \cF) \ast (\pi_{B \times B} \cG)$, the latter convolution
being convolution over a point.   
We find: 
$$
 \underline{H}^*(B^{ad} \backslash G, \mathrm{Chr}(I^* \mathcal{T}_\beta ) ) 
= \underline{H}^*\left(B^{ad} \backslash \bullet,
\underline{\Delta}^* \left(  (\underline{\pi_{B \times B*}} \mathrm{Chr}
\mathcal{T}_{i_1}^\pm) \underline{\ast} \cdots \underline{\ast} (\underline{\pi_{B \times B*}} \mathrm{Chr}
\mathcal{T}_{i_t^\pm}) \right) \right)
$$

The main 
theorem of Bernstein-Lunts \cite{BernsteinLunts} tells us that the constant
sheaf generates the category of $H^*(\bullet / G)$ for any connected Lie group
$G$, and that consequently there is an equivalence commuting with passage to cohomology
$ R\Gamma_G: D_m(\bullet / G) \to \mathrm{dgmod} \,\,H^*(\bullet / G) = H^*_G$.   Moreover for 
$i: H \subset G$ we have the pullback  $H^*_G \to H^*_H$
and compatibly $R \Gamma_H (i^* \cF) = H^*_H \otimes_{H^*_G} R \Gamma_G (\cF)$.
The tensor products is of dg-modules, and appropriately derived.  Substituting, 
$$
 \underline{H}^*(B^{ad} \backslash G, \mathrm{Chr}(I^* \mathcal{T}_\beta ) ) 
= \underline{H}^*\left(H_B \underline{\otimes}_{H_{B \times B}}  
\underline{R\Gamma}_{B \times B}   \left(\underline{\pi_{B \times B*}} \mathrm{Chr}
\mathcal{T}_{i_1}^\pm) \underline{\ast} \cdots \underline{\ast} (\underline{\pi_{B \times B*}} \mathrm{Chr}
\mathcal{T}_{i_t^\pm}) \right) \right)
$$
where now the outermost $\underline{H}$ means: each term in the complex to which $\underline{H}$ is being
applied is a dg-module; take its cohomology, these again form a complex with the differential induced by
that of the original complex.

Recall that the convolution (now over a point) is: take the diagram
$$B \backslash \bullet / B \times B \backslash \bullet / B \xleftarrow{res} B \backslash \bullet \times_B \bullet / B
\xrightarrow{m} B \backslash \bullet / B$$
and form $\cF \ast \cG = m_* res^* (\cF \boxtimes \cG)$.   
Viewing a $H_{B \times B}$-dgmodule as a $H_B$-dgbimodule, we have 
$$R\Gamma_{B \times B}(\cF \ast \cG) = R \Gamma_{B \times B}(\cF) \otimes_{H_B} R\Gamma_{B \times B}(\cG)$$

Thus we may expand
$$
\underline{R\Gamma}_{B \times B}\left(  (\underline{\pi_{B \times B*}} \mathrm{Chr}
\mathcal{T}_{i_1}^\pm) \underline{\ast} \cdots \underline{\ast} (\underline{\pi_{B \times B*}} \mathrm{Chr}
\mathcal{T}_{i_t^\pm}) \right)$$ into  
$$\underline{R\Gamma}_{B \times B} (\underline{\pi_{B \times B*}} \mathrm{Chr}
\mathcal{T}_{i_1}^\pm) \underline{\otimes}_{H_B} \cdots \underline{\otimes}_{H_B} 
\underline{R\Gamma}_{B \times B} (\underline{\pi_{B \times B*}} \mathrm{Chr}
\mathcal{T}_{i_t}^\pm)
$$

We have 
\begin{eqnarray}
{\bf T}_i & := &
\underline{R\Gamma}_{B \times B} (\pi_{B \times B*} \mathrm{Chr} \mathcal{T}_i) = 
H_{B \times B}(B) \langle -1 \rangle \to H_{B \times B} (\overline{G_{s_i}}) 
\\
{\bf T}_i^{-1} & := & \underline{R\Gamma}_{B \times B} (\pi_{B \times B*} \mathrm{Chr} \mathcal{T}_i^{-1}) = 
 H_{B \times B} (\overline{G_{s_i}}) \langle 1 \rangle  \to H_{B \times B}(B) \langle 1 \rangle 
\end{eqnarray}
In each case, the $G_{s_i}$ term is in degree $0$ in the complex.  Implicit in the above equality 
is an equivariant formality statement, which is however obvious: the sheaves involved,
viewed as $B$-equivariant sheaves on $G/B$, are respectively the constant sheaf on a point and a projective line.

The complexes ${\bf T}_i$ and ${\bf T}_i^{-1}$ are complexes of `Soergel' bimodules over
$H_B$.  Rouquier showed
that the map $s_{i_1}^{\pm} \mapsto {\bf T}_i^{\pm}$ determines a categorical representation of the braid
group, and in particular if $\beta = s_{i_1}^\pm \ldots s_{i_t}^\pm$ then, upto homotopy,
${\bf T}_\beta := {\bf T}_{i_1} \underline{\otimes}_{H_B} \cdots \underline{\otimes}_{H_B} {\bf T}_{i_n}$ only
depends on $\beta$.  (This would seem to follow from the above discussion, except that in \cite{Ro}
the well defined-ness of the representation on bimodules is used to prove the well defined-ness
of the representation on sheaves.)  We discuss ${\bf T}_i^\pm$ more explicitly in the next subsection.

In any case, we have established: 
$$H^i(\underline{H}^j(B^{ad} \backslash G, \mathrm{Chr}(I^* \mathcal{T}_\beta ))) = 
H^i(\underline{H}^j(H_B \underline{\otimes}_{H_B \otimes H_B} {\bf T}_\beta ) )$$

The left hand side is Williamson and Webster's definition of the invariant, and after
we write down the ${\bf T}_\beta^\pm$ explicitly it will be clear that the  
right hand side is Khovanov's \cite{Kh}.  It remains to match up the gradings. 
First a comment on $H^j (H_B \otimes_{H_B \otimes H_B} H_{B \times B}(\cF))$: 
the term $H_{B \times B}(\cF)$ is here regarded as $H_B \otimes H_B$ dg-module,
with trivial differential; so the tensor product gives the following $H_B$ dg-module with
trivial differential: $\oplus \mathrm{Tor}_k^{H_B \otimes H_B}(H_B, H_{B\times B}(\cF))[k]$. 
That is,  
$$H^j(H_B \otimes_{H_B \otimes H_B} H_{B \times B}(\cF) ) = 
\bigoplus_k \mathrm{Tor}_k^{H_B \otimes H_B}(H_B, H_{B\times B}(\cF))^{j+k}$$ where
the `$j+k$' refers to the degree coming from the fact that $H_B$, $H_{B \times B}(\cF)$ are
graded modules.  In Khovanov's picture, the three gradings
are $i, j, k$ (in the language of \cite{Kh}, $i$ is the homological grading,
$j+k$ is the polynomial grading, and $k$ is the Hochschild grading); 
we must explain how to recover $k$ on the LHS. 

According to \cite{WW-HOMFLY}, the answer is that $k$ can be recovered from the weight
filtration we have denoted $\Omega$.  Indeed, every $\cF$ which appears for us is the 
$B$-equivariant pushforward of either the constant sheaf on a point or the constant sheaf on a projective line
along $G/B \to \bullet/B$.  This map is {\em not} proper, but these sheaves are $B$-equivariantly formal
and their $B$-fixed loci are compact, which is enough to ensure that the pushforward is again
pure of weight zero.  Thus the only source of discrepancy between the weights and homological 
degrees comes from the shift by $k$ in the Tor.  That is, 

\begin{eqnarray*}
\mathrm{Gr}^W_{j+k} H^j(H_B \otimes_{H_B \otimes H_B} H_{B \times B}(\cF) ) & = &
\mathrm{Gr}^W_{j+k} 
\left ( \bigoplus_k \mathrm{Tor}_k^{H_B \otimes H_B}(H_B, H_{B\times B}(\cF))^{j+k} \right) \\
& = & \mathrm{Tor}_k^{H_B \otimes H_B}(H_B, H_{B\times B}(\cF))^{j+k} 
\end{eqnarray*}

Finally, we have established the following triply graded identity: 

\begin{equation} \label{WWequalsK}
\mathbb{W}^{i,j}_k = \mathrm{Gr}^W_{j+k} H^i(\underline{H}^j(B^{ad} \backslash G, \mathrm{Chr}(I^* \mathcal{T}_\beta ))) = 
H^i \left( \underline{\mathrm{Tor}}^{H_B \otimes H_B}_k (H_B, \mathbf{T}_\beta)^{\underline{j+k}}\right) 
\end{equation}
where as before, the $\underline{j+k}$ means that each term of the complex it's applied to is a graded vector
space, and you take the $j+k$'th graded piece.  The left hand side is identical to the invariant
introduced by Khovanov in \cite{Kh}, which we review in the next subsection.

\subsubsection{Calculations with Soergel Bimodules}
\label{subsubsec:Khovanov}

We now review from \cite{Kh} how to compute the RHS of Equation (\ref{WWequalsK}), and work out some examples.

Let $${\bf R} = \mathbb{Q}[x_1, \ldots, x_n] = H^*_B = H^*_{B \times B} (G_1)$$ and let
$${\bf B}_i = {\bf R} \otimes_{{\bf R}^{\sigma_i}} {\bf R}   = H^*_{B \times B}(\overline{G_{s_i}})$$
Note ${\bf B}_i$ is a ${\bf R}$-bimodule.  As in the previous section
the inclusion $G_1 \to \overline{G_{s_i}}$ induces the pullback
\begin{eqnarray*}  H^*( B \backslash \overline{G_{s_i}} / B ) & \to & H^* ( B \backslash G_1 / B) \\
{\bf B}_i & \to & {\bf R} \\ 
1 \otimes 1 & \mapsto & 1
\end{eqnarray*}
and, as $G_1 \to \overline{G_i}$ is a complex codimension 1 embedding of one 
smooth space in another, 
a corresponding Gysin pushforward map
\begin{eqnarray*}  H^{*} (B \backslash G_i / B) \langle -1 \rangle & \to & H^*( B \backslash \overline{G_{s_i}} / B ) \\
 {\bf R} \langle -1\rangle & \to & {\bf B}_i \langle 1 \rangle  \\
1 & \mapsto & (x_i - x_{i+1} \otimes 1) + 1 \otimes (x_{i+1} - x_i)  
\end{eqnarray*}
This Gysin map is of course not a map of rings, but does remain a map of ${\bf R}$-bimodules.  
In this section we use the term `polynomial degree' for $\deg(x_i) = 2$, although
this is in some sense a cohomological degree. We reserve `cohomological degree' for the external cohomological
degree, i.e. in a complex of bimodules, the cohomological degree will just be where
the bimodule is in the complex.  

The complexes of bimodules which arose in the last subsection were:
\begin{eqnarray*}
{\bf T}_i & = & [{\bf R}\langle -1 \rangle \to {\bf B} ] \\
{\bf T}_i^{-1} & = &  [{\bf B} \langle 1 \rangle \to {\bf R} \langle 1 \rangle] 
\end{eqnarray*}
where the ${\bf B}$ is in cohomological degree zero.

\begin{remark} This is identical to Khovanov's convention in \cite{Kh}, except that he writes
$\{-2\}$ for what we denote $\langle 1 \rangle$. \end{remark}

The Hochschild homology of a ${\bf R}$-bimodule $M$ is by definition 
$\mathrm{HH}_i(M) = \mathrm{Tor}_i^{{\bf R}\otimes {\bf R}^c}({\bf R}, M)$; we henceforth drop the $c$ as ${\bf R}$ is commutative.

Khovanov's prescription for the triply graded knot homology is:
$$\beta  \mapsto 
H^*(\underline{\mathrm{HH}}_* ({\bf T}_\beta) )$$
We have seen in the previous section that this agrees with the definition of Webster and Williamson, and
how to compare the gradings.  We indicate the differential of the complex ${\bf T}_\beta$, or the differential
it induces on Hochschild homology, by $d_R$. 

To compute the Hochschild homology, which is after all a Tor over ${\bf R} \otimes {\bf R}$, 
one needs a free graded bimodule resolution of ${\bf R}$. There
is the Koszul resolution: one tensors together the complexes
$$0 \to ({\bf R} \otimes {\bf R})\langle - 1 \rangle  \xrightarrow{[x_i, \cdot]} {\bf R} \otimes {\bf R} \to 0$$
In other words, to compute the Hochschild homology of a bimodule $M$, taking $\mathfrak{h}$ 
to be the vector space with basis $\delta x_i$, which we regard as living in polynomial degree
2, one takes on $M \otimes \Lambda^* \mathfrak{h}$ the differential
$$d_H(m \,\, \delta x_{i_1} \wedge \cdots \wedge \delta x_{i_k}) = \sum_{j = 1}^k (-1)^{j-1} [x_{i_j},m] \,\, \delta x_{i_1}
\wedge \cdots \wedge \widehat{\delta x_{i_j}} \wedge \cdots \wedge \delta x_{i_k}$$
The homology at the term $M \otimes \Lambda^k \mathfrak{h}$ gives $\mathrm{HH}_k$. 

We encode the dimensions of the triply graded vector spaces as follows. 
We use the letter $q$ to denote half the polynomial grading minus  the Hochschild grading, 
$q(x_i) = 1$, the letter $a$ to denote twice the Hochschild grading, $a(\mathrm{HH}_i) = 2i$, and $t$ to denote the
Hochschild grading minus the cohomological grading.  We write the Poincar\'e series of a 
triply graded vector space $V$ as 
$$[V]_{a,q,t} = \sum a^i q^{j} t^k \dim \{ v \in V \, | \, a(v) = i, q(v) = j, t(v) = k \}$$
Note in particular that $[V \langle -k \rangle]_{a,q,t} = q^{k} [V]_{a,q,t}$.

\begin{example}
Let $\bigcirc_1$ be the one-strand unknot.  The corresponding complex of bimodules is just 
${\bf R} = \bQ[x]$ in degree zero.  
The Hochschild homology is computed by the complex $\bQ[x]_{1} \langle -1 \rangle 
\xrightarrow{0}  \bQ[x]_{0}$, the subscripts indicating the 
Hochschild grading.  Thus the Poincar\'e series is:
$$[H^*(\underline{\mathrm{HH}}_*({\bf R}))]_{a,q,t} =  \frac{1 + a^2 t}{1-q}$$
\end{example}
\begin{example}
Let $(\bigcirc_1)^n$ be the $n$-strand $n$-component unlink.  The corresponding complex of bimodules is again just 
${\bf R} = \bQ[x_1,\ldots, x_n]$ 
in degree zero; ${\bf R}$ itself  now has Poincar\'e series $(1-q)^{-n}$.  
The complex computing Hochschild homology is 
a tensor product of many copies of the previous one, so now
$$[H^*(\underline{\mathrm{HH}}_* ({\bf R}))]_{a,q,t} = \left(  \frac{1 + a^2 t}{1-q} \right)^n$$
\end{example}

\begin{example}
 Let $\bigcirc_{2+}$ be the two-strand unknot with a single positive crossing.  The ring is 
 ${\bf R} = \bQ[x,y]$. The corresponding
 complex of bimodules is $\mathcal{T} = [{\bf R} \langle -1 \rangle \to {\bf B}]$, with ${\bf B}$ in cohomological
 degree 0, the map being 
  $1 \mapsto \Delta:= (x-y) \otimes 1 + 1 \otimes (x-y)$.  As a left $R$-module, $B$ is free of rank $2$ generated
 by $1 \otimes 1$ and $\rho := 1 \otimes (x -y)$.  The bimodule structure is recorded in the operators
 of right multiplication by $x$ and $y$.  Since in fact right multiplication by $x+y$ is the same as
 left multiplication by $x+y$, it suffices to record the matrix of right multiplication by $x-y$:
 $$R_{x-y} = \left(\begin{array}{cc} 0 & (x-y)^2 \\ 1 & 0 \end{array}\right)$$
 The complex computing the Hochschild homology of ${\bf B}$ is 
 $$0 \to {\bf B} \langle -2 \rangle \xrightarrow{([x-y, \bullet], [x-y, \bullet])/2} 
 {\bf B}\langle -1 \rangle \oplus {\bf B}\langle -1 \rangle 
  \xrightarrow{[x-y,\bullet_1 - \bullet_2]/2} 
 {\bf B}
 \to 0$$
 The matrix giving the operator of $[x-y, \bullet]:{\bf B}\langle -1 \rangle \to {\bf B}$ is 
 $$ \left(\begin{array}{cc} x-y & -(x-y)^2 \\ -1 & x-y \end{array}\right)$$
 Its kernel is  the submodule spanned by $(x-y) \mathbf{1} + \rho = 
 (x-y) \otimes 1 + 1 \otimes (x-y)$ and its image is the submodule spanned
 by $(x-y) \mathbf{1} - \rho = (x-y) \otimes 1 - 1 \otimes (x-y)$.  
 
 On the other hand the Hochschild homology of ${\bf R}$ is 
 computed by the complex with trivial differentials
 $$ 0 \to {\bf R} \langle -2 \rangle \to {\bf R} \langle -1 \rangle \oplus {\bf R} \langle -1 \rangle \to 
 {\bf R}
 \to 0$$
 
 The map ${\bf R}\langle -1 \rangle \to {\bf B} $ induces a map on Hochschild homology.
 We treat the three Hochschild
 gradings, enumerated below. We abbreviate  
 $\delta := (x-y) \otimes 1 - 1 \otimes (x-y)$.

 \begin{enumerate}
  \item[0.] We have $\mathrm{HH}_0({\bf B} ) = {\bf B}  / {\bf R}\delta$.  The map
  $d_R: {\bf R}\langle -1 \rangle = \mathrm{HH}_0({\bf R}\langle -1 \rangle) \to \mathrm{HH}_0({\bf B})$ is evidently 
  injective and has image ${\bf R} \Delta$; the cokernel is  
  $${\bf B} / {\bf R}(\delta, \Delta) \cong {\bf R} /(x-y)$$
  \item[1.] The kernel in the Hochschild complex is generated as a ${\bf R}$ left submodule of 
  ${\bf B} \langle -1 \rangle \oplus {\bf B} \langle -1 \rangle$ by 
  $(1,1), (\rho, \rho), (0, \Delta)$.  The image is generated by $(\delta, \delta)$, so 
  $$\mathrm{HH}_1({\bf B}) = \frac{{\bf R}(1,1) + {\bf R}(\rho, \rho) + {\bf R}(0,\Delta)}{{\bf R}(\delta, \delta)}
  \langle -1 \rangle
  $$
  Again $d_R: \mathrm{HH}_1({\bf R}\langle -1 \rangle) = {\bf R} \langle -2 \rangle \oplus {\bf R} \langle -2 \rangle 
  \to \mathrm{HH}_1({\bf B})$ is injective; its 
  image is ${\bf R}(\Delta, 0) + {\bf R}(0, \Delta)$.
  So the quotient is $$\frac{ {\bf R}(1,1) + {\bf R}(\rho,\rho)}{{\bf R}(\delta, \delta) + {\bf R}(\Delta, \Delta)} 
  \langle -1 \rangle
  \cong {\bf B} \langle -1 \rangle/(\delta,\Delta) \cong {\bf R} \langle -1 \rangle/(x-y)$$
  \item[2.] We have $\mathrm{HH}_2({\bf B}) = {\bf R} \Delta$.  This is exactly the image of the map from 
  $\mathrm{HH}_2({\bf R})$. 
 \end{enumerate} 
 In all, we have
 $$[H^*(\mathrm{HH}_* (\mathcal{T}))]_{a,q,t} = \left(\frac{1 + a^2 t}{1-q}\right)$$
 \end{example}

  \begin{example}
   Let $\bigcirc_{2-}$ be the two-strand unknot with a single negative crossing.  The Rouquier
   complex is $\mathcal{T}^{-1} = [{\bf B} \langle 1 \rangle \to {\bf R} \langle 1 \rangle]$, with 
   ${\bf B} \langle 1 \rangle$ in cohomological degree 0.  
   The map is 
   surjective and has kernel generated by $\delta$.  We have already computed the Hochschild
   homology of ${\bf R}$ and ${\bf B}$; we discuss  $d_R$.  
  \begin{enumerate}
  \item[0.] We had $\mathrm{HH}_0({\bf B}\langle 1 \rangle) = {\bf B}\langle 1 \rangle / {\bf R} \delta$.  The map to 
  $\mathrm{HH}_0({\bf R}\langle 1 \rangle) = {\bf R}\langle 1 \rangle $ is an isomorphism. 
  \item[1.] We had
    $$\mathrm{HH}_1({\bf B}\langle 1 \rangle) = \frac{{\bf R}(1,1) + {\bf R}(\rho, \rho) + {\bf R}(0,\Delta)}{{\bf R}(\delta, \delta)}
  $$
  so $d_R$ has no kernel.  Its image is spanned by $(1,1)$ and $(0,x+y)$ 
  in $\mathrm{HH}_1({\bf R}\langle 1 \rangle) = {\bf R} \oplus {\bf R} $. 
  So the cohomology here is ${\bf R}/(x+y)$, in cohomological degree 1. 
  \item[2.] We had $\mathrm{HH}_2({\bf B}\langle 1 \rangle) = {\bf R}\langle -1 \rangle \Delta$, 
  which  does not intersect $\ker d_R = {\bf R}\langle -1 \rangle \delta$. 
  The image is spanned by $(x+y)$ in $\mathrm{HH}_2({\bf R}\langle 1 \rangle) = {\bf R} \langle -1 \rangle$, so the cohomology here 
  is $ {\bf R} \langle -1 \rangle /(x+y)$.
 \end{enumerate}
  Note the cohomology of $d_R$ is now in cohomological degree $1$.  All in all the Poincar\'e series is 
   $$[H^*(\mathrm{HH}_* (\mathcal{T}^{-1}))]_{a,q,t}   = \frac{a^2 q^{-1} + a^4 q^{-1} t}{1-q} = \left(aq^{-1/2} \right)^2 
   \left(\frac{1+a^2 t}{1 - q}\right)$$
 \end{example}

\begin{definition} 
  Let $\beta$ be a braid; write $w(\beta)$ for its writhe and $n(\beta)$ for its number of strands. 
  The Poincar\'e series of the HOMFLY homology of $\overline{\beta}$ is 
  \begin{eqnarray*} \cP(\overline{\beta}) & := & (aq^{-1/2})^{w(\beta) - n(\beta)}  
  [H^*(\underline{\mathrm{HH}}_* (\mathbf{T}_\beta)])_{a,q,t} \\
  & = & (aq^{-1/2} )^{w(\beta) - n(\beta)} \sum a^{2k} q^{(j-k)/2} t^{k-i} \dim H^i(\underline{\mathrm{HH}}_k
  (\mathbf{T}_\beta))^{\underline{j+k}}   \\
  & = &  (aq^{-1/2} )^{w(\beta) - n(\beta)}\sum a^{2k} q^{(j-k)/2} t^{k-i} \dim \mathrm{Gr}^\Omega_{j+k} H^i(\underline{H}^j
  (\mathrm{Chr} I^* \mathcal{T}_\beta))^{\underline{j+k}} \\
  & = &   (aq^{-1/2} )^{w(\beta) - n(\beta)} \sum a^{2k} q^{(j-k)/2} t^{k-i} \dim \mathrm{Gr}^\Omega_{j+k} E_2^{ij} 
  \end{eqnarray*}
  It is an invariant of the knot.
\end{definition}

\begin{remark} This normalization convention differs from that of \cite{ORS} 
in that our $q^{1/2}$ is their $q$.
It agrees with the convention used in this paper for the HOMFLY polynomial, in that the HOMFLY polynomial is recovered
at $t = -1$. 
\end{remark}

\subsubsection{The proof of Theorem \ref{thm:HHH}}

Comparing the assertions of Theorems \ref{thm:HHH} and \ref{thm:WWH}, we need only 
to identify the spectral sequences 
$$E^{ij}_1 = H^{i+j}(G \backslash G, \mathrm{Gr}_{-i}^W \pi_{\beta^\circ *} \bQ) \Rightarrow H^{i+j}(G\backslash G, \pi_{\beta^\circ*} \bQ)$$ 
and
$$E^{ij}_1 = H^{i+j}(B^{ad} \backslash G, \mathrm{Gr}_{-i}^W I^* \mathcal{T}_\beta) \Rightarrow H^{i+j}(B^{ad} \backslash G, I^* \mathcal{T}_\beta)$$ 
To do so, note that the map $B^{ad} \backslash G \to G \backslash G$ is proper (it has fibers $G/B$) 
and hence the pushforward along it preserves weights.  Thus we may as well compute the latter
sequence after first pushing forward along this morphism.  Having done so, it can be rewritten as
$$E^{ij}_1 = H^{i+j}(G\backslash G, \mathrm{Gr}_{-i}^W \mathscr{H}(\mathcal{T}_\beta)) \Rightarrow H^{i+j}(G \backslash G, \mathscr{H}(\mathcal{T}_\beta))$$ 
by definition of the horocycle transform $\mathscr{H}$.  

It remains to identify $\mathscr{H}(\mathcal{T}_\beta)$ with $\pi_{\beta^\circ*} \bQ$.  But by 
Proposition \ref{prop:rrr}, we have $\mathcal{T}_\beta \cong \pi_{\beta *} \bQ$, and  by 
Proposition \ref{prop:hor}, we have  $\mathscr{H}(\pi_{\beta *} \bQ) \cong \pi_{\beta^\circ*} \bQ$. 
$\square$

\subsection{Rainbow braid closures, ruling filtrations, point counts, and HOMFLY} \label{subsec:rainbow}

Let $\beta$ be a positive $n$-stranded braid.  We write $\beta^{\succ}$ for its rainbow closure.  

We analyze
$\cC(\beta^{\succ})$ and $\mM(\beta^{\succ})$ in terms of the following covering.
\begin{center}  
\includegraphics[scale = .3]{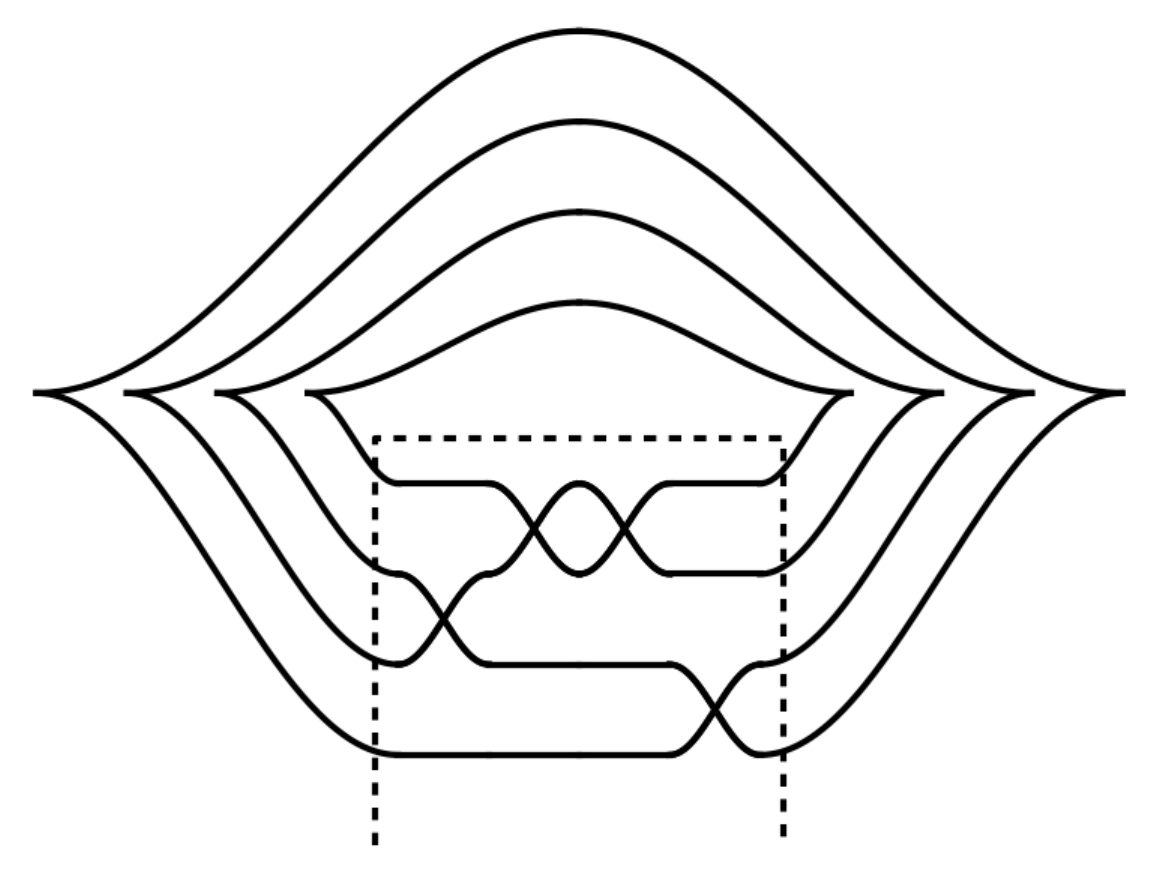} 
\end{center}
Writing $\rainbow$ for the region above the dashed line and $\Pi$ for a slight thickening of the dashed
line, we have by the sheaf axiom 
$$\cC(\beta^{\succ}) = \cC(\beta) \times_{\cC(\Pi)} \cC(\rainbow)$$
and similarly for all the other related categories and moduli spaces.  We focus on the latter.  We preserve
the notation of the previous sub-section, fixing in particular some microlocal rank $r$ and just writing
$\mM, \overline{\mM}$ for $\mM_r, \overline{\mM}_r$, and $G$ for $\mathrm{GL}_{nr}$ and $P$ for the parabolic
of $r \times r$ block-upper-triangular matrices.  

It is straightforward to see that $\mM(\Pi) = \mathrm{Bun}_{G,P}(\barbell) =  G \backslash (G /P \times G / P)$; in fact, 
the restriction map $\mM(\beta) \to \mM(\Pi) = \mathrm{Bun}_{G,P}(\barbell)$ is just the same map we were studying
in the previous section.  

On the other hand, the space $\mM(\rainbow)$ can be described in terms of the data of three flags: one each from the 
 left ($L_\bullet$) and right ($R_\bullet$) sides of where the braid was removed from the picture, and one from the top
 of the rainbow ($T_\bullet$).  The flag $T_\bullet$ is defined by taking $T_i$ to be the kernel of the map from the middle
 to the spot $i$ steps above the middle;  $T_i k^{\oplus nr} = \mathrm{ker}(k^{\oplus nr} \to k^{\oplus (n-i) r})$.  
 This gives
 three increasing flags.  The singular support conditions at the cusp amounts to the statement that the pairs
 ($L_\bullet, T_\bullet$) and $(T_\bullet, R_\bullet)$ are completely transverse flags; that is, for all $i$, we have
 $L_i k^{\oplus nr} + T_{n-i} 
 k^{\oplus nr}= k^{\oplus nr}$, or equivalently $L_i k^{\oplus nr} \cap T_{n-i} 
 k^{\oplus nr}= 0$.  Such flags are parameterized by the big open Schubert cell, which is isomorphic to the open 
 Bott-Samelson for any braid presentation of it; in particular, for the half-twist $\Delta$ which lies in the center of 
 the braid group.  Thus we have 
 $$\mM(\rainbow) = \mM(\Delta) \times_{\mM(\equiv)} \mM(\Delta) = \mM(\Delta^2)$$
 
 \begin{proposition} \label{prop:rainbowmoduli} 
   We have 
   $$\mM(\beta^\succ) = \mM(\beta) \times_{G \backslash (G /P \times G / P)} \mM(\Delta^2) = \mM(\Delta \beta \Delta) 
   \times_{G \backslash (G /P \times G / P)} \mM(\equiv) $$
   and moreover
   $$\mM(\beta^\succ) = 
   \mM(\Delta^2 \beta^\circ) 
   \times_{G\backslash G} G \backslash id$$ 
   where $G \backslash id \to G\backslash G$ is the inclusion of the identity.  
 \end{proposition}
\begin{proof}
The first  equality in the first statement was already established in the above discussion. 
For the second equality in the first statement, consider the covering determined by the following diagram:
\begin{center}
\includegraphics[scale = .3]{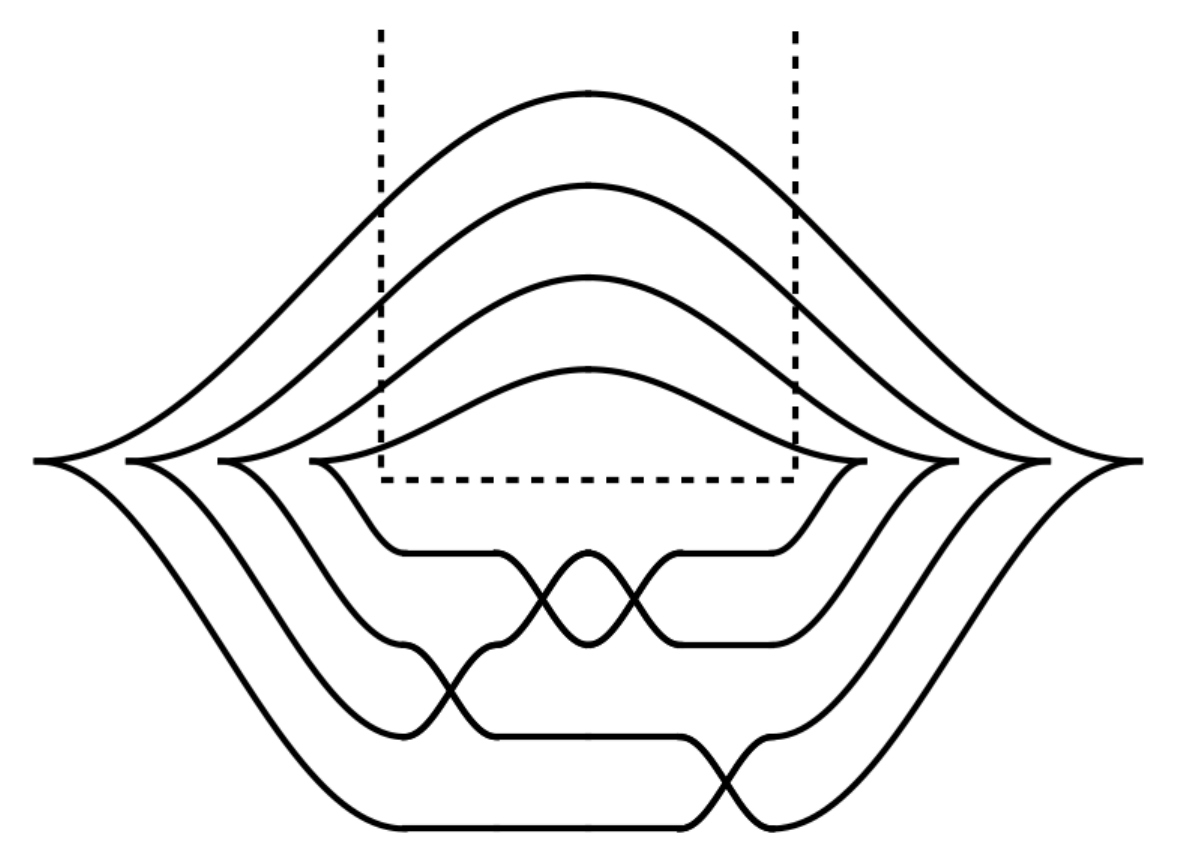} 
\end{center}
For the second statement, note $\Delta^2 \beta^\circ = \Delta \beta \Delta^\circ$.  We have 
$$\mM (\Delta^2 \beta^\circ) = \mM(\Delta\beta \Delta) \times_{\mM(\equiv) \times \mM(\equiv)} \mM(\equiv) = 
\mM(\Delta \beta \Delta) \times_{(G \times G) \backslash(G/P \times G/P )} G\backslash G/P$$
comparing this to 
$$\mM(\beta^\succ) = \mM(\Delta \beta \Delta) \times _{G \backslash (G /P \times G / P)} G\backslash G/P$$
gives the second equality.
\end{proof}

\begin{remark}
This Proposition has a more conceptual explanation: $\beta^\succ$ is the Legendrian
cable of the standard unknot by $\Delta \beta \Delta^\circ$.  A detailed account can be
found in \cite[Sec. 7]{STWZ}. 
\end{remark}

We study the moduli space $\mM_1(\beta^\succ)$ by considering ruling filtrations.  We restrict to a certain class of filtrations:

\begin{definition}
Let $\beta$ be a positive braid, and $\cF \in \cC_1(\beta^\succ)$.  Then 
we say a ruling filtration $R_\bullet \cF$ is normal if $\underline{\Gr_i F}$  connects the $i$-th innermost cusps. 
\end{definition}

\begin{proposition} \label{prop:oneruling}
Let $\beta$ be a positive braid, and let $\cF \in \cC_1(\beta^\succ)$.  Then there exists a unique normal ruling filtration 
$R_\bullet \cF$. The corresponding ruling $\underline{R}$ is normal. 
\end{proposition}
\begin{proof}
Note that by Proposition \ref{prop:bimapfiltrations}, we may assume that all the $R_i \cF$ are actual sheaves.  We work 
with legible objects as per Propositions \ref{prop:braidlegible},  \ref{prop:bos}.  

We imagine the braid to have all the cusps in the same horizontal line, which is moreover above the braid, as in Figure 
\ref{fig:braidclosure}.  Every region above this line admits a surjective map from the stalk of $\cF$ at the central region; we fix a 
framing $k^{\oplus n}$ for this stalk.  The normality assumption forces the filtration above the line through the cusps to be 
the filtration $T_\bullet$.  The stalks in the region below the line all admit canonical injective maps to the stalk in the central region. 
We write $\cF(p)^\uparrow$ for the image of the injection of the stalk $\cF(p) \to k^{\oplus n}$ given by
following the stalk along an upward path; it follows from
Proposition \ref{prop:bos} that any path going upward gives the same injection. 
Similarly for a stratum $S$ of the front diagram stratification in the braid, we write $\cF(S)^\uparrow$ 
for the image of the injection $R\Gamma(S, \cF) \to k^{\oplus n}$.  We do the same for filtered and graded pieces of $\cF$.  
We define a 
filtration  $\cT_\bullet \cF$ by  intersecting the images of these injective maps with the filtration $T_\bullet$, that is, 
$$(\cT_i \cF)(S)^\uparrow := \cF(S)^\uparrow \cap T_i k^{\oplus n}$$

It remains to show that $\cT_\bullet$ is the unique normal ruling filtration of $\cF$.  It is easy to see this away
from the braid; we restrict attention to the braid itself.  

We first check that $\cT_\bullet$ is a ruling filtration; i.e., that its associated graded pieces are eye sheaves.  Let 
$t_i = T_i k^{\oplus n} = \ker(k^{\oplus n} \to k^{\oplus (n-i) })$.  We have
by definition
$\cT_i \cF(S)^\uparrow = \cF(S)^\uparrow \cap t_i$.  Thus $$\Gr_i \cF(S)^\uparrow = \cF(S)^\uparrow \cap t_i / \cF(S)^\uparrow \cap t_{i-1} = \mathrm{im} 
(\cF(S) \to t_i / t_{i-1} )$$
and in particular, $\Gr_i \cF(S)$ is everywhere zero or one dimensional.  Since the upward maps within the braid
$\cF(S) \to \cF(N)$ are injective, the same holds for $\Gr_i \cF(S) \to \Gr_i \cF(N)$; it follows that $\Gr_i \cF$ is an eye sheaf. 

We check uniqueness.  We have already seen that any filtration $\cT'$ satisfying the assumptions is identified with $\cT$ away from 
the braid.  Throughout we take $S$ a stratum of the front diagram stratification of the braid.  
By Lemma \ref{splotch}, the observation that `normality' of a ruling filtration ensures that 
$\Lambda \setminus \underline{\Gr_n^{\cT'} \cF}$ is again a rainbow braid closure,  and by induction, it's enough to show that the 
penultimate step of the filtration $\cT'_{n-1}$ is characterized by $\cT'_{n-1} \cF (S)^\uparrow = \cT_{n-1} \cF(S)^\uparrow =  \cF(S)^\uparrow 
\cap t_{n-1}$.  Since we must have  $\cT'_{n-1}(S)^\uparrow \subset \cT'_{n-1}(k^{\oplus n})$, it follows that 
$\cT'_{n-1}\cF(S) \subset \cT_{n-1}\cF(S)$.   On the other hand, we must have $1 \ge \dim \Gr_{n-1}^{\cT'} \cF(S)  = 
\dim \cF (S) / \cT'_{n-1}\cF(S)$.   So if the inclusion $\cT'_{n-1}\cF(S) \subset \cT_{n-1}\cF(S)$ is strict, then 
we must have $\dim \cT'_{n-1}\cF(S) = \dim \cF(S) - 1$ and $\cT_{n-1}\cF(S) = \cF(S)$.  Let 
$\eta \in  \Gr_{n-1}^{\cT'}\cF (S) = \cF(S)/ \cT'_{n-1}\cF(S) $ be a nonzero element.  It maps to zero in 
$\Gr_{n-1}^{\cT} \cF(S) = 0$, hence to zero in $\Gr_{n-1}^{\cT}(k^{\oplus n})$.  This cannot happen in an eye sheaf: 
contradiction.

For normality, 
by induction and Lemma \ref{splotch}, it is enough to show that $\underline{\Gr_n F}$ meets any switch from the bottom.
At a crossing with bordering regions $N, E, S, W$, the data of $\cF$ is:
\[
\xymatrix{
& \cF(N)  \\
\cF(W)  \ar@{_{(}->}[ur] & & \cF(E) \ar@{^{(}->}[ul] &\\
& \cF(S) \ar@{_{(}->}[ul] \ar@{^{(}->}[ur]
}
\]

Suppose instead that $\underline{\Gr_n \cF}$ is an eye which meets this crossing in a switch and from the top, i.e. its boundary includes
the northeast and northwest arcs leaving the crossing.  Then
around this crossing it must be one of two forms:
\[
\xymatrix{
 &  k  & & & &  0 \\
  0 \ar[ur]  & & 0 \ar[ul]  & & k \ar[ur] & & k \ar[ul] &\\
  & 0 \ar[ul]  \ar[ur] & & & & k \ar@{=}[ul] \ar@{=}[ur]
}
\]

The eye $\underline{\Gr_n \cF (S)}$  cannot take the form on the right, because
then the isomorphism $ \Gr_n \cF (S) \to  \Gr_n \cF(S)^\uparrow$ would factor through zero, which is a contradiction.  
The form on the left can be ruled out because the surjection
$\cF \to \Gr_n \cF$ would induce a diagram
\[
\xymatrix{
 &  \cF(N) \ar@{-->>}[rrrr] & & & &  k \\
  \cF(W) \ar@{_{(}->}[ur] \ar@/^1pc/@{-->>}[rrrr] & & \cF(E) \ar@{^{(}->}[ul] \ar@/_1pc/@{-->>}[rrrr] & & 0 \ar[ur] & & 0 \ar[ul] &\\
  & \cF(S) \ar@{_{(}->}[ul]  \ar@{^{(}->}[ur] \ar@{-->>}[rrrr] & & & & 0 \ar[ul] \ar[ur]
}
\]
However, this diagram cannot commute: on the one hand, we must have $\cF(E) \oplus \cF(W)  
\twoheadrightarrow \cF(N) \twoheadrightarrow k$; on the other hand the maps $\cF(E) \to k$
and $\cF(W) \to k$ both factor through zero.  
This is again a contradiction, and completes the proof. 
\end{proof}

\begin{lemma}
\label{lemma:extlemma}
Let $\cE, \cF$ be two eye sheaves such that 
the cusps and upper strand of $\underline{\cF}$ 
are entirely contained in the interior of $\underline{\cE}$; suppose
moreover the bottom strands of $\underline{\cE}, \underline{\cF}$ meet only as 
eyes of a given ruling might meet, i.e. transversely or as at a switched crossing.
Say there are $t$ transverse intersections, and $s$ switched crossing intersections.  
\begin{itemize}
\item
 $R^0\mathcal{H}om(\cE, \cF)$ is constant with one dimensional stalk
  on $\mathrm{int} (\underline \cE \cap \underline \cF)$. 
It has standard boundary conditions
on the open interval along the top of $\cF$, and along the open intervals
on the $t/2$ open intervals $\partial \underline{\cE} \cap \mathrm{int} \underline{\cF}$; it has
costandard boundary conditions along the rest of the boundary. 
\item  
$R^1 \cH om(\cE, \cF)$ is a direct sum of one dimensional skyscraper sheaves on 
the $s$ switched crossings
\item All the other $R^i \cH om(\cE, \cF)$ vanish. 
\end{itemize}
In particular, $R^i\Hom(\cE, \cF)$ vanishes except when $i=1$, and $R^1 \Hom(\cE, \cF)   = k^{\oplus s + \frac{t}{2}}$.
\end{lemma}

\begin{lemma} \label{lem:ssswitch} Let $\cE, \cF$ be as in Lemma \ref{lemma:extlemma}; note the assumptions there
amount to the statement that $\partial \underline{\cE} \cup \partial \underline{\cF} = \beta^\succ$ for
some positive braid $\beta$.  Let
$0 \to \cF \to \cG \to \cE \to 0$ be an extension.  Then $\cG \in \cC_1(\beta^\succ)$ if and only if $\cG$ is
locally a nontrivial extension at every switched crossing, i.e. if
the image of $[\cG] \in R^1 \cH om(\cE, \cF)$ under the canonical map 
$R^1 \cH om(\cE, \cF) \to \mathrm{H}^0(R^1 \cH om(\cE, \cF)) = k^{\oplus s}$ lands inside 
$(k^*)^{\oplus s}$. 
\end{lemma}
\begin{proof} Since $\SS(\cG) \subset \SS(\cE) \cup \SS(\cF)$, the only thing to check is the singular support
at the switched crossings.  This is a straightforward local calculation. \end{proof}

\begin{proposition} \label{prop:rulingmoduli}
Let $\beta$ be a positive braid on $n$ strands.  Let $R$ be a graded normal ruling of 
$\beta^\succ$.  Let  $E_1, \ldots, E_n$ be the eyes of the ruling, ordered so that $E_i$ 
contains the $i$th strand above the central region.

Let $\cC_1(\beta^\succ)_R \subset \cC_1(\beta^\succ)$ be the full subcategory whose
objects  whose (unique) normal ruling filtration has underlying ruling $R$, 
and let $\mM_1(\beta^\succ)_R \subset \mM_1(\beta^\succ)$ be the corresponding moduli space.  

Then the space $\mM_1(\beta^\succ)_R$ is an iterated fiber bundle over point, each bundle having fibres of the form
$((k^*)^{a_i} \times k^{b_i}) / k^*$ for some $a_i, b_i$ such that $\sum a_i = s$ and $\sum b_i = {\frac{1}{2}(w-s)}$,
where there are $w$ total crossings and $s$ switched crossings. 

In particular, if $k$ is a finite field with $q$
elements, then the (orbifold) cardinality of $\cC_1(\beta^\succ)_R$ is given by
$$\# \cC_1(\beta^\succ)_R = (q-1)^{s-n} q^{\frac{w-s}{2}} = (q^{1/2} - q^{-1/2})^{s-n} q^{\frac{w-n}{2}}$$
\end{proposition}
\begin{proof}
For any $\cF \in \cC_1(\beta^\succ)_R$, we  write $R_\bullet \cF$ for its (unique) normal
ruling filtration.  We write $\cE_i$ for the eye sheaf
with $\underline{\cE_i} = E_i$.  

Note that $\beta^\succ \setminus E_n$ is again the rainbow closure of a braid, which we denote  
$\beta_{n-1}$.  We still write $R$ for the restriction of the ruling $R$ to $\beta_{n-1}$.  
The uniqueness of the ruling filtration $R_\bullet \cF$ together with Lemma \ref{splotch} asserts that the
following map is well defined: 
\begin{eqnarray*}
\pi_n : \cM_1(\beta^\succ)_R & \to & \cM_1(\beta_{n-1}^\succ)_R \\
\cF & \mapsto & R_{n-1} \cF
\end{eqnarray*}
Let $\Ext^1_{\beta}(\cE_n, R_{n-1} \cF) \subset \Ext^1(\cE_n, R_{n-1} \cF)$ denote those
extensions giving rise to sheaves in $\cC_1$, i.e., which satisfy the correct singular support
condition at the crossings.  Then the fibres of the above map $\pi_n$ are
$\Ext^1_\beta(\cE_n, R_{n-1} \cF) / (\mathrm{Aut}(\cF) / \mathrm{Aut}(R_{n-1} \cF) )$. 

We claim $\Ext^a(\cE_{i},R_{j} \cF)=0$ if $a\ne 1$ and $i>j$.  Indeed,
from the exact sequence $$0 \to R_{j-1} \cF \to R_j \cF \to \cE_j \to 0$$ we have the exact sequence
$$\Ext^a(\cE_{i},R_{j-1} \cF) \to \Ext^a(\cE_{i},R_{j} \cF) \to \Ext^a(\cE_{i},\cE_j)$$
We may assume the first term vanishes by induction; the third vanishes by Lemma \ref{lemma:extlemma}, so the
second vanishes as well. 
It follows that when $a=1$ and $i > j$ we have the short exact sequence

$$0 \to \Ext^a(\cE_{i},R_{j-1} \cF) \to \Ext^a(\cE_{i},R_{j} \cF) \to \Ext^a(\cE_{i},\cE_j) \to 0$$

In particular we have 
$$\mathrm{Aut}(\cF)  = \mathrm{Aut}(R_{n-1} \cF) \times \mathrm{Aut}(\cE) \times
\Hom^0(\cE, R_{n-1} \cF) = \mathrm{Aut}(R_{n-1} \cF) \times \mathrm{Aut}(\cE)$$
so $$\mathrm{Aut}(\cF) / \mathrm{Aut}(R_{n-1} \cF)  = k^*$$

Moreover, we see $\Ext^1(\cE_n, R_{n-1} \cF)$ admits a filtration with associated graded
pieces $\Ext^1(\cE_n, \cE_j)$ for $j < n$.  We claim that an extension $\cG$, i.e., 
$0 \to R_{n-1} \cF \to \cG \to \cE_n \to 0$, will have singular support in $\beta$ if and only if,
for all $j$, 
$\mathrm{Gr}_j [\cG] \in \Ext^1(\cE_n, \cE_j)$ has singular support in the two-strand braid
whose front diagram is $\underline{\cE_n} \cup \underline{\cE_j}$.  This is just because the singular
support calculation only needs to be made at the switched crossings, and only two eye sheaves contribute
nontrivially to this calculation at any crossing.  By Lemma \ref{lem:ssswitch}, this amounts to saying that
$\Ext^1_\beta(\cE_n, R_{n-1} \cF)$ is an $n-1$-step iterated bundle whose $j$'th fibre is
$(k^*)^{\oplus s_{nj}} \times k^{\oplus t_{nj}/2}$, where $s_{nj}$ is the number of switched crossings
between $E_n$ and $E_j$, and $t_{nj}$ is the number of transverse crossings between $E_n$ and $E_j$. 

Since $s = \sum s_{ij}$ and $t = \sum t_{ij}$, an induction on $n$ completes the proof. 
\end{proof}

\begin{remark}
The quantity  $t/2$ is equal to the number of ``returns''
of $R$ as defined in \cite{HR}.
\end{remark}

\begin{definition}
Let $K$ be a topological knot.  We write $P(K) \in \bZ[(q^{1/2}- q^{-1/2})^{\pm 1}, a^{\pm 1}]$ for the HOMFLY polynomial of
$K$.  Our conventions are given by the following skein relation and normalization:

\begin{eqnarray*}
  \label{eq:skein}
  a \, P(\undercrossing) -  a^{-1} \, 
  P(\overcrossing) & = & (q^{1/2} - q^{-1/2})  
  \, P(\smoothing) \\
  a - a^{-1} & = & (q^{1/2}-q^{-1/2})P(\bigcirc)
\end{eqnarray*}
With these conventions, the lowest degree power of $a$ in the HOMFLY
polynomial of a positive braid with $w$ crossings and $n$ strands is $a^{w-n}$.  
\end{definition}

\begin{theorem} \label{thm:homflycount}
Let $\beta$ be a positive braid on $n$ strands with $w$ crossings.  Let the coefficients $k$ be a finite field with $q$ elements. 
Then the (orbifold) cardinality of $\cC_1(\beta^\succ)$ is related to the HOMFLY
polynomial by the formula
$$\# \cC_1(\beta^\succ) = \left[(aq^{-1/2})^{n-w} P_{\beta^\succ}(a,q)\right|_{a=0}$$
\end{theorem} 
\begin{proof}
By Propositions \ref{prop:oneruling} and \ref{prop:rulingmoduli}, we have 
$$\# \cC_1(\beta^\succ) = q^{\frac{w-n}{2}} \sum_R  (q^{1/2} - q^{-1/2})^{s(R)-n}$$
where the sum is over all normal rulings, and $s(R)$ is the number of switched crossings
in the ruling $R$.  By \cite{Ru}, this is equal to the stated coefficient of the HOMFLY polynomial
(although slightly different normalization conditions for HOMFLY are used in \cite{Ru}.) 
\end{proof}

\begin{remark} 
Using the Lefschetz trace formula for Artin stacks \cite{Beh}
and arguing as in Katz's appendix to \cite{HRV}, 
Proposition \ref{prop:rulingmoduli} 
implies that $\# \cC_1(\beta^\succ)$
is given by the weight polynomial of the moduli
stack $\cM_1(\beta^\succ)$. \end{remark}

\begin{remark} \label{rem:maulik}
As shown in \cite{M},  for a positive braid of the form $\beta = \beta' \Delta^2$, where $\beta'$ is again positive, 
coefficients of the other powers of $a$, and indeed the HOMFLY polynomials
corresponding to other representations, can be recovered from the lowest degree term
in $a$ of the ordinary HOMFLY polynomials of certain auxiliary positive braids.  These braids
are made as follows: first consider the positive braid $\beta''$ formed from $\beta'$ 
by adding a strand which loops around the
others once, and then, for each partition $\pi$, consider the braid $\beta''_\pi$ formed
by cabling $\beta''$ by the Aiston-Morton idempotent \cite{AM} corresponding to $\pi$.  Note
the link of any plane curve singularity can be presented in the form $\beta' \Delta^2$, where $\beta'$ is a braid
presentation of the singularity obtained by blowing up.  Thus,
for all such links, we have moduli spaces whose point counts recover all HOMFLY invariants. 
\end{remark}

\begin{remark}
One might suppose that there should be some relation between Theorem \ref{thm:homflycount}, Theorem \ref{thm:HHH}, Proposition \ref{prop:rainbowmoduli}, and Remark \ref{rem:maulik}.  We consider it a very interesting problem to determine what this may be.  
\end{remark}

\begin{example}
\label{ex:rainbowtrefoil}
The trefoil is the $(2,3)$ torus knot, the closure of a braid with two
strands and three crossings.
After closing all strands in arcs ``over the top'' --- see
Figures~\ref{fig:planartrefoil}
and \ref{fig:trefoilpic} ---
the front diagram has four compact regions:  a rainbow region, a central
region (with the label $k^2$ in the diagram),
and left and right (lower) regions.
The front diagram $\Phi$ is a braid closure, and we equip it with its
canonical binary Maslov potential.  Write
$\Lambda$ for the associated Legendrian in $T^{\infty,-}\bR^2$.
Proposition \ref{prop:braidlegible} plus a further small argument
can be used to show that all
objects of $\dgfun_\Lambda(\bR^2,k)$ are legible objects.
A rank-one legible object $F$ is defined by injective
maps $a_1, a_2,a_3,a_4$ all from $k\rightarrow k^2$ and a projection $\pi:
k^2\rightarrow k.$
Write $l_i = Im(a_i)\subset k^2$, $i = 1,...,4,$ and set $l_0 = ker(\pi).$
The crossing condition --- see Definition \ref{def:legible}  --- requires
$$l_i \neq l_{i+1\; {\rm mod}\; 4}$$
and the moduli space $\cM_1$ is the set of such choices modulo $\PGL_2(k).$
\begin{figure}[H]
\includegraphics[scale = .4]{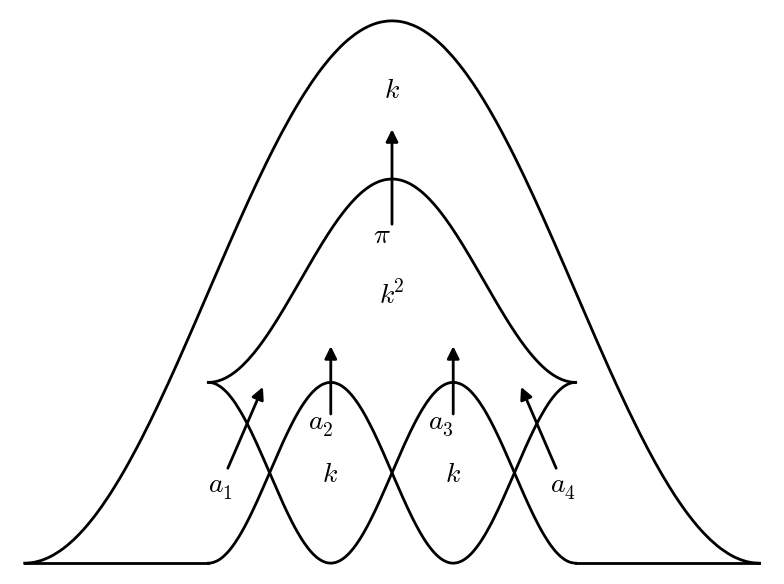}
\caption{Legible diagram for the trefoil as (2,3) braid closed over the top.}
\label{fig:trefoilpic}
\end{figure}

$\cM_1$ has a stratification by the loci where $l_2 = l_0,$ where
$l_3 = l_0,$ and the
complement of these.  Let $F$ be an object and let $E_C$
be the constant sheaf with fiber $l_0$ on the central region,
let $E_L$ be the
sheaf in $\cC$ which is constant with fiber $l_0$
on the union of the central and left regions,
and let $E_R$ be the right counterpart of $E_L.$  Then on the first
stratum, $E_L$
is a subobject of $F$ (in the category of sheaves), $E_R$ is a subobject on the
second
stratum, and $E_C$ is a subobject on the third stratum.  This stratification is
therefore
the \emph{ruling stratification}, since the sheaves $E_L, E_R$ and $E_C$
all represent
disks corresponding to the three different (graded, normal) rulings of the
front diagram, with switching sets respectively just the right crossing,
just the left crossing, and all crossings.  Write $\cM_1 = \cS_L\cup \cS_R
\cup \cS_C.$

Now the $l_i$ are in $\bP^1_k.$   By the $\PGL_2(k)$ action, we may assume
$l_0 = \infty$ and $l_1 = 0$ (leaving a $k^*$ symmetry group).
Then on $\cS_L$ we have $l_2 = \infty$ with $l_3 \neq l_4 \in k =
\cP^1_k\setminus \infty.$
Therefore $\cS_L \cong (k\times k^*)/k^* \cong k$ under the identification
$(l_3,l_4) \leftrightarrow (l_3,l_4-l_3)$.
Similarly, $\cS_R\cong k.$  On $\cS_C$ we can put $l_0, l_2, l_3 = \infty,
0, 1$
and then $l_1 \in k^*$ and $l_4 \in \bP^1 \setminus \{1,\infty\}\cong k^*,$
so $\cS_C\cong k^*\times k^*.$
Now let $k = \bF_q$ be a field with $q$ elements. We count
$q + q + (q-1)^2 = q^2 + 1$ elements of $\cC_1$; each has a $k^*$ symmetry group, so the orbifold
cardinality is $(q^2 + 1)/(q-1)$.  If $q = 2$, this equals 5.
\end{example}

\section{Relation to Legendrian Contact Homology}
\label{sec:conj}

The microlocalization theorem of \cite{N,NZ} relates the constructible
sheaves to the Fukaya category of a cotangent.  Local systems on a compact base
manifold are related to compact Lagrangian objects which do not intersect
the contact manifold at infinity.  Let $M = \bR^2$ or $S^1 \times \bR$ and $\Lambda \subset T^{\infty,-}M$
be a Legendrian knot or link (see Section \ref{sec:lkb} for definitions).  The category $\dgsh_{\Lambda}(M, k)_0$ that
we have defined is a category of constructible sheaves orthogonal to local systems.
It is meant to model aspects of the Fukaya category ``at infinity.''  

Legendrian contact homology is another construction which captures aspects
of the Fukaya category at contact infinity, so it is natural to expect some relationship
between the two.  While the precise relationship in a general setting may be more
subtle, for the two geometries under consideration, the relationship
is as simple as possible.  
In this section, we describe this relationship through the formulation
of a theorem (proved in \cite{NRSSZ}), and work through several illustrative
examples.

\subsection{Legendrian Contact Homology and Sheaves}

Legendrian contact homology provides recipe for producing a category out of a Legendrian submanifold of the jet space $J^1(M)$ of a manifold $M.$\footnote{Recall that 
$\bR^3$ and $\bR^3/\bZ \cong S^1 \times \bR^2$,
with their standard contact structures, are contactomorphic to $J^1(\bR)$ and
$J^1(S^1)$, respectively.}
Chekanov \cite{C} and Eliashberg \cite{E} have constructed a differential graded algebra (dga) defined
on the free tensor algebra generated by the Reeb chords.  Ng \cite{Ng} has given a
combinatorial recipe for
computing the C-E dga in terms of the front
projections of a Legendrian knot in $\bR^3$ based on his
technique of ``resolving'' a front projection into a Lagrangian projection, where
the Reeb chords are simply crossings.  The case of knots in $S^1 \times \bR^2$
was considered by Ng and Rutherford in \cite{NR}.

The differential of the C-E dga is computed from
counting holomorphic maps of
disks with punctures along the boundary:
the disks map to the symplectization of the contact manifold, with boundary
arcs mapping to the Lagrangian associated to the Legendrian.
Near a puncture the disk is conformal to a semi-infinite strip, and the
map must limit to a Reeb chord.

An augmentation $\epsilon$ of a commutative algebra $A$ over a ring $R$
is an algebra map $\epsilon:  A \rightarrow R$
to the ground ring, i.e. an $R$-rank one $A$-module. 
When $A$ is a noncommutative dga, an augmentation $\epsilon:  A \rightarrow R$ is a map of dga's,
where $R$ is given the zero differential, i.e. $\epsilon$ obeys $\epsilon\circ d = 0.$
In the symplectic setting where $A$ is the C-E dga of a
Legendrian, the authors of \cite{NRSSZ}
define a category whose objects are augmentations
and whose endomorphisms are essentially Legendrian contact homology.
The construction stems from the important
earlier work of
Bourgeois and Chantraine \cite{BC}. 
It is known (see, e.g., \cite{Ek}) that Lagrangian fillings of Legendrian knots
give rise to augmentations.
Likewise, they give rise to sheaves in our category by the
microlocalization theorems of \cite{N,NZ}.
In neither case is it clear precisely which objects arise from filling surfaces, although certainly some do not.
Nevertheless we have conjectured (in an earlier draft of this paper), based on
computations in examples described below, that our category and the augmentation category agree
over any coefficient field $k$.  This conjecture has since been proven in \cite{NRSSZ}.

\begin{theorem}[\cite{NRSSZ}]
\label{mainconj}
Let $\Lambda$ be a Legendrian knot or link in $\bR^3$ or $S^1 \times \bR^2$
with its standard contact structure.  Let $k$ be a field, and
recall $\cC_1(\Lambda)$ is the full subcategory of rank one objects in
$\dgsh_\Lambda(\bR^2, k)_0$ or $\dgsh_\Lambda(S^1 \times \bR, k)_0$.  
Let $\cA\mathit{ug}(\Lambda)$ be the category of $k$-augmentations of
the C-E dga over $k$ associated to $\Lambda$, as defined in \cite{NRSSZ}.
There is a quasi-equivalence of $A_{\infty}$-categories
\begin{equation}
\label{eq:}
\cC_1(\Lambda) \cong \cA\mathit{ug}(\Lambda)
\end{equation}
\end{theorem}

We will describe computations supporting this theorem and relating to moduli spaces of objects (an issue not addressed in \cite{NRSSZ})
in the next subsection.

\begin{remark}
Theorem \ref{mainconj} is limited.  
The category $\dgsh_\Lambda(\bR^2, k)_0$ is much larger than $\cC_1(\Lambda),$
as it contains all higher-rank objects.  One would like to find
corresponding objects and a corresponding category
in the setting of contact Legendrian homology.
\end{remark}

\subsection{Computations and Examples}

Here we perform explicit calculations in some illustrative examples.

\subsubsection{Spaces of morphisms for binary Maslov objects}

Recall from Section \ref{sec:pixelation} that by taking the knot $\Lambda$ to 
have a `grid' front diagram, we give a full 
quasi-embedding of
$\dgsh_\Lambda(\bR^2, k)$ into  
the triangulated
dg category of complexes of $\bZ^2$-graded modules
over the ring $k[x,y].$  In particular, we use Macaulay2 to 
compute Homs between objects in the latter category. 

Assume $\Lambda$ carries a binary Maslov potential.  Then an object $\cF \in \cC_1(\Lambda)$ 
can be represented by a $\bZ^2$-graded  $k[x,y]$ module, rather than a complex of them.  Moreover,
all such modules have the same graded dimensions.  Indeed, if $M$ is a such a module,
$\dim_k M_{i,j}$ is just the weighted number of times one encounters the knot travelling in the $-x$ (or equivalently
$-y$) direction in the grid diagram from $(i+1/2, j+1/2)$; the weight is $(-1)^\mu$ where $\mu$ is the Maslov potential. 

We recall that the meaning of the endomorphism $x: M_{i,j} \to M_{i, j+1}$ is that 
$M_{i,j} = \cF_{(i+1/2, j+1/2)} = \cF_{(i+1/2, j+1)} $, and $M_{i,j} = \cF_{(i+1/2, j+1+1/2)}$,
and $x: M_{i,j} \to M_{i, j+1}$ is the rightward generization map.  The action of $y$ similarly
captures the upward generization maps. 
Conditions at crossings ensure that $x$ and $y$ commute.  

We write $X, Y \in End_k(M)$ for the matrices giving the action of $x, y$.   One constructs the resolution
of $M$ as a $k[x,y]$-module in the following standard way.  Put $R = k[x,y]$.  Construct the diagram.
$$\xymatrix{M\otimes_k R\ar[r]^{x-X}&M\otimes_k R\\
{}&M\otimes_k R \ar[u]_{y-Y}}$$
Though the resolution is not free, $M\otimes _k R$ is the cokernel of the map
from $(M\otimes_k R)\oplus (M\otimes_k R) \rightarrow M\otimes_k R$
to the upper-right corner, and this data can be used to construct the
free resolution.  More prosaically, it can be fed into the Macaulay2
computer program to construct homomorphisms between $R$-modules.

\begin{example}
Let $\Lambda$ be the rainbow closure of the thrice twisted two-strand braid,
i.e., the trefoil $(2,3)$ torus knot of Figure \ref{fig:planartrefoil}.  When $k = \bF_2$, there are five objects
--- see Example \ref{ex:rainbowtrefoil}.

We have computed all Ext groups between all five objects
of the $(2,3)$ torus knot.  The diagonal entries --- the End groups
of each object --- have $h^0 = 1, h^1 = 2, h^2 = 0,$
while the Ext groups between any two different objects have
$h^0 = 0, h^1 = 1, h^2 = 0.$
The same Ext groups were calculated for the category of augmentations
of the Legendrian trefoil knot
in \cite{BC}.
\end{example}

\begin{example}
\label{ex:34pixels}
Consider the $(3,4)$ torus knot grid diagram pictured in
Figure \ref{fig:34knot}.
\begin{figure}[H]
\includegraphics[scale = .7]{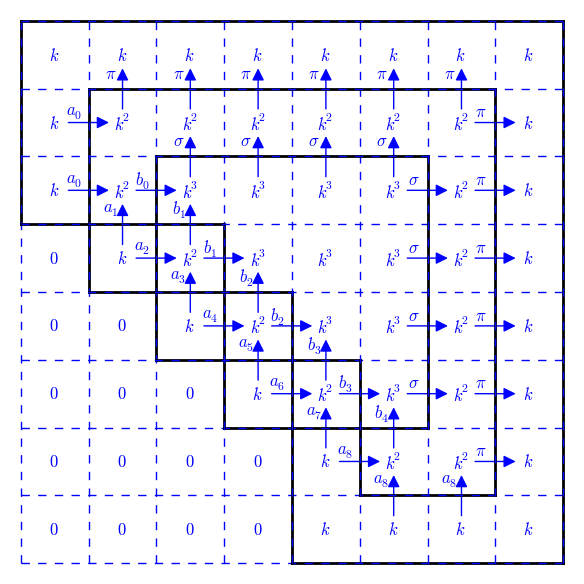}
\caption{Grid diagram for the $(3,4)$ torus knot.}\label{fig:34knot}
\end{figure}
There are 24 grid squares $(i,j)$ with $M_{i,j}$ one-dimensional, 16 within
a two-dimensional region, and 10 three-dimensional squares.  Therefore,
$M$ is a vector space of dimension
$24\cdot 1 + 16\cdot 2 + 10\cdot 3 = 86$.

There are $93$ objects of the category $\cC_1(\Phi)$
over $\bF_2.$
For each object $L$,
we compute (via Macaulay2) the dimensions
of the graded space of endomorphisms $h^i = {\rm dim}\, Ext^i(L,L)$
as $h^0 = 1, h^1 = 6, h^2 = 0.$  All other $Ext$ groups are zero, since $M$
has a free resolution of length three.
We can go from cohomological to homological
notation and define $h_i = h^{1-i}$, then organize the results into a Poincar\'e polynomial
$P = \sum_i h_i t^i = 6 + t.$  This agrees with the linearized contact homology
polynomial of the (3,4)
torus knot
computed in \cite{CN}.
\end{example}

\begin{example} The knot $m8_{21}$ of \cite{CN} is the first example in \cite{CN} of a knot
with two distinct linearized contact homologies.
The front diagram is pictured here.
\begin{figure}[H]
\begin{center}
\includegraphics[scale = .6]{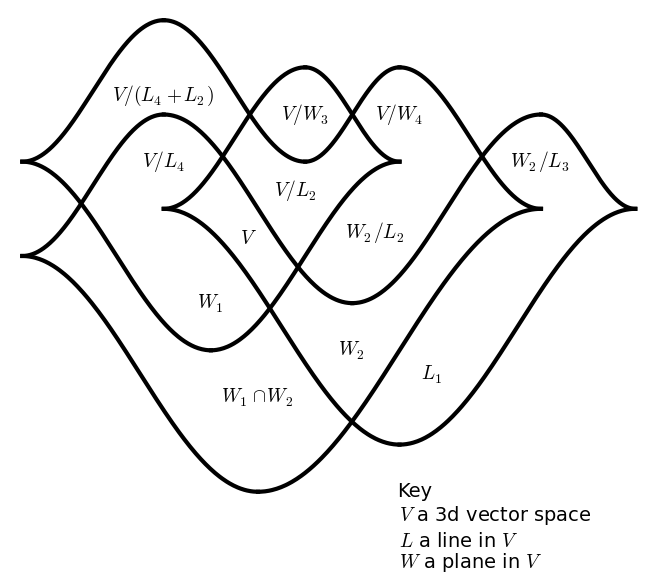}
\caption{Legible diagram for the knot $m8_{21}$ of \cite{CN}.}\label{fig:m821}
\end{center}
\end{figure}
The diagram admits a binary Maslov potential, so all objects of $\cC_1(m8_{21})$ can be represented
by sheaves, rather than complexes thereof.  Moreover, the stratification induced by the front diagram
fails to be regular only on account of the non-compact component, where in any case the sheaves in question
are zero; it follows that all objects are legible. 

In fact, objects can be rigidified by framing the stalk in the region labelled $V$ by a vector space $V = k^3$, 
as then all other stalks admit
either canonical 
injective maps to or surjective maps from a subquotient of $V$.  An object is then determined by the data of a collection of lines
$L$ and planes $W$ in $V$, subject to certain incidence conditions, some of which are captured in the diagram 
by the appearance of two labels in one region, indicating that the corresponding spaces are isomorphic by the evident map.  

As per Theorem \ref{thm:numbers}, 
we compare the augmentation
number of Ng and Sabloff -- which can be computed from the ruling polynomial of this knot, $R(z) = 3 + 2z^2$,
to be $2^1R(z^2=1/2) = 8$ -- to the number of rank one objects of $\cC_1(m8_{21})$.  In fact we find
$10$ objects, 4 of which have a $\bZ/2$ automorphism
group, 6 of which have no nontrivial automorphisms.\footnote{For an
example of an object $F$ with an automorphism, first label the seven nonzero vectors of $V$
by $1,...,7,$ and use this same labeling for the lines that they generate.
Label planes by the three lines they contain, e.g.~$(123).$
Objects are tuples $(L_1,L_2,L_3,L_4;W_1,W_2,W_3,W_4)$.
We define $F$ by setting $(L_1, L_2, L_3, L_4) = (2, 1, 3, 2)$ and notice that
they lie in a single plane $(123).$  Now set $(W_1,W_2,W_3,W_4)
= ((145),(123),(145),(167)).$ The nontrivial automorphism
is defined by its action on the basis of vectors $1, 2, 4.$
It fixes $1$ and $2$ and sends $4$ to $5$ (and $5$ to $4$).  It also swaps $6$ and $7$.
This automorphism clearly leaves the object $F$ fixed.} None of these objects have higher 
homotopies.
The weighted count of objects
is therefore $6 + 4/2 = 8.$

For each object $c$ we compute the self-hom dimensions $h^i := \dim \Ext^i(c,c).$
As in the previous example, we
define $h_i = h^{1-i}$ and create the Poincar\'e polynomial
$P(c) = \sum_i h_i t^i.$  For the six objects with trivial automorphism group,
we compute $P(c) = 2+t.$  For the four objects with $\bZ/2$ automorphism
group, we compute $P(c) = t^{-1} + 4 + 2t.$  These values agree with the two
linearized contact homologies associated to this knot.

In addition, we have computed all $\Ext$ groups between objects via pixelation,
using Macaulay2,
though no analogue of this computation seems to have been performed
in Legendrian contact homology.  Here we record the results.
Order the 10 objects in 3 groups, the first being those for which $L_2, L_3$ and $L_4$
are linearly independent, the second being those for which $W_1$ is not equal
to $W_3$ or $W_4$, and the third being the four objects for which $W_1$ equals
$W_3$ or $W_4$.  In this ordering of the objects, the Ext matrix has group block form
$$\begin{pmatrix}
X&Y&Y\\
Z&X&Z\\
Z&Y&W\\
\end{pmatrix}$$
where:
\begin{itemize}
\item $X$ is a $3\times 3$ matrix with entries $(1,2,0)$ on the diagonal and $(0,1,0)$ off
the diagonal.
\item $Y$ is a $3\times 3$ or $3\times 4$ or $4\times 3$ matrix with every entry equal to $(1,2,0).$
\item $Z$ is a $3\times 3$ or $3\times 4$ or $4\times 3$ matrix with every entry equal to $(0,2,1).$
\item$W$ is a $4\times 4$ matrix with diagonal entries $(2,4,1)$ and off-diagonal entries $(1,3,1).$
\end{itemize}
\end{example}

\subsubsection{The Chekanov Pair}
\label{sec:chekpair}

The Chekanov pair of 
Legendrian knots
are topologically isomorphic but are
inequivalent under Legendrian isotopies.
They have the same Thurston-Bennequin and rotation numbers.
Front diagrams $\Phi_1$ and $\Phi_2$ for this pair are in the figure below.
\begin{figure}[H]
\begin{center}
\includegraphics[scale = .5]{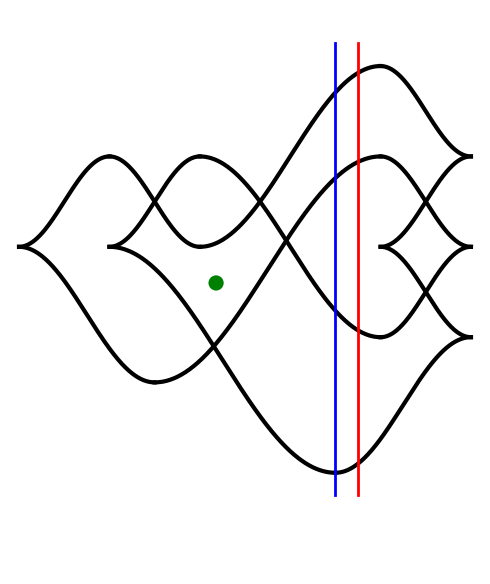}
\includegraphics[scale = .5]{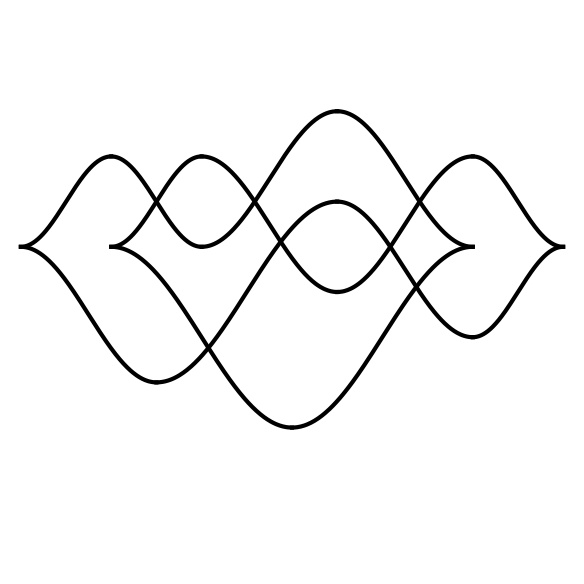}
\caption{Front diagrams $\Phi_1$ (left) and $\Phi_2$ (right) for the Chekanov pair.}\label{fig:cepair}
\end{center}
\end{figure}
Famously, the C-E dga differentiates this pair. We show now that 
the category $\cC_1$ does so, too:  $\cC_1(\Phi_1)\not\cong\cC_1(\Phi_2).$

Note first that $\Phi_2$ is binary Maslov, so every object in $\cC_1(\Phi_2)$
is equivalent to an ordinary sheaf, concentrated in degree zero.
As a result, 
$\Ext^m(M,M)=0$ for $m<0.$
To show $\cC_1(\Phi)\not\cong \cC_1(\Phi_2)$ we 
produce an $\cF \in \cC_1(\Phi_1)$ with  $\Ext^{-1}( \cF,\cF) \neq 0.$

\begin{proposition}
There exists an object $\cF \in \cC_1(\Phi_1)$ with $\Ext^{-1}(\cF,\cF) \neq 0.$
\end{proposition}
\begin{proof}

On the diagram $\Phi_1$ of Figure \ref{fig:cepair},
there is a unique graded normal ruling.  Its switches are 
on the top and bottom of the diamond at the right of the diagram; we denote
these points $t$ and $b$ respectively.  
The resulting
ruling has
three eyes:  the diamond $D$, an eye $T$ touching the top of the diamond
(and containing the left-most left cusp), and an eye $B$ touching the bottom of the diamond.

We write $\cE_D, \cE_T, \cE_B$ for the corresponding eye sheaves. 
Since $t = T \cap D$ and $b = B \cap D$, it is easy to calculate the following hom sheaves: 

\begin{eqnarray*} 
\cH om^\bullet(\cE_T, \cE_D) & = & k_t[-2] \\
\cH om^\bullet(\cE_D,  \cE_B) & = & k_b[-2] 
\end{eqnarray*}

It follows that $\Ext^1(\cE_D, \cE_T[1]) = k$.  As with Lemma \ref{lem:ssswitch}, if we take a nontrivial extension class,
the corresponding sheaf $\cG$ satisfies the appropriate singular support condition at $t$.  In a neighborhood
of $b$, we have
$\cH om^\bullet(\cE_B, \cG) = \cH om^\bullet(\cE_B,  \cE_D) = k_b[-2]$, again taking the image
in $\Ext^1(\cG, \cE_B[-1])$ of a nontrivial extension class in $\Ext^1(\cE_D,  \cE_B)$ gives a sheaf $\cF$ which satisfies
the singular support condition both at $t$ and $b$.  Everywhere else the singular support conditions were satisfied
by the constituents $\cE_T, \cE_D, \cE_B$, so we conclude 
that $\cF \in  \cC_1(\Phi_1)$. 

Let $L$ be the open set to the left of the red line, and $R$ the open set to the right of the blue line.  
We will construct an element of $\Ext^{-1}(\cF, \cF)$ by constructing a nonzero element of
$\Ext^{-1}(\cF|_L, \cF|_L)$ which restricts to zero in $\Ext^{-1}(\cF|_{L \cap R}, \cF|_{L \cap R})$
and hence can be glued to the zero element of $\Ext^{-1}(\cF|_R, \cF|_R)$. 

But $\cF|_L = \cE_T[1] \oplus \cE_B[-1]$, and so in particular we have 
\begin{eqnarray*}
\cH om^\bullet(\cF, \cF)|_{L \cap R} & = &  \cH om^\bullet(\cE_T[1] \oplus \cE_B[-1], \cE_T[1] \oplus \cE_B[-1])|_{L \cap R} \\
& = & \cH om^\bullet(\cE_T, \cE_T)|_{L \cap R} \oplus  \cH om^\bullet(\cE_B, \cE_B)|_{L \cap R} 
\end{eqnarray*}
because the eyes are disjoint in this region. 

It will therefore suffice to produce a nonzero element of the summand $\Ext^{-1}(\cE_B[-1]|_{L}, \cE_T[1]|_{L}) = 
\Ext^1(\cE_B|_{L}, \cE_T|_{L})$ of $\Ext^{-1}(\cF|_{L}, \cF|_{L})$, since this will certainly restrict to zero
in $\Hom^\bullet(\cF, \cF)|_{L \cap R}$.  The sheaf $\cH om(\cE_B|_{L}, \cE_T|_{L})$ is supported on the heart-shaped
intersection of these two eyes --- the region indicated by the green dot in
Figure \ref{fig:cepair} (left) ---
and is constant on the interior with standard boundary conditions along two intervals
and co-standard boundary conditions along another two.  Thus 
$\Hom(\cE_B|_{L}, \cE_T|_{L})$ is the cohomology of a disc relative two disjoint intervals on the boundary,
which has a nonzero class in $H^1$. 
\end{proof}

\end{document}